\pdfoutput=1
\documentclass{amsart}
\usepackage{amsmath,amsthm,amssymb,amsfonts,tikz,bm}
\usepackage[margin=1.2in]{geometry}
\usepackage{hyperref}
  \hypersetup{
    colorlinks,
    linkcolor={red!50!black},
    citecolor={blue!50!black},
    urlcolor={blue!80!black}
}
\usepackage[hypcap]{caption}
\usepackage{palatino}
\usepackage{enumitem}
\usepackage{graphicx,color}
\usepackage{wrapfig}
\usepackage{stmaryrd}
\usepackage{tikz-cd}
\usepackage{multirow}
\usepackage{listings}
\usepackage{natbib}
\usepackage{caption}
\allowdisplaybreaks

\newtheorem*{thm*}{Theorem}
\newtheorem{lem}{Lemma}
\newtheorem{prop}{Proposition}
\theoremstyle{definition}
\newtheorem{example}{Example}
\newtheorem{remark}{Remark}
\newtheorem*{warning*}{Warning}
\newtheorem*{definition*}{Definition}

\newcommand{\R}{\mathbb{R}}
\newcommand{\Q}{\mathbb{Q}}
\newcommand{\Z}{\mathbb{Z}}

\newcommand{\A}{\mathcal{A}}

\newcommand{\G}{\mathcal{G}}

\newcommand{\W}{\mathcal{W}}

\newcommand{\tr}{\mathrm{tr}}
\newcommand{\odd}{\vec{odd}}
\newcommand{\even}{\vec{even}}
\renewcommand{\vec}[1]{\textbf{#1}}

\title{On Meta-monoids and the Fox-Milnor Condition}
\author{Huan Vo}
\address{
  Department of Mathematics\\
  University of Toronto\\
  Toronto Ontario M5S 2E4\\
  Canada
}
\email{vohuan@math.toronto.edu}
\urladdr{http://www.math.toronto.edu/vohuan}

\begin{document}

\begin{abstract}
  In this paper we introduce an algebraic structure known as meta-monoids which is suited for the study of knot theory. We define a particular meta-monoid called $\Gamma$-calculus that gives an Alexander invariant of tangles. Using the language of $\Gamma$-calculus, we rederive certain important properties of the Alexander polynomial, most notably the Fox-Milnor condition on the Alexander polynomials of ribbon knots \cite{Lic97,FM66}. We argue that our proof has some potential for generalization which may help tackle the slice-ribbon conjecture. In a sense this paper is an extension of \cite{BNS13}.  
\end{abstract}

\maketitle

\tableofcontents

\section{Introduction}

The Alexander polynomial is one of the most important invariants in knot theory. Originally discovered by Alexander in 1928 \cite{Al28}, many things are known about the polynomial. For instance, it has a topological description \cite{Lic97}, an interpretation as a quantum invariant \cite{Oht02, KS91}, and recently has been categorified via Heegaard-Floer homology \cite{OS04}. To compute the Alexander polynomial of a knot with many crossings, one strategy would be to break the knot into smaller pieces called tangles, find an appropriate extension of the Alexander polynomial to tangles, compute the said extension for each constituent tangle, and then ``glue'' the results together. One can obtain an Alexander invariant of tangles in several ways, which are roughly based on two perpestives: from the quantum invariant point of view \cite{Oht02, Sar15} or from the topological/combinatorial point of view \cite{CT05, Arc10, Pol10, BNS13, BCF15, DF16}.

One important aspect of knot theory is its implementation on a computer. For that purpose, two definitions of the Alexander polynomial are particularly useful: in terms of $R$-matrix, i.e. quantum invariant \cite{Oht02}, or in terms of Fox derivatives \cite{Arc10}. These two formulations come from two ways of viewing a tangle. One approach views a tangle as a morphism in a category and the other views a tangle in terms of planar/circuit algebra \cite{BND14}. We introduce yet another way to view a tangle, as an element of a meta-monoid (Section \ref{sec:metamonoids}) \cite{BNS13, Hal16}. Namely, we can just decompose a tangle into a disjoint union of crossings and then stitch the strands together to recover the tangle, with the condition that we cannot stitch the same strand to itself, in order not to produce closed components. Hence one can think of meta-monoids as restrictions of circuit algebras where we do not allow closed components. An advantage of seeing things in this way as opposed to the usual categorical approach is that one no longer needs the data of the bottom/top of a tangle. Moreover, virtual crossings are automatically included in the meta-monoid structure, so we can talk about the bigger class of virtual tangles (or more precisely w-tangle).  

On the algebraic side the meta-monoid that gives us a tangle invariant is called Gassner calculus or $\Gamma$-calculus. Roughly speaking, $\Gamma$-calculus assigns to a tangle with $n$ open components a rational function and an $n\times n$ matrix. In the case where a tangle has only one component, we recover the Alexander polynomial. One can obtain a topological interpretation of $\Gamma$-calculus along the lines of the arguments in \cite{CT05} and \cite{DF16} but we will not pursue it here. On a computer, $\Gamma$-calculus is quite simple to implement (see the \hyperref[sec:appendix]{Appendix}) and it also runs faster than current algorithms that compute the Alexander polynomial. Furthermore, its complexity is a polynomial in $n$, whereas in the quantum approach, the complexity tends to be exponential (since one needs to look at tensor products of representations). One can think of $\Gamma$-calculus as a generalization of the Gassner-Burau representation \cite{KT08,BNT14} to tangles (compare also with \cite{KLW01}). 
Moreover, we break the determinant formula into a step-by-step gluing instruction with each step involving some simple algebraic manipulations. This approach may play a role if one wants to categorify the invariant, which can lead to a simpler way to compute the formidable Heegaard-Floer homology.

The bulk of the paper is devoted to rederiving the Fox-Milnor condition on the Alexander polynomials of ribbon knots \cite{Lic97,FM66}, which simply says that the Alexander polynomial of a slice knot factors as a product of two Laurent polynomials $f(t)f(t^{-1})$, in the framework of $\Gamma$-calculus. Our ultimate goal is to say something about the   ribbon-slice conjecture \cite{GSA10}, which asks whether every slice knot is also ribbon. Let us give a brief overview of our approach (see \cite{BNT17} for more details). First of all, given a tangle $T_{2n}$ with $2n$ components, there are two closure operations, denoted by $\tau$ and $\kappa$ (Section \ref{sec:foxmilnor}), which gives an $n$-component tangle $T_n$ and a one-component tangle $T_1$, i.e. a long knot, respectively 
   \[T_n\xleftarrow{\tau} T_{2n}\xrightarrow{\kappa} T_1.\]  
Now we have the following characterization of ribbon knots (Proposition \ref{ribbon}), namely a knot $K$ is ribbon if and only if there exists a 2n-component tangle $T_{2n}$ such that $\kappa(T_{2n})=K$ and $\tau(T_{2n})$ is the trivial tangle. More succinctly, if we denote the set of all $m$-component tangles by $\mathcal{T}_{m}$, then 
 \[
   \{\text{ribbon knots}\}=\bigcup_{n=1}^{\infty}\Big\{\kappa(T_{2n}):T_{2n}\in \mathcal{T}_{2n}\ \text{and}\ \tau(T_{2n})=U_n\in \mathcal{T}_n\Big\},
 \]  
where $U_n$ denote the trivial $n$-component tangle. Therefore if we have an invariant $Z:\mathcal{T}_k\to \mathcal{A}_k$ of tangles, where $\mathcal{A}_k$ is some algebraic space which is well-understood (think of matrices of polynomials), together with the corresponding closure operations $\tau_{\mathcal{A}}$ and $\kappa_{\A}$ which intertwine with $\tau$ and $\kappa$:
  \[Z(\kappa(T_{2n}))=\kappa_{\A}(Z(T_{2n})),\quad Z(\tau(T_{2n}))=\tau_{\A}(Z(T_{2n})),\]   
then we have an ``algebraic criterion'' to determine if a given knot $K$ is ribbon. Specifically, if a knot $K$ is ribbon then there exist some $n$ and an element $\zeta\in \A_{2n}$ such that $Z(K)=\kappa_{\A}(\zeta)$ and $\tau_{\A}(\zeta)=\mathrm{Id}_n\in \A_n$, or more simply
 \begin{equation}\label{ribboncondition}
 Z(K)\in \bigcup_{n=1}^{\infty} \kappa_{\A}(\tau_{\A}^{-1}(\mathrm{Id}_n)).
 \end{equation} 
 We denote the set on the right hand side by $\mathcal{R}_{\mathcal{A}}$. Of course to have any practical values, we need to make sure that $\mathcal{R}_{\A}$ is strictly smaller than $\A_1$. Then a knot $K$ is not ribbon if $Z(K)\not\in \mathcal{R}_{\A}$. 
 
In \cite{GSA10} the authors propose several potential counter-examples to the ribbon-slice conjecture. These are knots with a high number of crossings. Our long term goal is to construct a class of invariants of tangles which are computable in polynomial time and behave well under various operations in order to test these counter-examples in the framework proposed above (see partial progress in \cite{BN16}). The simplest example of such invariants is $\Gamma$-calculus, and condition \eqref{ribboncondition} yields the familiar Fox-Milnor condition, as to be expected since $\Gamma$-calculus is an extension of the Alexander polynomial to tangles. Although the original proof of the Fox-Milnor condition \cite{FM66} is quite short and elegant, we believe our proof offers several advantages as summarized below. 
  \begin{itemize}
    \item The original proof uses homology and so cannot distinguish between slice and ribbon properties, whereas our proof follows quite mechanically from the characterization of ribbon knots in terms of tangles. Furthermore, the bulk of our proof uses just elementary linear algebra, which is more accessible to a novice. 
    \item In our proof the function $f$ appears naturally as the invariant of a tangle obtained from a tangle presentation (Propostion \ref{ribbon}) of a ribbon knot. 
    \item We believe our proof could be generalized to a stronger class of invariants which we are currently developing, and we hope these better invariants will give a much stronger condition for ribbon knots.  
  \end{itemize}
So as it stands this paper serves as a warm-up step in a long project and we hope that it will whet the readers' appetite to join our quest.      

The paper is organized as follows. In Section \ref{sec:metamonoids} we give the main definitions and properties of meta-monoids as well as some main examples. In Section \ref{sec:gassnercalculus} we describe our main meta-monoid: $\Gamma$-calculus and derive various formulae therein. Section \ref{sec:foxmilnor} is the main main part of this paper where we introduce ribbon knots and prove the Fox-Milnor condition. In Section \ref{sec:link} we show how one can extend the scalar part of $\Gamma$-calculus to links and derive the classic Alexander-Conway skein relation. Finally in the \hyperref[sec:appendix]{Appendix} we record the Mathematica implementation of $\Gamma$-calculus.    

\subsection{Acknowledgement} The author is grateful to Prof Dror Bar-Natan for all his kindness and support during the writing of this paper. The author would also like to thank Travis Ens and Andrey Boris Khesin for many stimulating discussions.

\section{Meta-Monoids}\label{sec:metamonoids}
\subsection{Definitions}\label{subsec:metamonoids} 
A \emph{meta-monoid} $\mathcal{G}$ (see \cite{BNS13,BN13,Hal16}) is a collection of sets $(\G^X)$ indexed by finite sets $X$ (each set $X$\footnote{We usually take $X$ to be a subset of a fixed set, say the natural numbers} can be thought of as a set of \emph{labels}) together with the following operations:
   \begin{itemize} 	  
  \item[] ``stitching'' $m^{x,y}_z: \G^{\{x,y\}\cup X}\to \G^{\{z\}\cup X}$, where $\{x,y,z\}\cap X=\emptyset$ and $x\neq y$,
   \item[] ``identity'' $e_x: \G^X\to \G^{\{x\}\cup X}$, where $x\notin X$,
   \item[] ``deletion'' $\eta_x: \G^{X\cup \{x\}}\to \G^X$, where $x\notin X$,
   \item[] ``disjoint union'' $\sqcup: \G^X\times \G^Y\to \G^{X\cup Y}$, where $X\cap Y=\emptyset$,
   \item[] ``renaming'' $\sigma^x_z: \G^{X\cup\{x\}}\to \G^{X\cup\{z\}}$, where $\{x,z\}\cap X=\emptyset$.
  \end{itemize} 
These operations satisfy the following axioms\footnote{In this paper we use the notation $\sslash$ to denote function compositions, namely $f\sslash g=g\circ f$.}:
 \begin{itemize}
   \item[] \emph{Monoid axioms:}
      \begin{itemize}
        \item[] $m^{x,y}_u\sslash m^{u,z}_v=m^{y,z}_u\sslash m^{x,u}_v\quad$ (\emph{meta-associativity}),
        \item[] $e_a\sslash m^{a,b}_c=\sigma^b_c\quad$   (\emph{left identity}),
        \item[] $e_b\sslash m^{a,b}_c=\sigma^a_c\quad$ (\emph{right identity}),
      \end{itemize}
   \item[] \emph{Miscellaneous axioms:}
      \begin{itemize}
        \item[] $e_a\sslash\sigma^a_b=e_b,\quad \sigma^a_b\sslash\sigma^b_c=\sigma^a_c,\quad \sigma^a_b\sslash \sigma^b_a=Id$,
        \item[] $\sigma^a_b\sslash \eta_b=\eta_a,\quad e_a\sslash \eta_a=Id,\quad m^{a,b}_c\sslash \eta_c=\eta_b\sslash \eta_a$,
        \item[] $m^{a,b}_c\sslash \sigma^c_d=m^{a,b}_d,\quad \sigma^a_b\sslash m^{b,c}_d=m^{a,c}_d$.
      \end{itemize}
 \end{itemize}   
We also require that operations with distinct labels commute, for instance $\eta_a\sslash \eta_b=\eta_b\sslash \eta_a$ or $m^{a,b}_c\sslash m^{d,e}_f=m^{d,e}_f\sslash m^{a,b}_c$ etc. Moreover, the disjoint union operation $\sqcup$ commutes with all other operations, for example $\sqcup\sslash m^{a,b}_c\sslash m^{d,e}_f=(m^{a,b}_c,m^{d,e}_f)\sslash \sqcup$. Given two meta-monoids $\G$ and $\mathcal{H}$, a \emph{meta-monoid homomorphism} is a collection of maps $(f^X: \G^X\to \mathcal{H}^X)$ that commute with the operations.  

In practice, usually the only non-trivial relation we have to check is meta-associativity. While the definition is quite lengthy, a couple of examples will make it clear how to think about meta-monoids and where the name comes from.

\begin{example}[Monoids]
 Given a monoid $G$ with identity $e$ (or an algebra), one obtains a meta-monoid as follows. Let 
  \[G^X=\{\text{functions $f:X\to G$}\}.\]
We write an element of $G^X$ as $(x\to g_x,\dots)$, where $x\in X$ and $g_x\in G$. In the following operations, $\dots$ denotes the remaining entries, which stay unchanged:
  \[(x\to g_x,y\to g_y,\dots)\sslash m^{x,y}_z=(z\to g_xg_y,\dots),\]
  \[(y\to g_y,\dots)\sslash e_x=(x\to e,y\to g_y,\dots),\]
  \[(x\to g_x,y\to g_y,\dots)\sslash \eta_x=(y\to g_y,\dots),\]
  \[(x\to g_x,\dots)\sqcup (y\to g_y,\dots)=(x\to g_x,\dots,y\to g_y,\dots),\]
  \[(x\to g_x,y\to g_y,\dots)\sslash \sigma^x_z=(z\to g_x,y\to g_y,\dots).\] 
Let us check meta-associativity. Suppose $\Omega\in G^{X\cup \{x,y,z\}}$ and we only write the relevant entries, the others are left unchanged:
 \[\Omega=(x\to g_x,y\to g_y,z\to g_z).\]
Then 
 \[\Omega\sslash m^{x,y}_z\sslash m^{u,z}_v=(v\to (g_xg_y)g_z),\]
 and 
  \[\Omega\sslash m^{y,z}_u\sslash m^{x,u}_v=(v\to g_x(g_y g_z)).\]
Thus we see that meta-associativity follows from the associativity of multiplication $(g_x g_y)g_z=g_x(g_y g_z)$. Similarly the left identity and right identity are consequences of $eg=ge=g$ for all $g\in G$. The other axioms are straightforward to verify. In general, $m^{x,y}_z\neq m^{y,x}_z$, unless $G$ is commutative. This meta-monoid also satisfies the following property: 
   \begin{equation}\label{quadratic}
   \Omega=(\Omega\sslash \eta_y)\sqcup (\Omega\sslash \eta_x),\quad\Omega\in G^{\{x,y\}}.
   \end{equation}
 Indeed if $\Omega=(x\to g_x,y\to g_y)$, then $\Omega\sslash \eta_y=(x\to g_x)$ and $\Omega\sslash \eta_x=(y\to g_y)$ and so the right hand side is exactly $\Omega$. Most examples will not satisfy this property, and so we see that not every meta-monoid comes from a monoid.\hfill$\clubsuit$       
\end{example}

\begin{example}[Groups (see also \cite{BNS13})] \label{metamonoidgroups}
  Consider the meta-monoid $\G$ given as follows. Let $\G^X$ consist of triples of the form $(F,m,l)$, where $F$ is a finitely presented group and $m:X\to F$ and $l:X\to F$ ($m$ is called a \emph{meridian map} and $l$ is called a \emph{longitude map}, which we also allow up to conjugation). Now the operations are 
   \[(F,m,l)\sslash m^{x,y}_z=(F/\left\langle m_y=l_x^{-1}m_x l_x\right\rangle,m\setminus\{m_x,m_y\}\cup\{z\to m_x\},l\setminus\{l_x,l_y\}\cup\{z\to l_xl_y\}),\]  
   \[(F,m,l)\sslash e_x=(F*\left\langle x\right\rangle,m\cup \{x\to x\},l\cup \{x\to 1\}),\]
   \[(F,m,l)\sslash \eta_x=(F/\left\langle m_x=1\right\rangle, m\setminus\{m_x\}, l\setminus\{l_x\}),\]
   \[(F,m,l)\sqcup (F',m',l')=(F*F',m\cup m',l\cup l'),\]
   \[(F,m,l)\sslash \sigma^x_z=(F,m\setminus\{m_x\}\cup\{z\to m_x\},l\setminus\{l_x\}\cup\{z\to l_x\}).\]
We leave the verification of the axioms to the reader. Notice that property \eqref{quadratic} does not hold here, for instance if $F=\mathbb{Z}\oplus \mathbb{Z}$, then $(F\sslash \eta_x)\sqcup (F\sslash \eta_y)$ is the free product $\mathbb{Z}*\mathbb{Z}$.\hfill $\clubsuit$ 
\end{example}

\subsection{The meta-monoid of w-tangles} A \emph{w-tangle diagram} is a finite collection of oriented arcs (or \emph{components}) smoothly drawn on a plane, with finitely many intersections, divided into \emph{virtual crossings}$\ $ \includegraphics[scale=0.13]{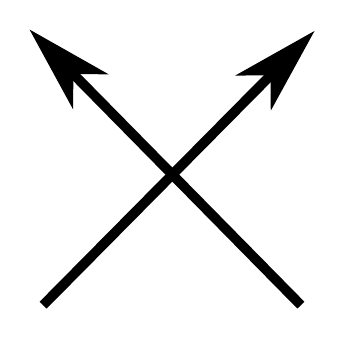}, \emph{positive crossings} $\ $ \includegraphics[scale=0.13]{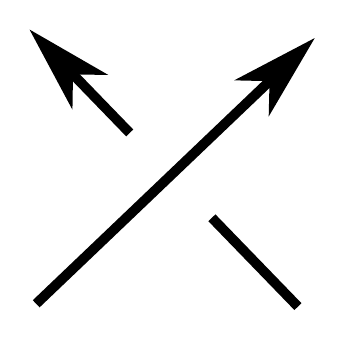}, and \emph{negative crossings} $\ $ \includegraphics[scale=0.13]{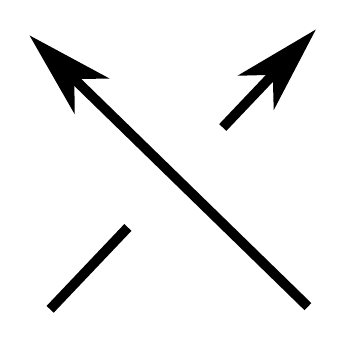}; and regarded up to planar isotopy. We also require distinct components to be labeled by distinct labels from some set of labels $X$. An example of a w-tangle diagram is
  \begin{center}
  \includegraphics[scale=0.35]{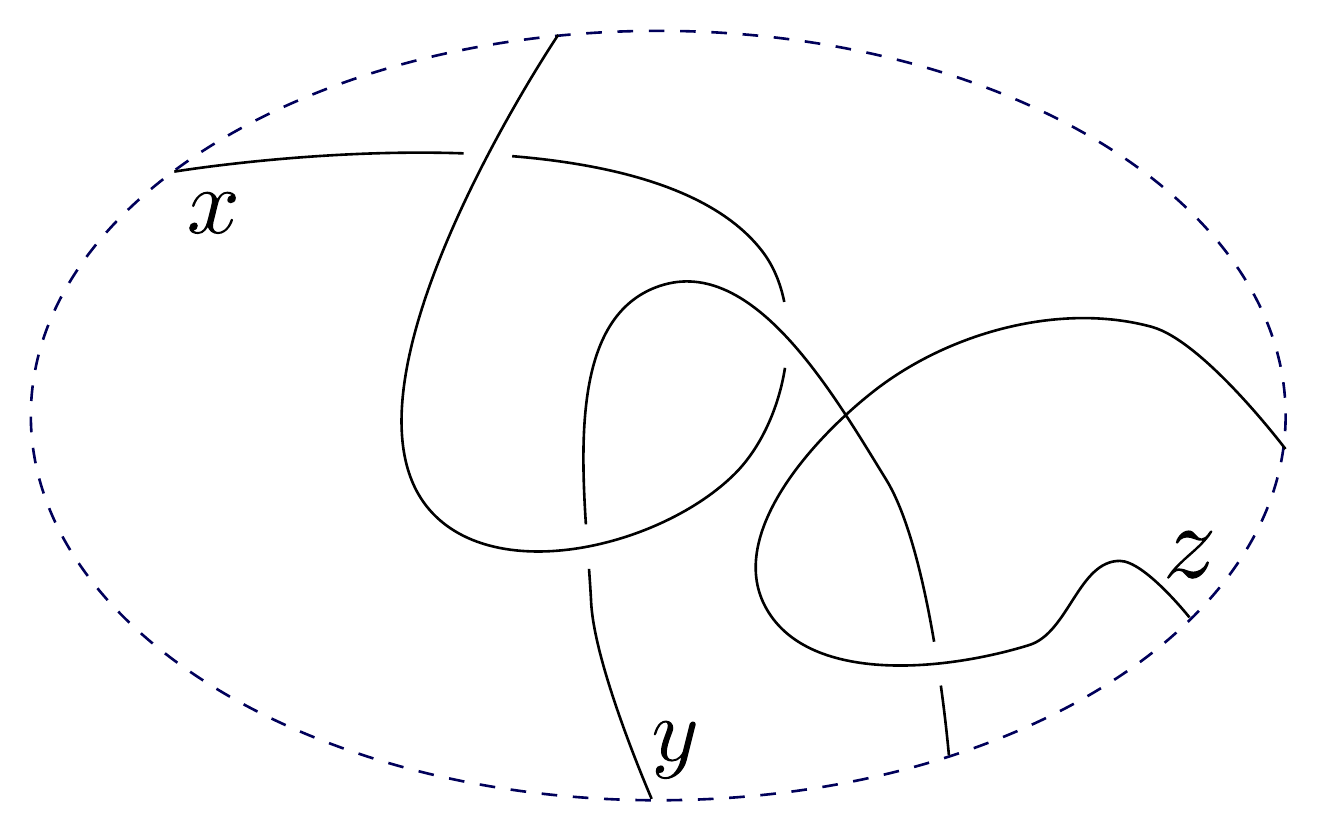}
 \end{center}
In the figure the boundary circle is drawn in dashed line for ease of visualization, but is not part of the data. In particular, we do not care about the position of the endpoints. So for instance 
  \begin{center}
    \includegraphics[scale=0.4]{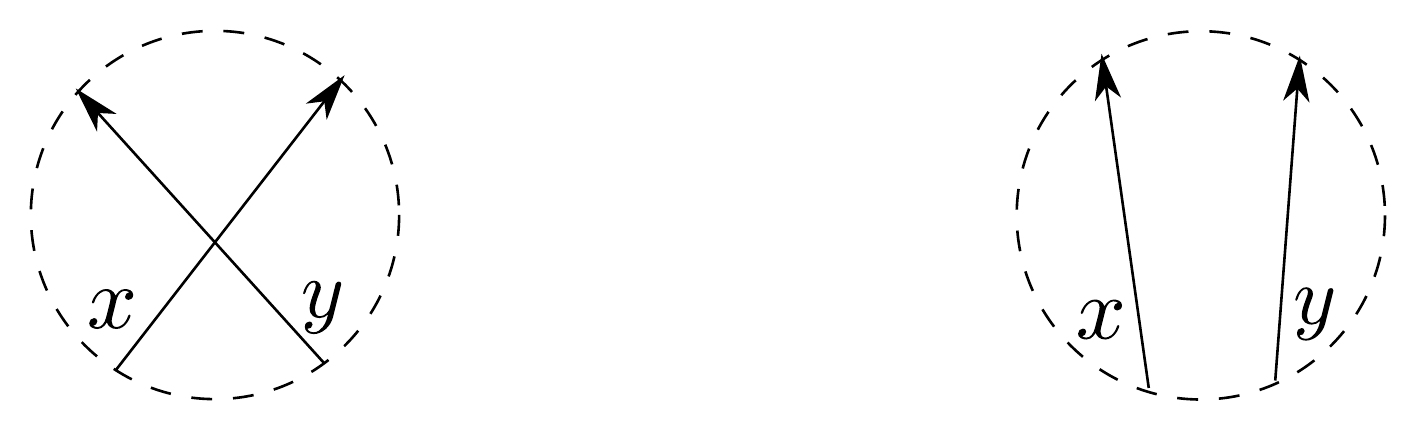}
  \end{center} 
represent the same diagram. 
 
A \emph{w-tangle} is an equivalence class of w-tangle diagrams, modulo the equivalence generated by the Reidemeister 2 and 3 moves ($R2$, and $R3$), the virtual Reidemeister 1 through 3 moves ($VR1$, $VR2$, $VR3$), the mixed relations ($M$), and the overcrossings commute relations ($OC$). 
     \begin{center}
       \includegraphics[scale=0.42]{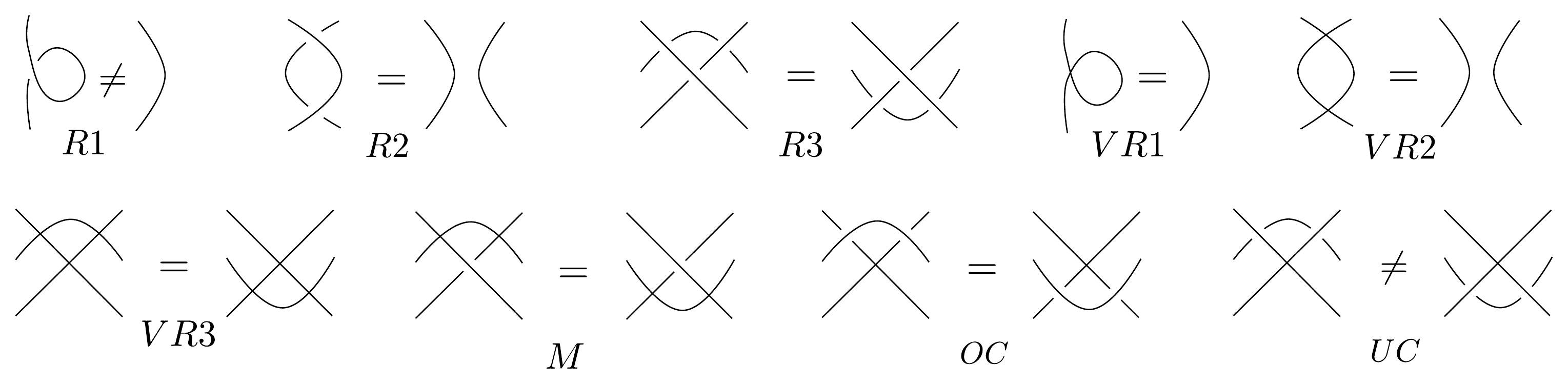}
       \captionof{figure}{The relations defining w-tangles.}  \label{fig:Reidemeistermoves} 
 \end{center} 
Note that we do not mod out by the Reidemeister 1 move ($R1$) nor by the undercrossings commute relations ($UC$). Also we do not allow closed components. For a topological interpretation of w-tangles, see \cite{BND16,BN13}.

Let us describe the $OC$ relations in a bit more details. 
  \begin{center}
    \includegraphics[scale=0.5]{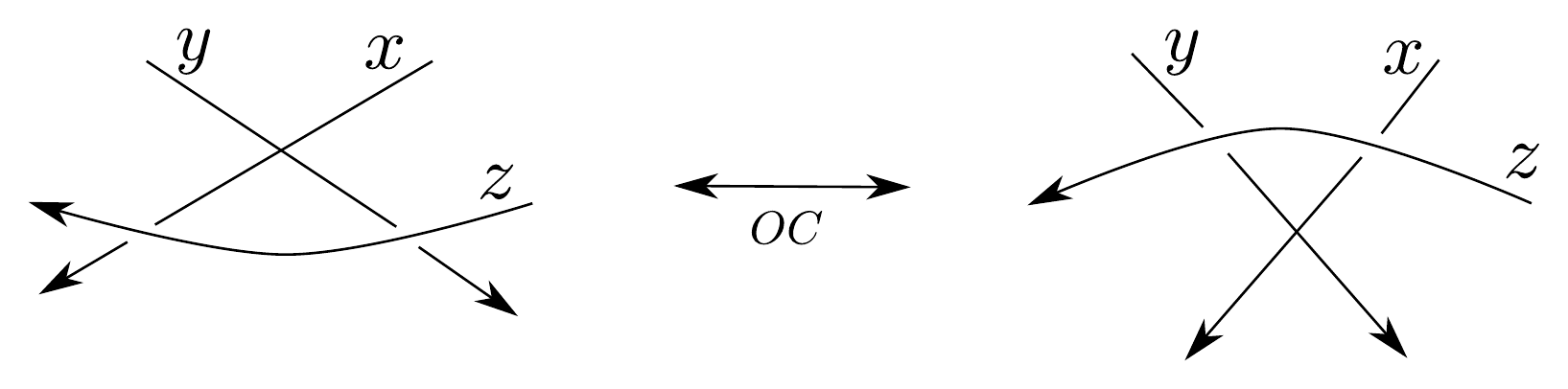}
  \end{center} 
Notice that in our description of w-tangles, virtual crossings play no essential role. However the $OC$ relations are not trivial, since it says that from the perspective of strand $z$, going over strand $y$ and then strand $x$ is the same as going over strand $x$ and then strand $y$.  

When a w-tangle has only one component, we obtain a theory of long w-knots. It is well-known that \emph{ordinary long knots} (which are the same as closed knots, see for example \cite{openknotoverflow}) inject into long w-knots \cite{BND16}. In other words, two knots are isotopic (as ordinary knots) if and only if they are isotopic as w-knots. Thus an invariant of long w-knots yields a knot invariant. However note that long w-knots are not equivalent to closed w-knots (see Example \ref{longvsclosed}). 

Now we would like to introduce our main object of study: the \emph{meta-monoid $\W$ of w-tangles}. Specifically, for a finite set $X$, let $\W^X$ be the collection of w-tangles on whose components are labeled by the elements of $X$. It is important to note that we do not allow closed components (embeddings of $S^1$) in the definition of $\W^X$. Now the meta-monoid operations have a very explicit geometric interpretation in this context given as follows.
 \begin{itemize}
    \item Stitching $m^{x,y}_z$ means connecting the head of strand $x$ to the tail of strand $y$ and calling the resulting strand $z$, note that if strand $x$ and strand $y$ are far away, we can always bring them together via virtual crossings. Note that from  now on we use the convention that a dashed line means it can be knotted freely.
   \begin{center} 
    \includegraphics[scale=0.4]{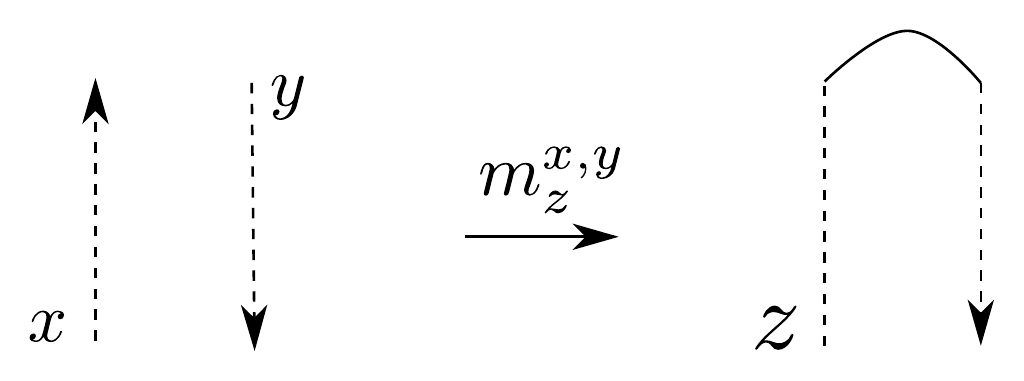}
   \end{center}
   \item Identity $e_x$ means adding a trivial strand labeled $x$ which does not cross any other strand.
      \begin{center}
      \includegraphics[scale=0.4]{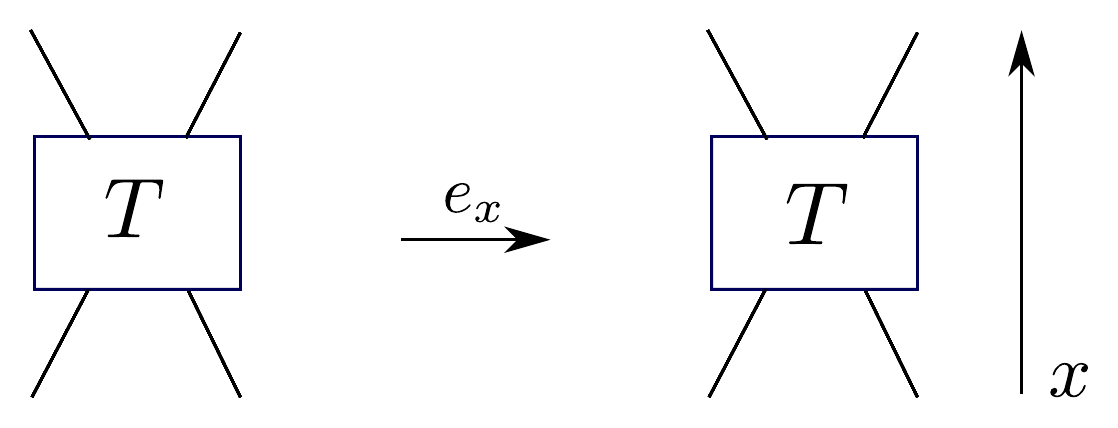}
     \end{center}
    \item Deletion $\eta_x$ means deleting strand $x$ from the w-tangle.
    \item Disjoint union $\sqcup$ means putting the two w-tangles side by side, again by the nature of virtual crossings it does not matter how we put these two tangles together. To simplify notation, we abbreviate $T_1\sqcup T_2$ as just $T_1 T_2$.
      \begin{center}
    \includegraphics[width=0.4\textwidth]{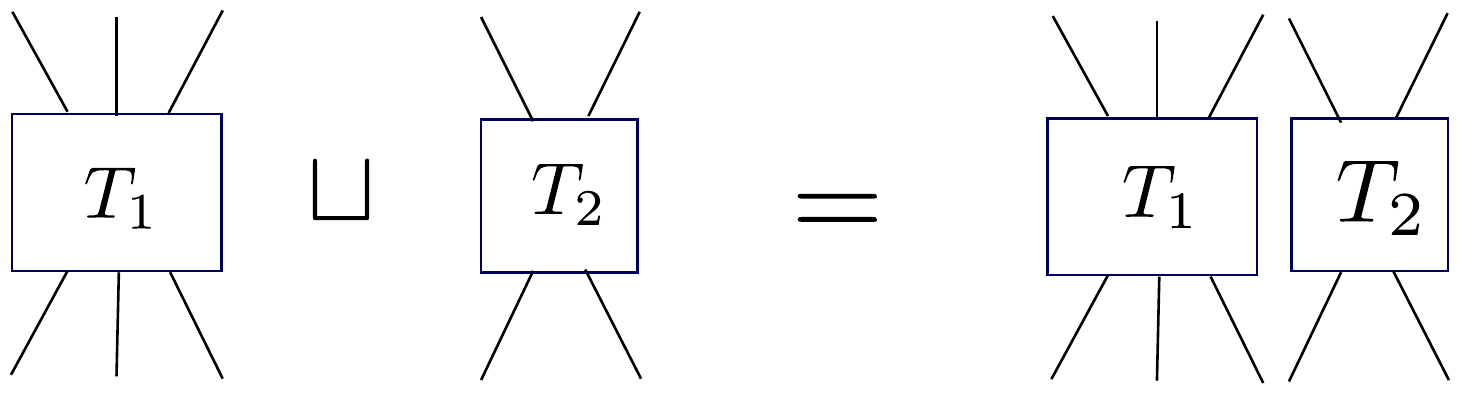}
    \end{center}
    \item Renaming $\sigma^x_z$ means relabeling strand $x$ to strand $z$.     
 \end{itemize}
Then all the monoid axioms have explicit geometric interpretation given as follows. 
  \begin{itemize}
  \item The meta-associativity relation: 
  \begin{center}
    \includegraphics[scale=0.5]{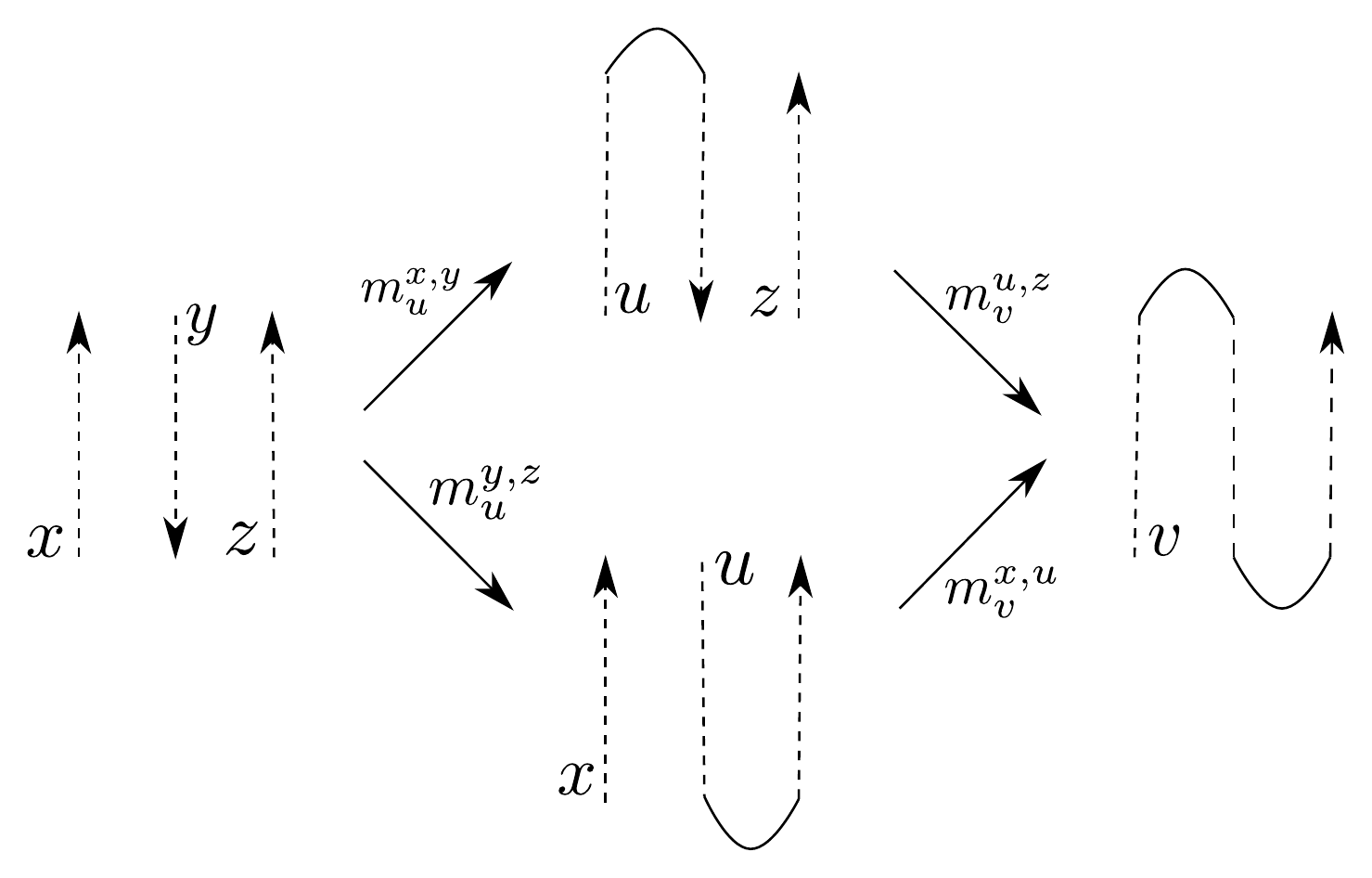}
   \end{center}
  \item The left identity relation: 
\begin{center}
   \includegraphics[scale=0.4]{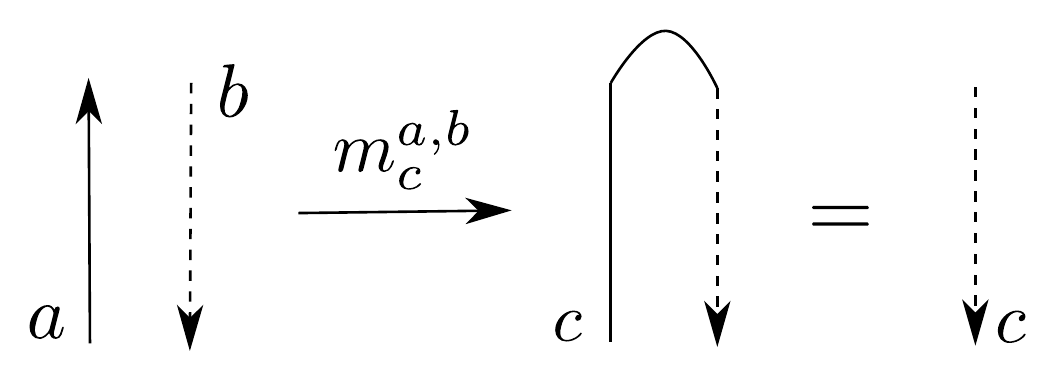}
  \end{center}
\item The right identity relation:  
  \begin{center}
  \includegraphics[scale=0.4]{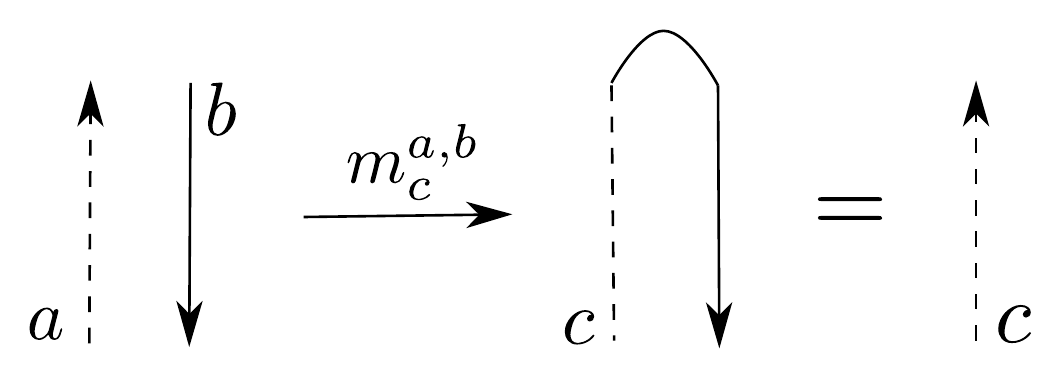}
\end{center}
\end{itemize} 

In the framework of meta-monoids, a w-tangle can be described in terms of generators and relations. Specifically, given a w-tangle, we can first decompose it into a disjoint union of positive crossings $R_{i,j}^+$ and negative crossings $R_{i,j}^-$:
  \begin{center}
      \includegraphics[scale=0.4]{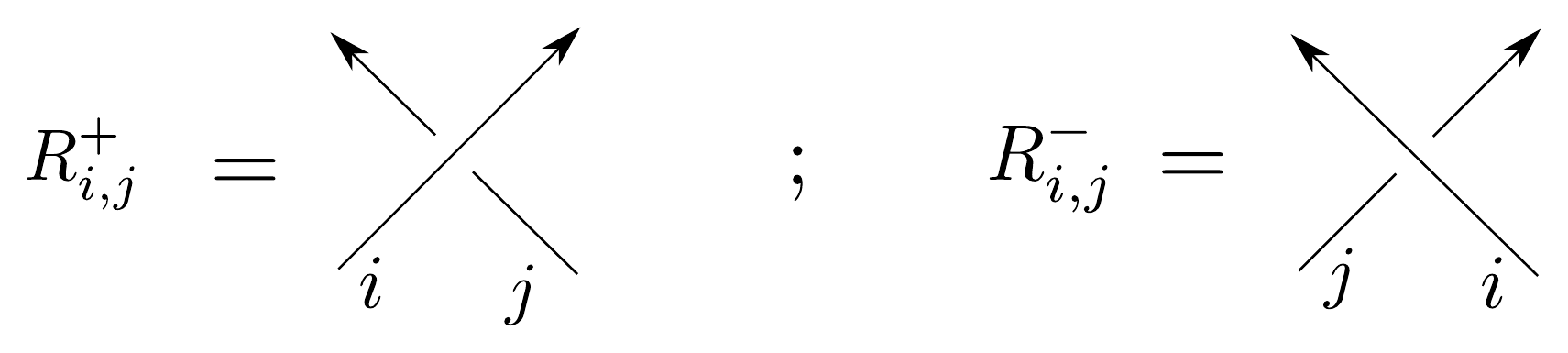}
  \end{center}
Here $i$ and $j$ are the labels of the incoming strands and $i$ is the label of the overstrand. Then we obtain the original w-tangle by stitching the crossings appropriately. For a concrete example, let us look at the long figure-eight knot: 
   \begin{center}
      \includegraphics[scale=0.6]{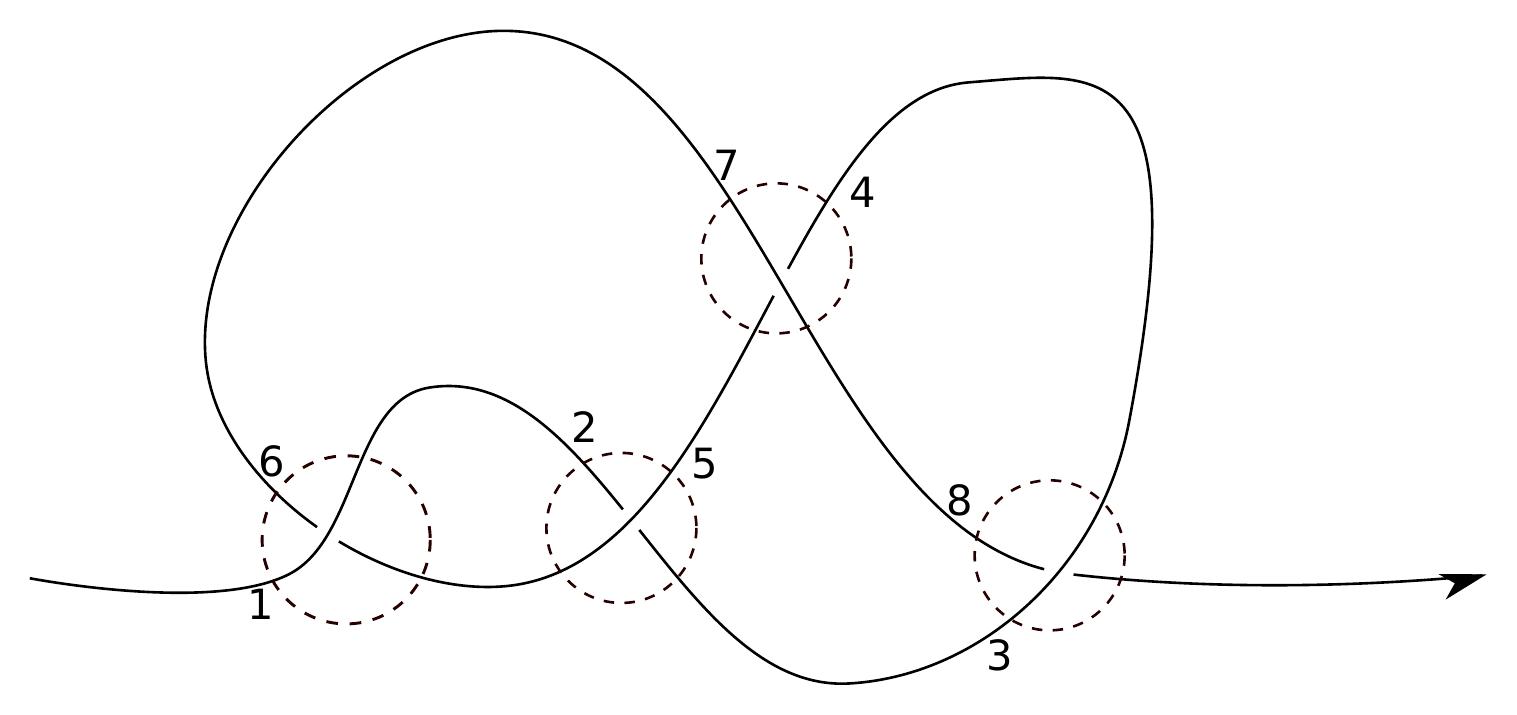}
   \end{center} 
With the labeling as in the above figure, we can first write the long figure-eight knot as a disjoint union of crossings: 
  \[
      R_{1,6}^+R_{5,2}^+R_{3,8}^-R_{7,4}^-.
  \]
The long figure-eight knot consists of four crossings, two positive and two negative. Then we stitch strand 1 to strand 2 through to strand 8. Therefore the long figure-eight knot is given by 
  \[
     R_{1,6}^+R_{5,2}^+R_{3,8}^-R_{7,4}^-\sslash m^{1,2}_1\sslash m^{1,3}_1\sslash m^{1,4}_1\sslash m^{1,5}_1\sslash m^{1,6}_1\sslash m^{1,7}_1\sslash m^{1,8}_1.
  \]
We can then summarize the above observation in the following proposition.   

\begin{prop}\label{finitegenerated}
  The meta-monoid $\W$ of w-tangles is \emph{generated} by the crossings $R_{i,j}^+$ and $R_{i,j}^-$, i.e. all expressions that can be formed using the crossings and the meta-monoid operations, modulo the relations $R1^s$, $R2$, $R3$ and $OC$ (Figure \ref{fig:Reidemeistermoves}).  
\end{prop}

\begin{remark} To describe a w-tangle using the categorical language, one would have to use positive crossings, negative crossings and virtual crossings. One would also need to include all the relations involving virtual crossings. In the context of meta-monoids, the relations $VR1$, $VR2$, $VR3$, $M$ are automatically true. Therefore meta-monoids give a more succinct way to talk about w-tangles. Note that meta-monoids provide a shift in perspective where one focuses on the strands instead of the endpoints (as opposed to the categorical language where we need to split the endpoints into the top and the bottom).   
\end{remark}

The next proposition will not be needed in the rest of the paper. It is a generalization of knot groups to w-tangles. 
\begin{prop}[\cite{BN13}]
  There is a meta-monoid homomorphism from the meta-monoid of w-tangles $\W$ to the meta-monoid of groups $\G$ given in Example \ref{metamonoidgroups}. 
\end{prop}

\begin{proof}
  An observant reader will realize that $(F,m,l)$ is simply the \emph{peripheral system} of a tangle. Given a w-tangle $T$, we can compute its fundamental group $F=\pi_1(T)$ using the Wirtinger presentation \cite{Rol03} (ignore virtual crossings). Then $m$ and $l$ are the images of the meridians and longitudes in $F$. More specifically, taking as the basepoint our eyes, for a strand labeled $x$, $m_x$ is the loop starting from the basepoint to the right of the tail of strand $x$, going perpendicularly to the left under strand $x$ and then back to the base point. For the longitudes, let $l_x$ be the loop starting from the basepoint to the right of the tail of strand $x$
and then going along the framing (here we use the convention of blackboard framing) to the head of $x$ and then back to the basepoint. For example, the image of the positive crossing $R_{i,j}^+$ is  
  \[R_{i,j}^+\mapsto (\left\langle i\right\rangle *\left\langle j \right\rangle,\{m_i=i,m_j=j\},\{l_i=1,l_j=i\}).\]
We leave it to the readers to verify the operations and the axioms.      
\end{proof}

\section{Gassner Calculus}\label{sec:gassnercalculus}

\subsection{Definition and Properties of Gassner Calculus} In this section we introduce a meta-monoid that will serve as the target space of an algebraic invariant for w-tangles. Let $\Gamma$ be the meta-monoid given as follows. For a finite set $X$, let $R_X$ be $\Q(\{t_i:i\in X\})$, the field of rational functions in the variables $t_i$, $i\in X$, and $M_{X\times X}(R_X)$ be the collection of $|X|\times|X|$ matrices with rows labeled by $y_i$, $i\in X$ and columns labeled by $x_j$, $j\in X$. Suppose that a finite set $X$ has the form $X=\{a,b\}\cup S$, where $S\cap \{a,b\}=\emptyset$. An element of $R_X\times M_{X\times X}(R_X)$ consists of an element in $R_X$, which we call the \emph{scalar part}, and an element in $M_{X\times X}(R_X)$, which we call the \emph{matrix part}, can be represented as   
  \[\left(\begin{array}{c|ccc}
      \omega & x_a & x_b & x_S \\ \hline
      y_a & \alpha &\beta &\theta\\
      y_b & \gamma & \delta & \epsilon\\
      y_S &\phi & \psi & \Xi
    \end{array}
  \right).
  \]
Let us explain a bit about the notations. For a finite \emph{ordered} set $S$, we let $x_S$ be the tuple $(x_i)_{i\in S}$ and $y_S$ be the tuple $(y_j)_{j\in S}$. Here $\theta$ and $\epsilon$ are row vectors (notice the horizontal line in each letter), whereas $\phi$ and $\psi$ are column vectors (notice the vertical line in each letter) and $\Xi$ is a square matrix (as evident from the shape of the letter $\Xi$). In general we require the rows and columns of the matrix part to have the same indices, but sometimes we also allow permutations of the rows and columns. So to avoid ambiguity a matrix part should come with labels of the rows and columns. 
  
Now let $\Gamma^X$ be the subset of $R_X\times M_{X\times X}(R_X)$ satisfying the condition
     \[
        \left(\begin{array}{c|c}
             \omega & x_X \\ \hline 
             y_X & M
          \end{array}
        \right)_{t_i\to 1}=\left(\begin{array}{c|c}
            \omega|_{t_i\to 1} & x_X \\ \hline
            y_X & I
          \end{array}
        \right).
     \]   
Here $t_i\to 1$ means substituting all the variables $t_i$ by 1 and $I$ is the identity matrix. In particular, we see that the matrix part is always invertible (since the determinant is not identically 0). Then the operations in a meta-monoid are given by, where $t_a\to t_b$ means substituting $t_a$ by $t_b$: 
 \begin{itemize}
   \item[] identity $\quad\left(\begin{array}{c|c}
        \omega & x_X \\
        \hline
        y_X & M
      \end{array}
    \right)\sslash e_a=\left(\begin{array}{c|cc}
        \omega & x_a & x_X \\ \hline
        y_a & 1 & \vec{0}\\
        y_X & \vec{0} & M
      \end{array}
    \right)$,
   \item[] disjoint union $\quad\left(\begin{array}{c|c}
        \omega_1 & x_{X_1} \\
        \hline
        y_{X_1} & M_1
      \end{array}
    \right)\sqcup \left(\begin{array}{c|c}
        \omega_2 & x_{X_2} \\
        \hline
        y_{X_2} & M_2
      \end{array}
    \right)=\left(\begin{array}{c|cc}
        \omega_1\omega_2 & x_{X_1} & x_{X_2} \\
        \hline
        y_{X_1} & M_1 & \vec{0}\\
        y_{X_2} & \vec{0} & M_2
      \end{array}
    \right)$, 
    \item[] deletion $\quad\left(\begin{array}{c|cc}
        \omega & x_a & x_S \\
        \hline
        y_a & \alpha & \theta \\
        y_S & \phi & \Xi
      \end{array}
    \right)\sslash \eta_a=\left(\begin{array}{c|c}
        \omega & x_S \\
        \hline
        y_S & \Xi
      \end{array}
    \right)_{t_a\to 1}
    $, 
  \item[] renaming $\quad
     \left(\begin{array}{c|cc}
        \omega & x_a & x_S \\
        \hline
        y_a & \alpha & \theta \\
        y_S & \phi & \Xi
      \end{array}
    \right)\sslash \sigma^a_b=\left(\begin{array}{c|cc}
        \omega & x_b & x_S \\
        \hline
        y_b & \alpha & \theta \\
        y_S & \phi & \Xi
      \end{array}
    \right)_{t_a\to t_b}
   $, 
 \item[] stitching 
  \begin{equation}\label{stitchingformula} 
  \left(
      \begin{array}{c|ccc}
         \omega & x_a & x_b & x_S \\
         \hline
         y_a & \alpha & \beta &\theta\\
         y_b & \gamma &\delta &\epsilon \\
         y_S &\phi &\psi &\Xi 
       \end{array}          
    \right)\sslash m^{a,b}_c=
   \left( \begin{array}{c|cc}
        (1-\gamma)\omega & x_c & x_S \\
        \hline 
        y_c & \beta+\frac{\alpha\delta}{1-\gamma} & \theta+\frac{\alpha\epsilon}{1-\gamma} \\
        y_S & \psi +\frac{\delta \phi}{1-\gamma} & \Xi +\frac{\phi\epsilon}{1-\gamma}
    \end{array}\right)_{t_a,t_b\to t_c}.
    \end{equation}  
 \end{itemize}
Here by $\vec{0}$ we denote a matrix of zeros with size depending on the context.  

\begin{lem}
  The above operations are well-defined, i.e. it makes sense to divide by $1-\gamma$ and when all the variables are set to 1, we obtain the identity matrix.
\end{lem} 

\begin{proof}
  The only non-trivial thing to check is the stitching operation. First of all, note that since $(1-\gamma)|_{t_x\to 1}=1$, it makes sense to divide by $1-\gamma$. Now when all the variables are set to 1, we have $\alpha=\delta=1$, and $\beta,\theta,\gamma,\epsilon,\phi,\psi$ all vanish, $\Xi$ is the identity matrix. Plug these into the matrix after stitching we obtain the identity matrix, as required. 
\end{proof}

The stitching formula may seem mysterious at first. Nevertheless it has an elementary interpretation in terms of linear algebra. Specifically we can think of the matrix part of
   \[\left(\begin{array}{c|ccc}
              \omega & x_a & x_b & x_S \\ \hline
              y_a & \alpha & \beta & \theta\\
              y_b & \gamma & \delta & \epsilon \\
              y_S & \phi & \psi &\Xi
         \end{array}     
     \right)
   \]
as an operator with input strands labeled by $y_a$, $y_b$, $y_S$ and output strands labeled by $x_a$, $x_b$, $x_S$. In other words, the strands are labeled by $\{a,b\}\cup S$. We label the tail of strand $a$ by $y_a$ and the head of strand $a$ by $x_a$. 
  \begin{center}
     \includegraphics[scale=0.4]{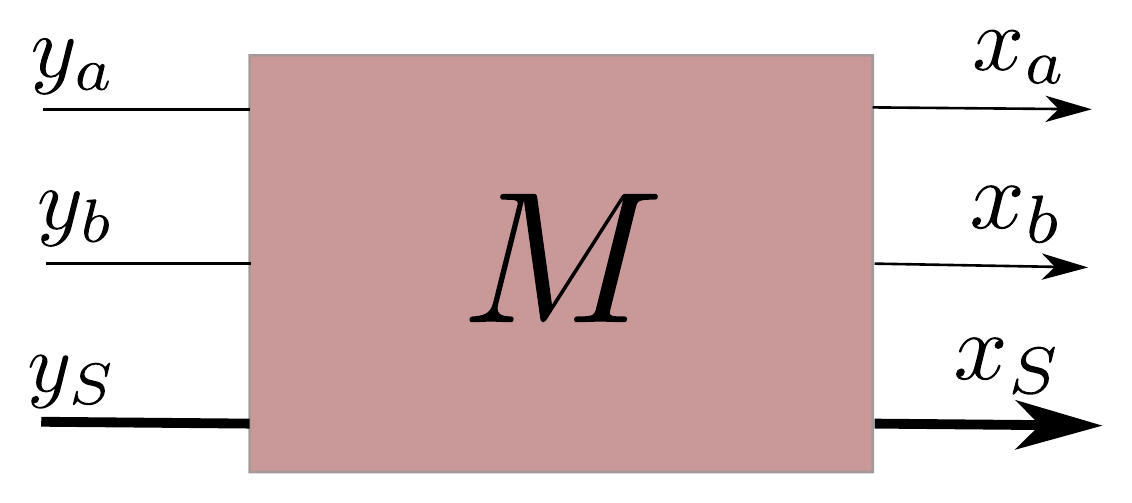}
  \end{center}
In the language of linear algebra we have a system of equations 
  \[
    \begin{cases}
       y_a =\alpha x_a+\beta x_b+\theta x_S, \\
       y_b = \gamma x_a +\delta x_b+\epsilon x_S,\\
       y_S = \phi x_a + \psi x_b+\Xi x_S.
    \end{cases}
  \]  
Now the stitching operation $m^{a,b}_c$ can be interpreted as connecting the head of strand $a$ to the tail of strand $b$ and labeling the resulting strand $c$
  \begin{center}
      \includegraphics[scale=0.4]{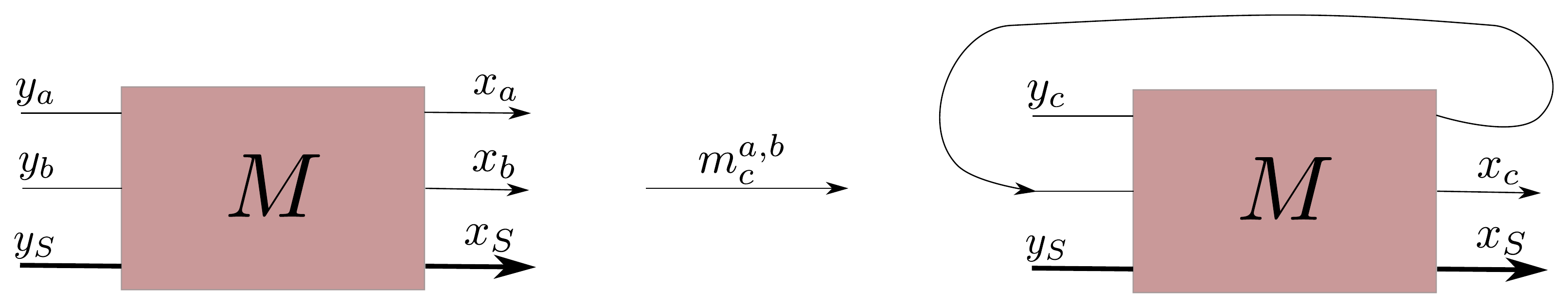}
  \end{center}
In terms of linear algebra we obtain the extra equation $y_b=x_a$. Plugging it in the second equation we obtain 
  \[x_a=\gamma x_a+\delta x_b+\epsilon x_S,\quad\text{i.e.}\quad x_a=\frac{\delta}{1-\alpha}x_b+\frac{\epsilon}{1-\gamma} x_S.\]
It follows that 
 \[
    \begin{cases}
      y_a=\left(\beta+\frac{\alpha\delta}{1-\gamma}\right)x_b+\left(\theta+\frac{\alpha\epsilon}{1-\gamma}\right)x_S\\
      y_S=\left(\psi+\frac{\delta\phi}{1-\gamma}\right)x_b+\left(\Xi+\frac{\phi\epsilon}{1-\gamma}\right)x_S
    \end{cases}
 \]    
Finally, since the new strand is labeled $c$, we need to rename the variables on strand $a$ and strand $b$, namely substituting $t_a$ and $t_b$ by $t_c$, and changing $y_a$ to $y_c$, $x_b$ to $x_c$:
   \[
    \begin{cases}
      y_c=\left(\beta+\frac{\alpha\delta}{1-\gamma}\right)_{t_a,t_b\to t_c}x_c+\left(\theta+\frac{\alpha\epsilon}{1-\gamma}\right)_{t_a,t_b\to t_c}x_S\\
      y_S=\left(\psi+\frac{\delta\phi}{1-\gamma}\right)_{t_a,t_b\to t_c}x_c+\left(\Xi+\frac{\phi\epsilon}{1-\gamma}\right)_{t_a,t_b\to t_c}x_S
    \end{cases}
 \]    
which is precisely the stitching formula for the matrix part. This does not tell us the formula for the scalar part however.

It will be useful to have a formula for stitching many strands as the same time, provided we do not stitch the same strand to itself (so we cannot stitch strand 1 to strand 2 and then strand 2 to strand 1). A priori the order in which these stitching operations are carried out matters. For instance, stitching strand 1 to strand 2 and then strand 2 to strand 3 may not be the same as stitching strand 2 to strand 3 and then strand 1 to strand 2. 

Consider an element $\zeta$ of $\Gamma^X$ 
  \[
    \zeta= \left(\begin{array}{c|c}
          \omega & x_X \\ \hline 
          y_X & M 
       \end{array}
     \right)
  \]
and given two vectors $\vec{a}=(a_1,a_2,\dots,a_n)$ and $\vec{b}=(b_1,b_2,\dots,b_n)$ where $a_i,b_j\in X$. Suppose we want to stitch strand $a_1$ to $b_1$, strand $a_2$ to $b_2$, $\dots$, strand $a_n$ to $b_n$ in that order, where $a_i$ and $b_j$ are chosen in such a way that we don't stitch the same strand to itself. We denote these operations simply by $m^{\vec{a},\vec{b}}$ (note that if we relabel the strands then we also have to rename the variables $t_i$ accordingly, but let us not worry about it now). In order to describe the result it is convenient to rearrange the matrix part as follows. Let $\vec{c}=X\setminus\vec{a}$ and $\vec{d}=X\setminus\vec{b}$, we can then rewrite $\zeta$ as
   \[
      \left(\begin{array}{c|cc}
      \omega & x_{\vec{a}} & x_{\vec{c}}\\ \hline
      y_{\vec{b}} & \gamma &\epsilon\\
      y_{\vec{d}} & \phi &\Xi
    \end{array}
  \right).
   \]
We record the stitching-in-bulk formula in the next proposition. 

\begin{prop}[Stitching in Bulk]\label{stitchingbulk}
 With the above data we have 
  \begin{equation}\label{stitchingbulkformula}
  \left(\begin{array}{c|cc}
      \omega & x_{\vec{a}} & x_{\vec{c}}\\ \hline
      y_{\vec{b}} & \gamma &\epsilon\\
      y_{\vec{d}} & \phi &\Xi
    \end{array}
  \right)\xrightarrow{m^{\vec{a},\vec{b}}} \left(\begin{array}{c|c}
     \omega\det(I-\gamma) & x_{\vec{c}}\\ \hline
     y_{\vec{d}} & \Xi+\phi(I-\gamma)^{-1}\epsilon
   \end{array}
  \right),
  \end{equation}  
where $I$ denotes the $n\times n$ identity matrix and each of $\gamma$, $\epsilon$, $\phi$, $\Xi$ is a square matrix.  
\end{prop}

Note that if we relabel the strands after doing the stitching operations then we need to relabel the $x's$ and $y's$ and rename the variables. But we can just do this at the end.    
 
\begin{proof}
 We will prove the formula by induction on the number $n$ of strands being stitched. When $n=1$, let us show that we recover the stitching formula \eqref{stitchingformula}. Suppose we want to stitch strand $a$ to strand $b$, we first write the left hand side as follows.  
   \[\left(\begin{array}{c|ccc}
      \omega & x_a & x_b & x_S\\ \hline
      y_b & \gamma & \delta & \epsilon \\
      y_a & \alpha & \beta & \theta\\
      y_S & \phi & \psi & \Xi
   \end{array}
   \right).
   \]  
Then $\omega\mapsto \omega(1-\gamma)$, and 
  \[
    \left(\begin{array}{c|cc}
        \omega(1-\gamma) & x_b & x_S \\ \hline
        y_a & \multicolumn{2}{c}{\multirow{2}{*}{$          \begin{pmatrix}
        \beta & \theta\\ 
        \psi & \Xi
       \end{pmatrix}+\begin{pmatrix}
       \alpha \\
       \phi
      \end{pmatrix}(1-\gamma)^{-1}\begin{pmatrix}
       \delta & \epsilon
        \end{pmatrix}$}}\\
        y_S & &
      \end{array}
    \right)=\left(\begin{array}{c|cc}
      \omega(1-\gamma) & x_b & x_S \\ \hline
      y_a & \beta+\frac{\alpha\delta}{1-\gamma} & \theta+\frac{\alpha\epsilon}{1-\gamma} \\
      y_S & \psi+\frac{\delta\phi}{1-\gamma} & \Xi+\frac{\phi\epsilon}{1-\gamma}
      \end{array}
    \right)
  \]  
which yields the stitching formula \eqref{stitchingformula} (before the renaming of the variables). Now for the induction step, we write $\vec{a}=(\vec{a}',a)$ and $\vec{b}=(\vec{b}',b)$, then from the inductive hypothesis   
  \[\left(\begin{array}{c|ccc}
      \omega & x_{\vec{a}'} & x_a & x_{\vec{c}} \\ \hline
      y_{\vec{b}'} & \gamma_1 &\gamma_2 &\epsilon_1 \\
      y_b & \gamma_3 &\gamma_4 &\epsilon_2\\
      y_{\vec{d}} &\phi_1 &\phi_2 &\Xi
     \end{array}
    \right)\xrightarrow[]{m^{\vec{a}',\vec{b}'}} 
    \left(\begin{array}{c|cc}
       \omega\det(I-\gamma_1) & x_a & x_{\vec{c}}\\ \hline
       y_b & \gamma_4+\gamma_3(I-\gamma_1)^{-1}\gamma_2 & \epsilon_2+\gamma_3(I-\gamma_1)^{-1}\epsilon_1\\
       y_{\vec{d}} & \phi_2+\phi_1(I-\gamma_1)^{-1}\gamma_2 & \Xi+\phi_1(I-\gamma_1)^{-1}\epsilon_1
     \end{array}
    \right).
   \] 
Then stitching strand $a$ to strand $b$ we obtain 
  \[
   \left(\begin{array}{c|c}
      \omega(1-\gamma_4-\gamma_3(I-\gamma_1)^{-1}\gamma_2)\det(I-\gamma_1) & x_{\vec{c}}\\ \hline
      y_{\vec{d}} & \Xi+\phi_1(I-\gamma_1)^{-1}\epsilon_1+\frac{(\phi_2+\phi_1(I-\gamma_1)^{-1}\gamma_2)(\epsilon_2+\gamma_3(I-\gamma_1)^{-1}\epsilon_1)}{1-\gamma_4-\gamma_3(I-\gamma_1)^{-1}\gamma_2}
    \end{array}
   \right).
  \]
To finish the induction step we need to show that the above is the same as 
  \[
     \left(\begin{array}{c|c}
        \omega\det\left[I-\begin{pmatrix}
           \gamma_1 & \gamma_2 \\
           \gamma_3 & \gamma_4
        \end{pmatrix}\right] & x_{\vec{c}} \\ \hline
        y_{\vec{d}} & \Xi+\begin{pmatrix}
          \phi_1 & \phi_2
        \end{pmatrix}\left(I-\begin{pmatrix}
          \gamma_1 & \gamma_2 \\
          \gamma_3 & \gamma_4
        \end{pmatrix}\right)^{-1}\begin{pmatrix}
          \epsilon_1 \\
          \epsilon_2
        \end{pmatrix}
      \end{array}
     \right).
  \] 
For that we record the following elementary result from linear algebra (see \cite{Pow11})

\begin{lem}\label{blockdeterminant}
     Consider the block matrix 
       \[\begin{pmatrix}
           A & B \\
           C & D
       \end{pmatrix}
       \]  
    where $A$ and $D$ are square matrices not necessarily of the same size and $D$ is invertible. Then 
     \[\det\begin{pmatrix}
           A & B \\
           C & D
       \end{pmatrix}=\det(A-BD^{-1}C)\det(D).\]
   \end{lem}
   
 \begin{proof}[Proof of lemma]
  It is easy to check that
    \[
     \begin{pmatrix}
       A & B\\
       C & D
     \end{pmatrix} 
     \begin{pmatrix}
      I & 0 \\
      -D^{-1}C & I 
     \end{pmatrix}=\begin{pmatrix}
       A-BD^{-1}C & B \\
       0 & D
     \end{pmatrix}.
    \]
  Now taking the determinant of both sides and using the fact that the determinant of a block triangular matrix is the product of the determinants of the diagonal blocks (one can prove this by induction) we obtain the required identity.  
 \end{proof}

Back to our proof, we have that  
 \begin{align*}
    \det\left[I-\begin{pmatrix}
           \gamma_1 & \gamma_2 \\
           \gamma_3 & \gamma_4
        \end{pmatrix}\right]&=\det\begin{pmatrix}
          I-\gamma_1 & -\gamma_2 \\
          -\gamma_3 & 1-\gamma_4
        \end{pmatrix}=\det\begin{pmatrix}
          1-\gamma_4 & -\gamma_3 \\
          -\gamma_2 & I-\gamma_1
        \end{pmatrix}\\
        &=\det(1-\gamma_4-\gamma_3(I-\gamma_1)^{-1}\gamma_2)\det(I-\gamma_1),
   \end{align*}
which agrees with the scalar part. Now for the matrix part, we have to show that 
  \begin{align*}
    \begin{pmatrix}
      I-\gamma_1 & -\gamma_2 \\
      -\gamma_3 & 1-\gamma_4
    \end{pmatrix}^{-1}=\begin{pmatrix}
      (I-\gamma_1)^{-1}+\frac{(I-\gamma_1)^{-1}\gamma_2\gamma_3(I-\gamma_1)^{-1}}{1-\gamma_4-\gamma_3(I-\gamma_1)^{-1}\gamma_2} & \frac{(I-\gamma_1)^{-1}\gamma_2}{1-\gamma_4-\gamma_3(I-\gamma_1)^{-1}\gamma_2}\\
      \frac{\gamma_3(I-\gamma_1)^{-1}}{1-\gamma_4-\gamma_3(I-\gamma_1)^{-1}\gamma_2} & \frac{1}{1-\gamma_4-\gamma_3(I-\gamma_1)^{-1}\gamma_2}
    \end{pmatrix},
  \end{align*}    
which we can verify easily by computing the products of the two matrices. We leave the details to the readers. (Note that we ignore the renaming of the variables, which we can just do at the end.)        
\end{proof} 

Let us present a shortcut to obtain formula \eqref{stitchingbulkformula}. The matrix part gives the system of equation 
  \[
      \begin{cases}
         y_{\vec{b}}=\gamma x_{\vec{a}}+\epsilon x_{\vec{c}}, \\
         y_{\vec{d}}=\phi x_{\vec{a}}+\Xi x_{\vec{c}}.
      \end{cases}
  \]
Now the stitching instruction yields the equation $y_{\vec{b}}=x_{\vec{a}}$. Thus the first equation becomes
 \[
    x_{\vec{a}}=\gamma x_{\vec{a}}+\epsilon x_{\vec{c}},\quad\text{or}\quad x_{\vec{a}}=(I-\gamma)^{-1}\epsilon x_{\vec{c}}.
 \]
Plugging it in the second equation we obtain  
   \[
      y_{\vec{d}}=(\Xi+\phi(I-\gamma)^{-1}\epsilon)x_{\vec{c}},
   \]
as required. 
  
As a corollary, when $a_i\neq b_j$ for $1\leq i,j\leq n$, we have the following stitching formula
   \begin{equation}\label{stitchingdisjoint}
     \left(
      \begin{array}{c|ccc}
         \omega & x_{\vec{a}} & x_{\vec{b}} & x_S \\
         \hline
         y_{\vec{a}} & \alpha & \beta &\theta\\
         y_{\vec{b}} & \gamma &\delta &\epsilon \\
         y_S &\phi &\psi &\Xi 
       \end{array}          
    \right)\xrightarrow[] {m^{\vec{a},\vec{b}}_{\vec{c}}}
      \left(\begin{array}{c|cc}
         \det(I-\gamma)\omega & x_{\vec{c}} & x_S \\
         \hline
         y_{\vec{c}} & \beta + \alpha(I-\gamma)^{-1}\delta & \theta +\alpha(I-\gamma)^{-1}\epsilon \\
         y_S & \psi+\phi(I-\gamma)^{-1}\delta & \Xi+\phi(I-\gamma)^{-1}\epsilon 
      \end{array}
      \right)_{t_{\vec{a}},t_{\vec{b}}\to t_{\vec{c}}}.
   \end{equation}
The proof is a straightforward application of formula \eqref{stitchingbulkformula}.

\begin{prop}
  The order in which one performs the stitching operations does not matter. 
\end{prop}

\begin{proof}
  From formula \eqref{stitchingbulkformula} we see that switching two stitching operations amounts to switching the corresponding entries of $y_{\vec{b}}$ and $x_{\vec{a}}$, which in turn will switch the corresponding columns of $\gamma$ and $\epsilon$ and the corresponding rows of $\gamma$ and $\phi$. The matrix $\Xi$ stays unchanged. Therefore 
    \[
      \Xi+\phi(I-\gamma)^{-1}\epsilon
    \]
 will be invariant. For the scalar part, since we switch the rows and columns of $\gamma$ of the same indices, we preserve $I$ and the determinant is unchanged. (One can make the argument more precise using permutation matrices.)   
\end{proof}

Let us illustrate the above proposition in a concrete case to show meta-associativity. Let 
  \[
     \zeta=\left(\begin{array}{c|cccc}
         \omega & x_1 & x_2 & x_3 & x_S \\ \hline 
         y_1 & \alpha_{11} & \alpha_{12} & \alpha_{13} & \theta_1 \\
         y_2 & \alpha_{21} & \alpha_{22} & \alpha_{23} & \theta_2 \\
         y_3 & \alpha_{31} & \alpha_{32} & \alpha_{33} & \theta_3 \\
         y_S & \phi_1 & \phi_2  & \phi_3 & \Xi
       \end{array}
     \right).
  \]
To stitch strand 1 to strand 2 and strand 2 to strand 3 we rewrite $\zeta$ as 
  \[
     \left(\begin{array}{c|cccc}
           \omega & x_1 & x_2 & x_3 & x_S \\ \hline
           y_2 & \alpha_{21} & \alpha_{22} & \alpha_{23} & \theta_2 \\
           y_3 & \alpha_{31} & \alpha_{32} & \alpha_{33} &\theta_3 \\
           y_1 & \alpha_{11} & \alpha_{12} & \alpha_{13} & \theta_1 \\
           y_S & \phi_1 & \phi_2 & \phi_3 & \Xi
       \end{array}
     \right).
  \]    
Then $\zeta\sslash m^{1,2}_1 \sslash m^{1,3}_1$ is given by  
  \[
    \left(
       \begin{array}{c|cc} 
          \omega \det\begin{pmatrix}
            1-\alpha_{21} & -\alpha_{22} \\
            -\alpha_{31} & 1-\alpha_{32}
          \end{pmatrix} & x_1 & x_S \\ \hline
          y_1 & \multicolumn{2}{c}{\multirow{2}{*}{$
             \begin{pmatrix}
                \alpha_{13} & \theta_1 \\
                \phi_3 & \Xi
             \end{pmatrix}+\begin{pmatrix} 
               \alpha_{11} & \alpha_{12} \\
               \phi_1 & \phi_2
             \end{pmatrix}\begin{pmatrix}
                1-\alpha_{21} & -\alpha_{22} \\
                -\alpha_{31} & 1-\alpha_{32}
             \end{pmatrix}^{-1}\begin{pmatrix}
                \alpha_{23} & \theta_2 \\
                \alpha_{33} & \theta_3
             \end{pmatrix}      
          $}} \\
          y_S & &
       \end{array}
    \right)_{t_2,t_3\to t_1}.
  \]
Similarly $\zeta\sslash m^{2,3}_2 \sslash m^{1,2}_1$ is given by 
  \[
    \left(
       \begin{array}{c|cc} 
          \omega \det\begin{pmatrix}
            1-\alpha_{32} & -\alpha_{31} \\
            -\alpha_{22} & 1-\alpha_{21}
          \end{pmatrix} & x_1 & x_S \\ \hline
          y_1 & \multicolumn{2}{c}{\multirow{2}{*}{$
             \begin{pmatrix}
                \alpha_{13} & \theta_1 \\
                \phi_3 & \Xi
             \end{pmatrix}+\begin{pmatrix} 
               \alpha_{12} & \alpha_{11} \\
               \phi_2 & \phi_1
             \end{pmatrix}\begin{pmatrix}
                1-\alpha_{32} & -\alpha_{31} \\
                -\alpha_{22} & 1-\alpha_{21}
             \end{pmatrix}^{-1}\begin{pmatrix}
                \alpha_{33} & \theta_3 \\
                \alpha_{23} & \theta_2
             \end{pmatrix}      
          $}} \\
          y_S & &
       \end{array}
    \right)_{t_2,t_3\to t_1}.
  \]
Observe that 
  \begin{align*}
     &\begin{pmatrix} 
               \alpha_{12} & \alpha_{11} \\
               \phi_2 & \phi_1
             \end{pmatrix}\begin{pmatrix}
                1-\alpha_{32} & -\alpha_{31} \\
                -\alpha_{22} & 1-\alpha_{21}
             \end{pmatrix}^{-1}\begin{pmatrix}
                \alpha_{33} & \theta_3 \\
                \alpha_{23} & \theta_2
             \end{pmatrix}\\
      =&\begin{pmatrix} 
               \alpha_{11} & \alpha_{12} \\
               \phi_1 & \phi_2
             \end{pmatrix}\begin{pmatrix}
                 0 & 1 \\
                 1 & 0
             \end{pmatrix}\begin{pmatrix}
                1-\alpha_{32} & -\alpha_{31} \\
                -\alpha_{22} & 1-\alpha_{21}
             \end{pmatrix}^{-1}\begin{pmatrix}
                   0 & 1\\
                   1 & 0
             \end{pmatrix}\begin{pmatrix}
                \alpha_{23} & \theta_2 \\
                \alpha_{33} & \theta_3
             \end{pmatrix} \\
       =&\begin{pmatrix} 
               \alpha_{11} & \alpha_{12} \\
               \phi_1 & \phi_2
             \end{pmatrix}\begin{pmatrix}
                1-\alpha_{21} & -\alpha_{22} \\
                -\alpha_{31} & 1-\alpha_{32}
             \end{pmatrix}^{-1}\begin{pmatrix}
                \alpha_{23} & \theta_2 \\
                \alpha_{33} & \theta_3
             \end{pmatrix}.                    
  \end{align*}
Thus it follows that   
  \[
    \zeta\sslash m^{1,2}_1\sslash m^{1,3}_1=\zeta\sslash m^{2,3}_2\sslash m^{1,2}_1.\]
 This establishes the meta-associativity property. The other axioms of a meta-monoid are straightforward to verify. Thus $\Gamma$ is indeed a meta-monoid. The meta-monoid $\Gamma$ is called the \emph{Gassner Calculus} or \emph{$\Gamma$-Calculus}, for reasons which will be clear below (Proposition \ref{gammabraid}).     

\begin{prop}
  There is a meta-monoid homomorphism $\varphi$ from the meta-monoid $\W$ of w-tangles to $\Gamma$-calculus.
\end{prop}

\begin{proof}
  By Proposition \ref{finitegenerated} we only need to define $\varphi$ on the generators and then check that the relations given in Figure \ref{fig:Reidemeistermoves} are satisfied. We set 
   \[\varphi(R_{a,b}^{\pm})= \left(\begin{array}{c|cc}
     1 & x_a & x_b\\
    \hline
    y_a & 1 & 1-t_a^{\pm 1} \\
    y_b & 0 & t_a^{\pm}
  \end{array}\right).\]
Let us check the Reidemeister $R3$ move 
  \begin{center}
    \includegraphics[scale=0.4]{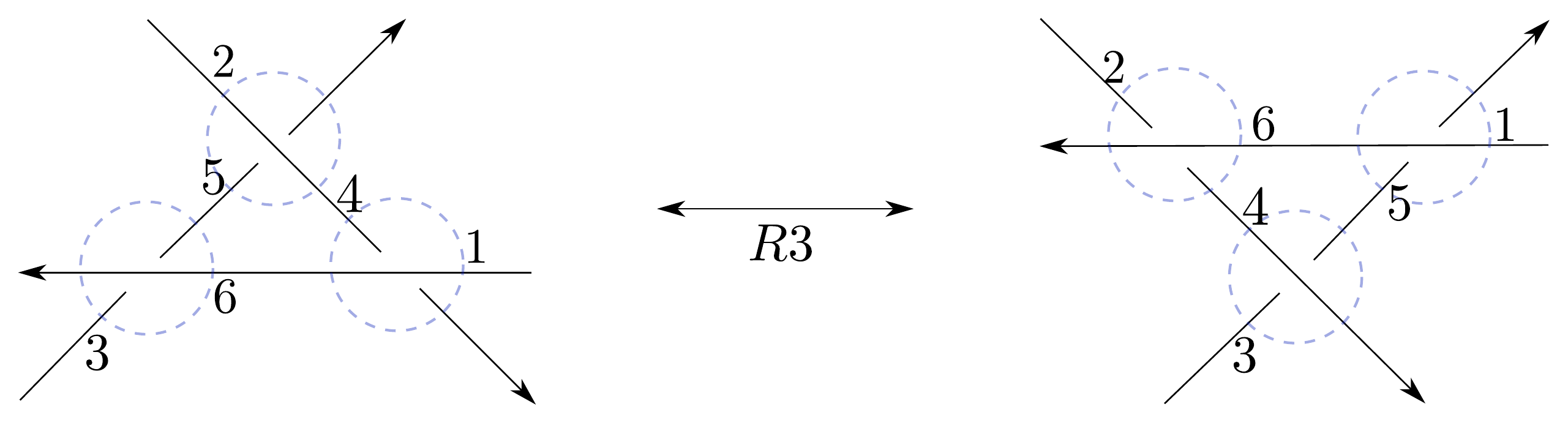}
  \end{center}   
In the language of meta-monoids we need to show that 
  \[
   \varphi(R_{1,4}^+R_{2,5}^+R_{6,3}^-)\sslash m^{1,6}_1\sslash m^{2,4}_2\sslash m^{3,5}_3=\varphi(R_{1,5}^-R_{4,3}^+R_{6,2}^+)\sslash m^{1,6}_1\sslash m^{2,4}_2\sslash m^{3,5}_3.  \]  
 Let us first compute the left hand side. The image of $R_{1,4}^+R_{2,5}^+R_{6,3}^-$ under $\varphi$ is 
   \[
     \left(\begin{array}{c|cccccc}
         1 & x_1 & x_2 & x_3 & x_4 & x_5 & x_6 \\ \hline 
         y_1 & 1 & 0 & 0 & 1-t_1 & 0 & 0\\           
         y_2 & 0 & 1 & 0 & 0 & 1-t_2 & 0\\
         y_3 & 0 & 0 & t_6^{-1} & 0 & 0 & 0 \\
         y_4 & 0 & 0 & 0 & t_1 & 0 & 0\\
         y_5 & 0 & 0 & 0 & 0 & t_2 & 0\\
         y_6 & 0 & 0 & 1-t_6^{-1} & 0 & 0 & 1
       \end{array}
     \right).
   \]   
 To perform all the stitching operations at once we rearrange the rows and columns as follows. 
   \[
      \left(\begin{array}{c|cccccc}
          1 & x_1 & x_2 & x_3 & x_4 & x_5 & x_6 \\ \hline
          y_6 & 0 & 0 & 1-t_6^{-1} & 0 & 0 & 1 \\
          y_4 & 0 & 0 & 0 & t_1 & 0 & 0 \\
          y_5 & 0 & 0 & 0 & 0 & t_2 & 0 \\
          y_1 & 1 & 0 & 0 & 1-t_1 & 0 & 0\\
          y_2 & 0 & 1 & 0 & 0 & 1-t_2 & 0\\
          y_3 & 0 & 0 & t_6^{-1} & 0 & 0 & 0
        \end{array}
      \right).
   \]  
Then according to formula \eqref{stitchingbulkformula}, the left hand side is given by 
  \[
     \left(\begin{array}{c|ccc}
         1 & x_4 & x_5 & x_6 \\ \hline
         y_1 & 1-t_1 & t_2-\frac{t_2}{t_6} & 1 \\
         y_2 & t_1 & 1-t_2 & 0 \\
         y_3 & 0 & \frac{t_2}{t_6} & 0
       \end{array}
     \right).
  \]  
According to the relabeling we rename $x_4\to x_2$, $x_5\to x_3$, $x_6\to x_1$ and $t_4\to t_2$, $t_5\to t_3$, $t_6\to t_1$ to obtain 
   \[
     \left(\begin{array}{c|ccc}
         1 & x_2 & x_3 & x_1 \\ \hline
         y_1 & 1-t_1 & t_2-\frac{t_2}{t_1} & 1 \\
         y_2 & t_1 & 1-t_2 & 0 \\
         y_3 & 0 & \frac{t_2}{t_1} & 0
       \end{array}
     \right).
  \] 
Finally we rearrange the columns 
   \[\left(\begin{array}{c|ccc}
         1 & x_1 & x_2 & x_3 \\ \hline 
         y_1 & 1 & 1-t_1 & t_2-\frac{t_2}{t_1}\\
         y_2 & 0 & t_1 & 1-t_2 \\
         y_3 & 0 & 0 & \frac{t_2}{t_1}
       \end{array}
     \right).
   \] 
We leave it as an exercise to show that the right hand side also yields the same result. Note that the Reidemeister move $R1$  
   \begin{center}
       \includegraphics[scale=0.5]{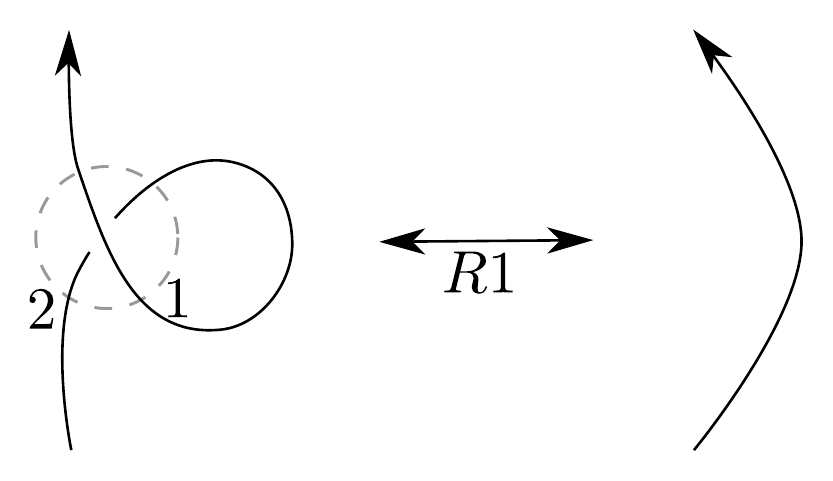}   
   \end{center} 
does not hold because the left hand side is given by  
  \[
    \varphi(R_{1,2}^-)\sslash m^{2,1}_1=\left(\begin{array}{c|c}
       t_1^{-1} & x_1 \\ \hline
       y_1 & 1
     \end{array}
    \right)
  \]
whereas the right hand side is trivial. The other relations are straightforward to verify.
\end{proof}

\begin{example}\label{longvsclosed}
  In this example we show that long w-knots and closed w-knots are not equivalent. Consider the long w-knots $L$ and $L'$ given by
   \begin{center}
     \includegraphics[scale=0.5]{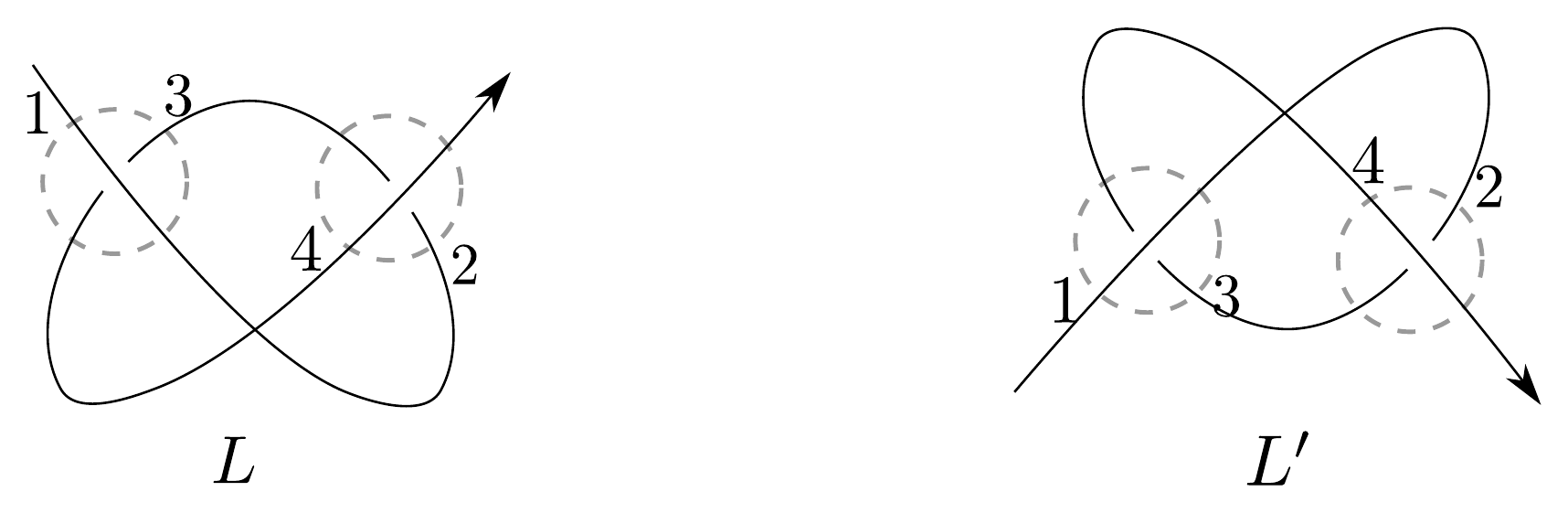}
   \end{center}
In the language of meta-monoids, $L$ has the description 
  \[L=R_{1,3}^-R_{4,2}^+\sslash m^{1,2}_1\sslash m^{1,3}_1 \sslash m^{1,4}_1.\] 
Then its invariant in $\Gamma$-calculus is 
 \[\varphi(L)=\left(\begin{array}{c|c}
     2-t_1^{-1} & x_1 \\
     \hline 
     y_1 & 1
   \end{array}
 \right).\]
In the language of meta-monoids, $L'$ has the description 
   \[L'=R_{1,3}^+R_{4,2}^-\sslash m^{1,2}_1\sslash m^{1,3}_1\sslash m^{1,4}_1.\]
So its invariant in $\Gamma$-calculus is 
   \[\varphi(L')=\left(\begin{array}{c|c}
     2-t_1 & x_1 \\
     \hline 
     y_1 & 1
   \end{array}
 \right).\] 
Thus $L$ and $L'$ are not isotopic as long w-knots and are non-trivial. However when we close $L$ and $L'$ we obtain the trivial (closed) knot. \hfill $\clubsuit$         
\end{example}

Observe that Proposition \ref{finitegenerated} gives an inductive framework to prove properties for w-tangles. Namely, one first check the property for the crossings, and then show that the property still holds under disjoint union and stitching. Let us illustrate this method with an important property of w-tangles.

\begin{prop}\label{columnsum}
    Let $T$ be a w-tangle whose components are labeled by the set $X$ and 
     \[\varphi(T)=\left(\begin{array}{c|c}
            \omega & X\\ \hline
            X & M
          \end{array}
     \right).\]
  Then the sum of the entries in each column of $M$ is 1.    
 \end{prop}

\begin{proof}
    The property clearly holds for crossings and is preserved under disjoint union. So we only need to show that it is invariant under stitching:
    \[
    \left(
      \begin{array}{c|ccc}
         \omega & a & b & S \\
         \hline
         a & \alpha & \beta &\theta\\
         b & \gamma &\delta &\epsilon \\
         S &\phi &\psi &\Xi 
       \end{array}          
    \right) \xrightarrow[t_a,t_b\rightarrow t_c]{m^{a,b}_c} 
   \left( \begin{array}{c|cc}
        (1-\gamma)\omega & c & S \\
        \hline 
        c & \beta+\frac{\alpha\delta}{1-\gamma} & \theta+\frac{\alpha\epsilon}{1-\gamma} \\
        S & \psi +\frac{\delta \phi}{1-\gamma} & \Xi +\frac{\phi\epsilon}{1-\gamma}
    \end{array}\right).
 \]  
 Assume that the property is true for the matrix on the left, i.e. 
   \[\begin{cases}
      \alpha+\gamma+\left\langle\phi\right\rangle=1\\
      \beta+\delta+\left\langle\psi\right\rangle=1\\
      \theta+\epsilon+\left\langle \Xi\right\rangle=\vec{1},
   \end{cases}\]
where $\vec{1}$ denotes a matrix consisting of 1's with size depending on the context and $\left\langle \vec{c}\right\rangle$ of a column vector $\vec{c}$ means taking the sum of the entries. For the case of $\Xi$, we apply $\left\langle\right\rangle$ to each column to obtain a row vector. Then we have 
 \begin{align*}
   \beta+\frac{\delta\alpha}{1-\gamma}+\left\langle\psi\right\rangle+\frac{\delta\left\langle \phi\right\rangle}{1-\gamma}=1-\delta+\frac{\delta(\alpha+\left\langle\phi\right\rangle)}{1-\gamma}=1-\delta+\frac{\delta(1-\gamma)}{1-\gamma}=1,
 \end{align*}
and 
 \begin{align*}
   \theta+\frac{\alpha\epsilon}{1-\gamma}+\left\langle\Xi\right\rangle+\frac{\left\langle\phi\right\rangle\epsilon}{1-\gamma}=\vec{1}-\epsilon+\frac{(\alpha+\left\langle\phi\right\rangle)\epsilon}{1-\gamma}=\vec{1}-\epsilon+\frac{(1-\gamma)\epsilon}{1-\gamma}=\vec{1},  
 \end{align*} 
as required. 
\end{proof}  

As a corollary we have that when $K$ is a long w-knot the matrix part is 1, so only the scalar part is interesting, i.e.
  \[
    \varphi(K)=\left(\begin{array}{c|c}
        \omega_K & x_1 \\ \hline 
        y_1 & 1
      \end{array}
    \right),
  \]
where we denote the scalar part by $\omega_K$.

\begin{example}
   Let us look at the long trefoil 
      \begin{center}
         \includegraphics[scale=0.5]{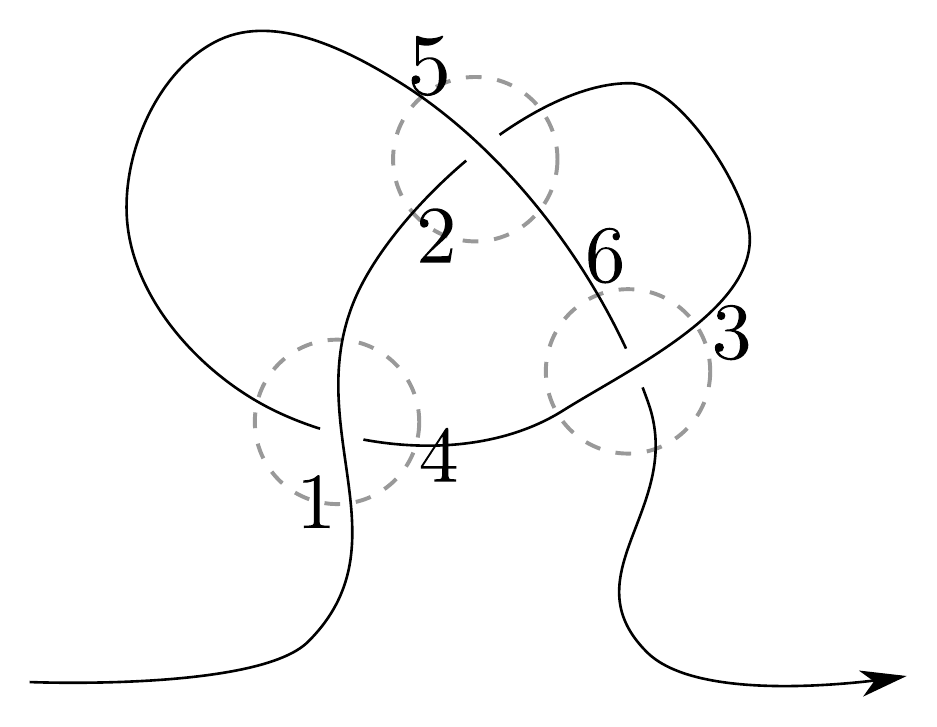}
      \end{center}
  It is given by 
    \[
       R_{1,4}^+R_{5,2}^+R_{3,6}^+\sslash m^{1,2}_1\sslash m^{1,3}_1\sslash m^{1,4}_1\sslash m^{1,5}_1\sslash m^{1,6}_1.
    \]
 The image of $R_{1,4}^+R_{5,2}^+R_{3,6}^+$ under $\varphi$ is 
  \[
     \left(\begin{array}{c|cccccc}
          1 & x_1 & x_2 & x_3 & x_4 & x_5 & x_6 \\ \hline
          y_1 & 1 & 0 & 0 & 1-t_1 & 0 & 0 \\
          y_2 & 0 & t_5 & 0 & 0 & 0 & 0 \\
          y_3 & 0 & 0 & 1 & 0 & 0 & 1-t_3 \\
          y_4 & 0 & 0 & 0 & t_1 & 0 & 0\\
          y_5 & 0 & 1-t_5 & 0 & 0 & 1 & 0\\
          y_6 & 0 & 0 & 0 & 0 & 0 & t_3
       \end{array}
     \right).
  \]
After we perform all the stitching operations the matrix part is 1 and the scalar part by formula \eqref{stitchingbulkformula} is the determinant of the matrix $I-\gamma$, where $\gamma$ is obtained by removing  the first row and the last column of the above matrix, i.e.
  \[
    \det\begin{pmatrix}
       1 & -t_5 & 0 & 0 & 0\\
       0 & 1 & -1 & 0 & 0 \\
       0 & 0 & 1 & -t_1 & 0\\
       0 & t_5-1 & 0 & 1 & -1 \\
       0 & 0 & 0 & 0 & 1
    \end{pmatrix}_{t_5\to t_1}=1-t+t^2,
  \]          
which one recognizes to be the Alexander polynomial of the trefoil (Proposition \ref{omegaalexander}). \hfill $\clubsuit$  
\end{example}

As another application of the stitching-in-bulk formula \eqref{stitchingbulkformula}, observe that a priori, the scalar $\omega$ and the matrix entries are rational functions. However, it turns out that $\omega$ is a Laurent polynomial, as shown in the following proposition. 
 
\begin{prop}\label{polynomial}
  Let $T$ be a w-tangle with scalar part $\omega$ and matrix part $M$, then $\omega$ is a Laurent polynomial and $\omega M$ is a matrix whose entries are Laurent polynomials. 
\end{prop} 

\begin{proof}
  One can obtain $T$ starting with a collection of crossings and then stitching all these crossings at once using formula \eqref{stitchingbulkformula}. Observe that when we take the disjoint union of crossings, the matrix part consists of Laurent polynomials (since each crossing is) and the scalar part is 1. Then after stitching the scalar part becomes $\det(I-\gamma)$ where $\gamma$ is specified by the stitching instruction. Since $\gamma$ consists of Laurent polynomials, $\det(I-\gamma)$ is a Laurent polynomials, thus establishes the polynomiality of $\omega$. Now for the other property, we look at 
   \[\omega\det(I-\gamma)(\Xi+\phi(I-\gamma)^{-1}\epsilon).\]
  All the matrices have Laurent polynomial entries, except for $(I-\gamma)^{-1}$. Recall that $(I-\gamma)^{-1}$ can be computed by dividing its adjugate (which are Laurent polynomials) by $\det(I-\gamma)$. Therefore multiplying with $\det(I-\gamma)$ removes the denominator, and so the resulting entries are Laurent polynomials.   
\end{proof}

\begin{example}\label{tangledemo}
  Let us compute the invariant for the tangle $T$ given by. 
   \begin{center}
       \includegraphics[scale=0.35]{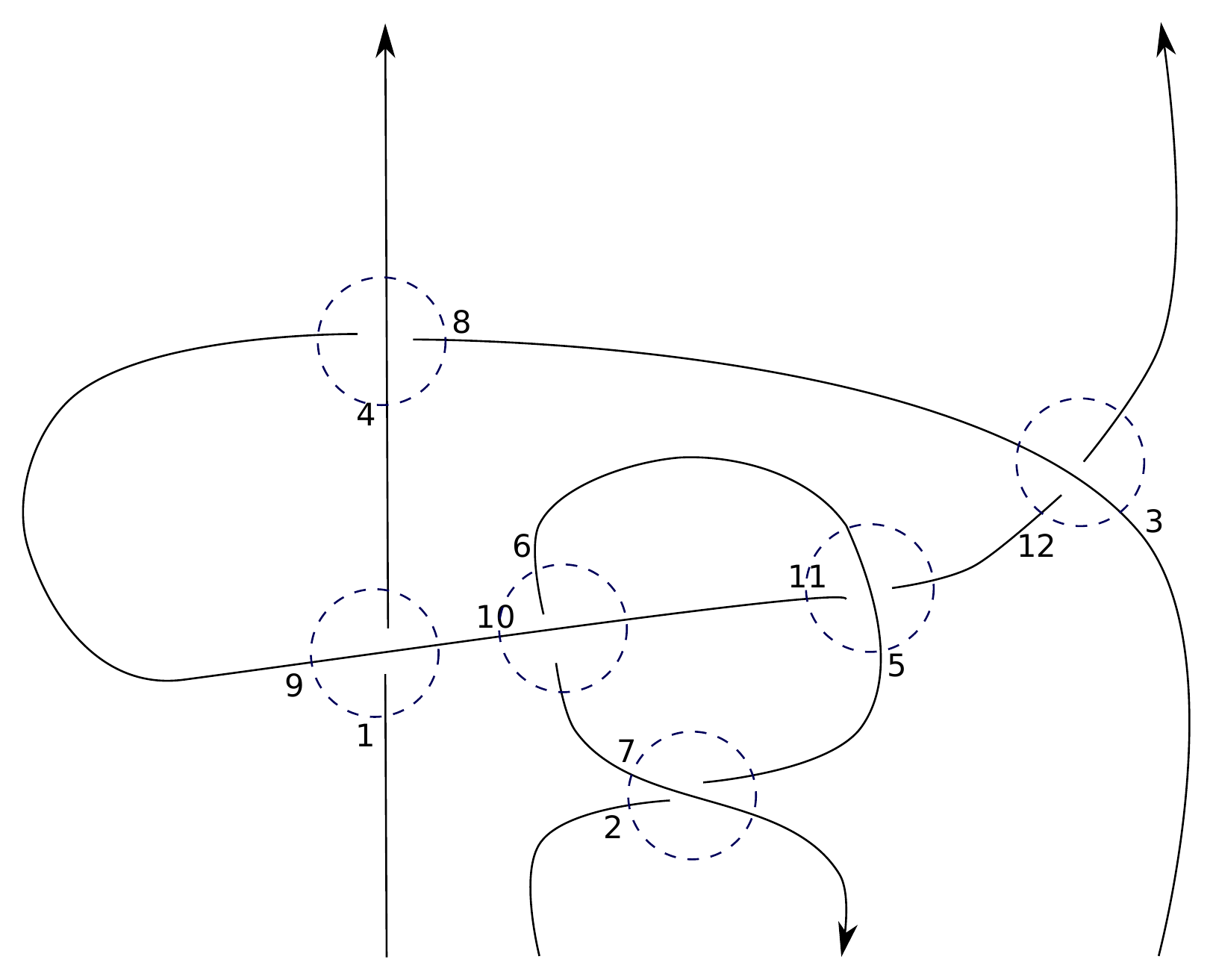}
       \captionof{figure}{A tangle.}  \label{fig:tangledemo} 
 \end{center}  
As a disjoint union of crossings, $T$ is given as follows
  \begin{multline*}
  R_{7,2}^+R_{10,6}^-R_{5,11}^-R_{3,12}^-R_{4,8}^+R_{9,1}^+\sslash m^{1,4}_1 \sslash m^{2,5}_2\sslash m^{2,6}_2\sslash m^{2,7}_2\sslash m^{3,8}_3\sslash m^{3,9}_3\sslash m^{3,10}_3 \sslash m^{3,11}_3\sslash m^{3,12}_3.
  \end{multline*}
Then its invariant in $\Gamma$-calculus is 
 \[\left(
\begin{array}{c|ccc}
  \left(\frac{t_2-1}{t_3}+1\right) \left(t_3-t_1 \left(t_3-1\right)\right)& x_1 & x_2 & x_3 \\ \hline
 y_1 & -\frac{t_3}{t_3 t_1-t_1-t_3} & \frac{\left(t_1-1\right) \left(t_3-1\right)
   t_3}{\left(t_2+t_3-1\right) \left(t_3 t_1-t_1-t_3\right)} & -\frac{\left(t_1-1\right)
   \left(t_3 t_2-t_2-2 t_3+1\right)}{\left(t_2+t_3-1\right) \left(t_3 t_1-t_1-t_3\right)}
   \\
 y_2 & 0 & \frac{t_2}{t_2+t_3-1} & \frac{t_2-1}{t_2+t_3-1} \\
 y_3 & \frac{t_1 \left(t_3-1\right)}{t_3 t_1-t_1-t_3} & -\frac{t_1
   \left(t_3-1\right)}{\left(t_2+t_3-1\right) \left(t_3 t_1-t_1-t_3\right)} & \frac{t_1
   t_3^2-t_3^2-3 t_1 t_3+t_1 t_2 t_3-t_2 t_3+2 t_3+t_1-t_1
   t_2+t_2-1}{\left(t_2+t_3-1\right) \left(t_3 t_1-t_1-t_3\right)} \\
\end{array}
\right).\]
If we multiply the matrix part with the scalar part then we get 
 \[
    \left(
\begin{array}{ccc}
 -1+t_2+t_3 & -1 + t_1+t_3-t_1t_3 & t_2 t_1-\frac{t_2 t_1}{t_3}+\frac{t_1}{t_3}-2 t_1-t_2+\frac{t_2}{t_3}-\frac{1}{t_3}+2 \\
 0 & -t_1 t_2+\frac{t_1 t_2}{t_3}+t_2 & -t_2 t_1+\frac{t_2 t_1}{t_3}-\frac{t_1}{t_3}+t_1+t_2-1 \\
 -t_2 t_1-t_3 t_1+\frac{t_2 t_1}{t_3}-\frac{t_1}{t_3}+2 t_1 & t_1-\frac{t_1}{t_3} & -t_2 t_1-t_3 t_1+\frac{t_2 t_1}{t_3}-\frac{t_1}{t_3}+3
   t_1+t_2+t_3-\frac{t_2}{t_3}+\frac{1}{t_3}-2
\end{array}
\right).
 \]
The fact that each entry is a Laurent polynomial suggests that it might be possible to categorify the invariant.      \hfill $\clubsuit$
\end{example}

\subsection{The Gassner Representation of String Links} In this section we restrict $\Gamma$-calculus to string links (compare with \cite{KLW01}). Given a positive integer $n$, fix $n$ points in the interior of the 2-disk $p_1,\dots,p_n$. A \emph{string link of $n$ components} is a smooth, proper, oriented 1-dimensional submanifold of $D^2\times [0,1]$ homeomorphic to the disjoint union of $n$ intervals such that the initial point of each interval coincides with some $p_i\times \{0\}$ and the endpoint coincides with $p_j\times \{1\}$. Two string links are \emph{isotopic} (the same) if there is a smooth family of string links interpolating between the two. In our setting the string links are \emph{colored}, i.e. each component is labeled with a natural number.
  \begin{center}
     \includegraphics[scale=0.5]{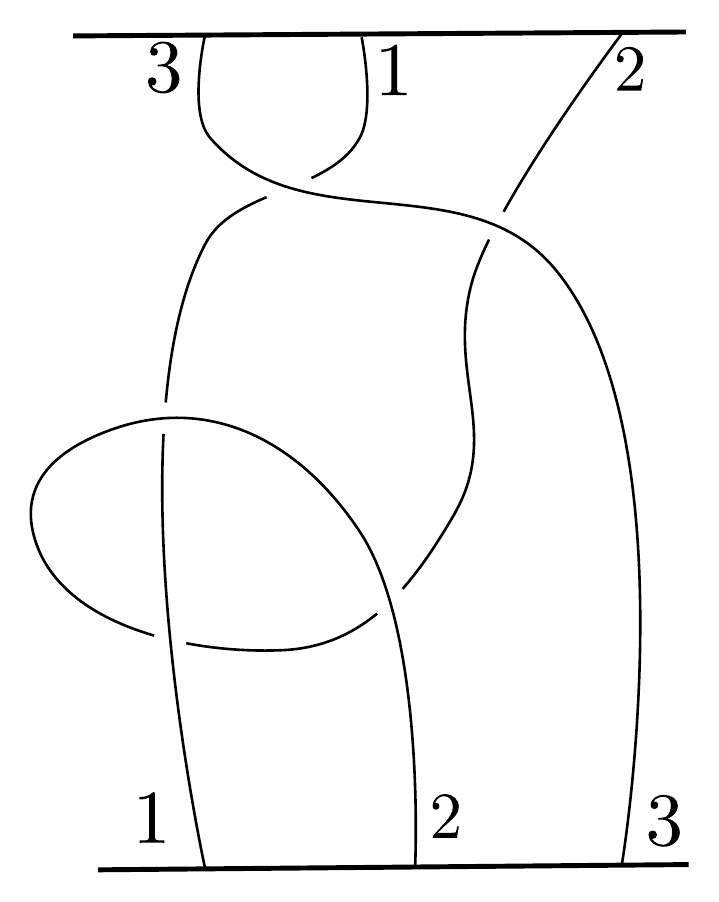}       
  \end{center}   
In the figure the orientation is such that the components run from the bottom to the top of the diagram. Suppose that the bottom endpoints of $\beta$ are labeled by $a_1,a_2,\dots,a_n$ and the top endpoints of $\beta$ are labeled by $b_1,b_2,\dots,b_n$ (where we read the endpoints from left to right). The labeling  yields a permutation $\rho$ given by 
 \[a_i\sslash \rho=b_i,\quad 1\leq i\leq n.\]
Note that here permutations act on the right. We call $\rho$ the \emph{permutation induced by $\beta$}. To simplify notation, we denote
  \[\vec{a}\sslash \rho=\vec{a}\rho=(b_1,b_2,\dots,b_n)=(a_1\rho,a_2\rho,\dots,a_n\rho).\]
For instance in the above figure the string link induces the permutation $(3,1,2)$. 

There is a \emph{composition} or \emph{multiplication} of string links $(\beta_1,\beta_2)\to \beta_1\cdot\beta_2$ obtained by staking $\beta_2$ on top of $\beta_1$. Note that in general we consider string links with distinct labels, so we can multiply any two string links with the same number of components, provided we identify the labels of the top endpoints of $\beta_1$ and the labels of the bottom endpoints of $\beta_2$. So for instance in the following 
  \begin{center}
     \includegraphics[scale=0.5]{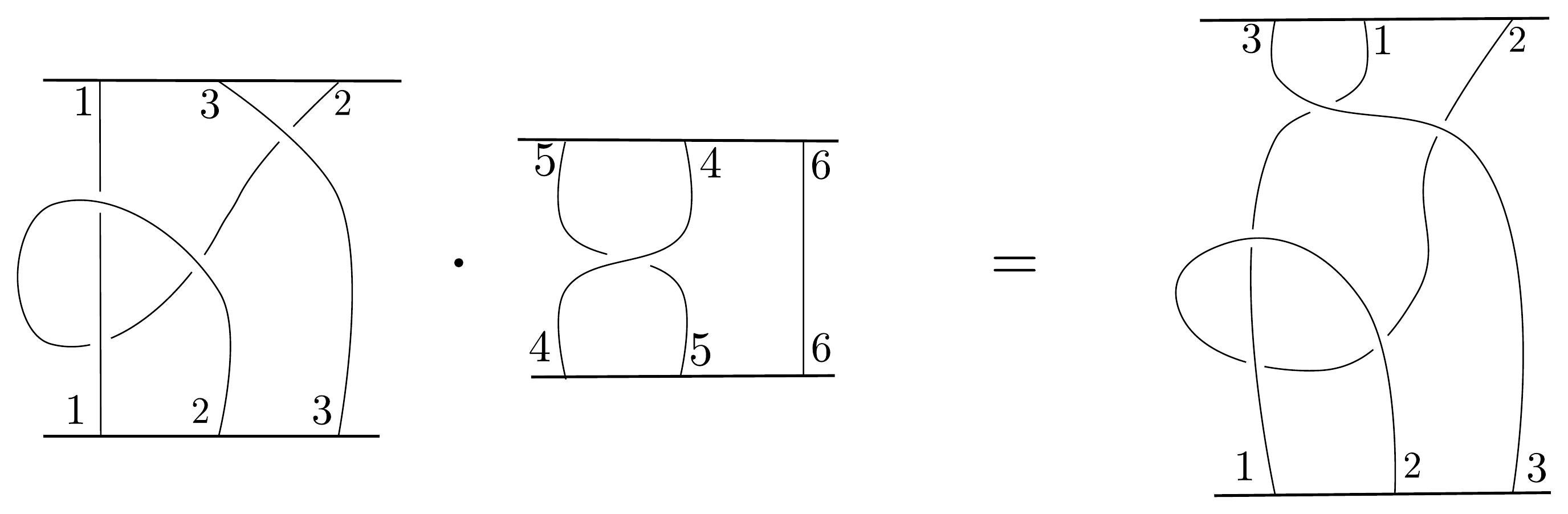}
  \end{center}  
we identify the label 4 with 1, 5 with 3, and 6 with 2. In terms of meta-monoids, the composition $\beta_1\cdot\beta_2$ can be described by the sequence of stitching 
  \[
    (\beta_1\beta_2) \sslash m^{1,4}_1\sslash m^{3,5}_3\sslash m^{2,6}_2. 
  \]
  
  Let us find out the permutation induced by $\beta_1\cdot\beta_2$. Suppose that the bottom endpoints of $\beta_1$ have labels $a_1,\dots,a_n$ and the top endpoints of $\beta_1$ have labels $a_1\rho_1,\dots,a_n\rho_1$; the bottom endpoints of $\beta_2$ have labels $b_1,\dots,b_n$ and the top endpoints of $\beta_2$ have labels $b_1\rho_2,\dots,b_n\rho_2$. In the composition $\beta_1\cdot\beta_2$ we relabel $b_i$ to $a_i\rho_1$. Therefore the labels of the top endpoints of $\beta_1\cdot\beta_2$ is $a_1\rho_1\rho_2,\dots, a_n\rho_1\rho_2$. In other words, the permutation induced by $\beta_1\cdot \beta_2$ is $\rho_1\rho_2$, where 
  \[
      a_i\sslash(\rho_1\rho_2)=(a_i\sslash \rho_1)\sslash \rho_2=a_i\rho_1\rho_2,\quad i=1,\dots,n.
   \] 
    
  Correspondingly, if $\varphi(\beta)$ is the image of $\beta$ in $\Gamma$-calculus, we can label the columns and rows of the matrix part of $\varphi(\beta)$ as follows 
 \begin{equation}\label{permutecolumn}
 \varphi(\beta)=\left(\begin{array}{c|ccc}
       \omega & x_{a_1} & \cdots & x_{a_n} \\
       \hline 
       y_{a_1}  \\
       \vdots & & \scalebox{2}{$M$}  \\
       y_{a_n}
 \end{array}\right)\xrightarrow[\text{according to $\rho$}]{\text{permute the columns}} \left(\begin{array}{c|ccc}
       \omega & x_{a_1\rho} & \cdots & x_{a_n\rho} \\
       \hline 
       y_{a_1}  \\
       \vdots & & \scalebox{2}{$M^{\rho}$}  \\
       y_{a_n}
 \end{array}\right).
 \end{equation}
In other words column $j$ of $M^{\rho}$ is column $a_j\rho$ of $M$. 

Now let $\beta_1$ and $\beta_2$ be two string links with $n$ components and suppose the bottom and top of $\beta_1$ are labeled by $\vec{a}=(a_1,\dots,a_n)$ and $\vec{a}\rho_1=(a_1\rho_1,\dots,a_n\rho_1)$, the bottom and top of $\beta_2$ are labeled by $\vec{b}=(b_1,\dots,b_n)$ and $\vec{b}\rho_2=(b_1\rho_2,\dots,b_n\rho_2)$, where $\rho_1$ and $\rho_2$ are the permutations induced by $\beta_1$ and $\beta_2$, respectively, and we also assume that $a_i\neq b_j$ for $1\leq i,j\leq n$. Assume that
  \[\varphi(\beta_1)=\left(\begin{array}{c|c}
      \omega_1 & x_{\vec{a}} \\
      \hline
      y_{\vec{a}} & M_1
    \end{array}  
  \right)\quad\text{ and }\quad \varphi(\beta_2)=\left(\begin{array}{c|c}
      \omega_2 & x_{\vec{b}} \\
      \hline
      y_{\vec{b}} & M_2
    \end{array}  
  \right),
  \]    
then we have the following result.
 
\begin{prop}\label{gammagassner}
  In $\Gamma$-calculus, the composition $\beta_1\cdot \beta_2$ is given by 
  \[\left(\begin{array}{c|c} 
       \omega_1\omega_2 & x_{\vec{a}\rho_1\rho_2} \\ \hline
       y_{\vec{a}} & M_1^{\rho_1}M_2^{\rho_2}
    \end{array}  
  \right)_{t_{\vec{b}}\to t_{\vec{a}\rho_1}}.\]
\end{prop}

\begin{proof}
  In the stitching language, the composition $\beta_1\cdot\beta_2$ is obtained by stitching the strands $a_i\rho_1$ to the strands $b_i$. By formula \eqref{stitchingbulkformula} we obtain 
   \[
      \left(\begin{array}{c|cc}
         \omega_1\omega_2 & x_{\vec{a}\rho_1} & x_{\vec{b}\rho_2} \\ \hline 
         y_{\vec{b}} & \vec{0} & M_2^{\rho_2} \\
         y_{\vec{a}} & M_1^{\rho_1} & \vec{0}
        \end{array}
      \right)\xrightarrow[]{m^{\vec{a}\rho_1,\vec{b}}} \left(\begin{array}{c|c}
      \omega_1\omega_2 & x_{\vec{b}\rho_2} \\ \hline
      y_{\vec{a}} & M_1^{\rho_1}M_2^{\rho_2}
      \end{array} 
    \right).
   \] 
Then identifying the labels $b_i$ with the labels $a_i\rho_1$ we obtain 
  \[ 
      \left(\begin{array}{c|c}
        \omega_1\omega_2 & x_{\vec{a}\rho_1\rho_2} \\ \hline
         y_{\vec{a}} & M_1^{\rho_1}M_2^{\rho_2}
        \end{array}
      \right)_{t_{\vec{b}\to \vec{a}\rho_1}},
  \]   
as required.   
\end{proof}

When $\beta$ is a (colored) braid, recall that its \emph{Gassner representation} (see \cite{BNT14}) is given by 
 \[
    R_{a,b}^+\mapsto \begin{pmatrix}
         1-t_a & 1 \\
         t_a & 0
    \end{pmatrix},\quad\quad R_{a,b}^-\mapsto \begin{pmatrix}
         0 & t_a^{-1} \\ 
         1 & 1-t_a^{-1}
    \end{pmatrix}
 \]
and extends by the identity matrix. For instance the following braid  
  \begin{center}
    \includegraphics[scale=0.4]{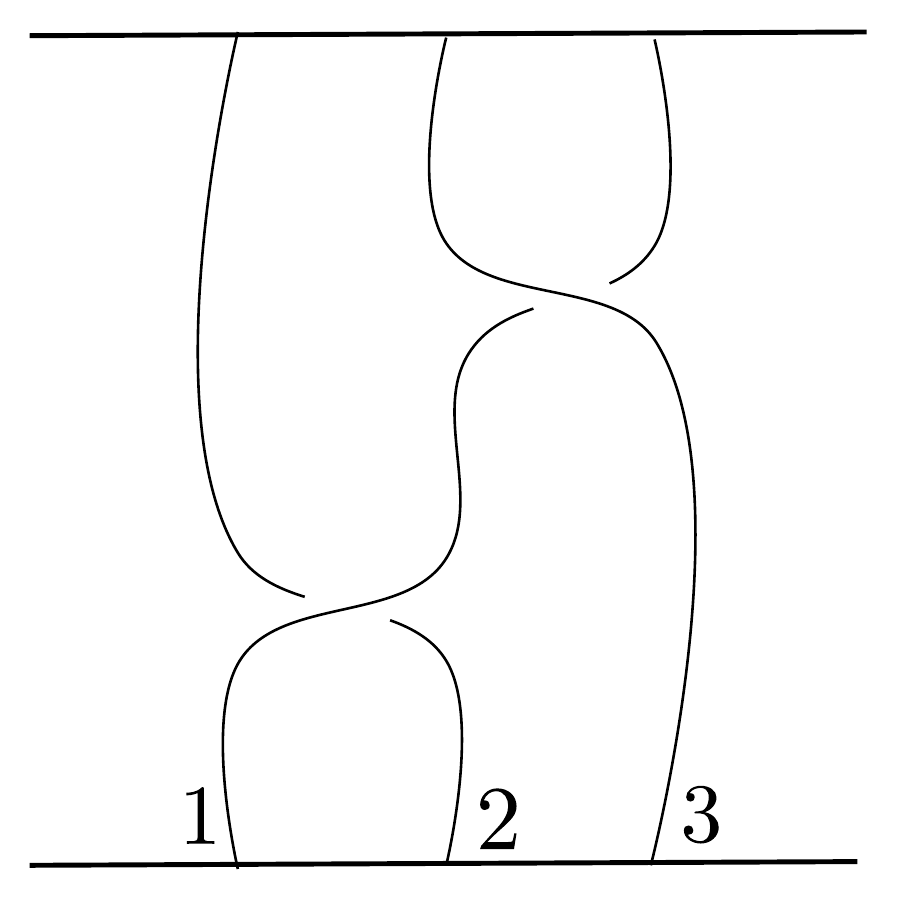}
   \end{center}
has the Gassner representation 
  \[\begin{pmatrix}
       1-t_1 & 1 & 0 \\
       t_1 & 0 & 0 \\
       0 & 0 & 1
    \end{pmatrix}\begin{pmatrix}
       1 & 0 & 0 \\
       0 & 0 & t_3^{-1} \\
       0 & 1 & 1-t_3^{-1}
    \end{pmatrix}=\begin{pmatrix}
       1-t_1 & 0 &t_3^{-1} \\
       t_1 & 0 & 0 \\
       0 & 1 & 1-t_3^{-1}
    \end{pmatrix},  
  \]
as required.  

\begin{prop}\label{gammabraid}
  Let $\beta\in B_n$ be an $n$-braid with induced permutation $\rho$. Suppose that 
    \[
       \varphi(\beta)=\left(\begin{array}{c|ccc}
       \omega & x_{a_1} & \cdots & x_{a_n} \\
       \hline 
       y_{a_1}  \\
       \vdots & & \scalebox{2}{$M$}  \\
       y_{a_n}
 \end{array}\right),
    \]
then $\omega=1$ and $M^{\rho}$ is the \emph{Gassner representation} of $\beta$.    
\end{prop}

\begin{proof}
  We first look at the standard generators of the braid groups $\sigma_i^{\pm 1}$, $1\leq i\leq n-1$. Notice that the permutation induced by each generator is a transposition. Ignoring the identity part, we have 
 \[\varphi(R_{a,b}^{+})=\left(\begin{array}{c|cc}
     1 & x_a & x_b \\
     \hline
     y_a & 1 & 1-t_a\\
     y_b & 0 & t_a
 \end{array}\right)\xrightarrow[\text{according to the permutation}]{\text{permute the columns}} \left(\begin{array}{c|cc}
     1 & x_b & x_a \\
     \hline
     y_a & 1-t_a & 1\\
     y_b & t_a & 0
 \end{array}\right),
\]
and 
 \[\varphi(R_{a,b}^{-})=\left(\begin{array}{c|cc}
     1 & x_b & x_a \\
     \hline
     y_b & t_a^{-1} & 0\\
     y_a & 1-t_a^{-1} & 1
 \end{array}\right)\xrightarrow[\text{according to the permutation}]{\text{permute the columns}} \left(\begin{array}{c|cc}
     1 & x_a & x_b \\
     \hline
     y_b & 0 & t_a^{-1}\\
     y_a & 1 & 1-t_a^{-1}  
 \end{array}\right).
\]
We see that the right hand sides are exactly the Gassner representation. From Proposition \ref{gammagassner}, compositions of braids correspond to products of matrices. Thus $M^{\rho}$ agrees with the Gassner representation of $\beta$. Furthermore, since the scalar part of each generator is 1, the scalar part of $\beta$ is still 1.
\end{proof}

\subsection{The Alexander Polynomial} In this section we relate $\Gamma$-calculus and the Alexander polynomial. First it is well-known that for a (usual) knot $K$, the operation of cutting $K$ open is well-defined, i.e. the isotopy class of $K$ as a long knot does not depend on where we cut $K$. The same result holds for links, provided we cut a fixed component. For a ``cute'' explanation of this fact the readers can refer to \cite{openknotoverflow}. Note that the proof will not work if we allow virtual crossings.      




\begin{prop}\label{omegaalexander}
   Let $K$ be a long knot and 
     \[\varphi(K)=\left(\begin{array}{c|c}
          \omega & x_1 \\ \hline
          y_1 & 1
        \end{array}
     \right).\]
   Then $\omega\doteq \Delta_{\widetilde{K}}(t)$. Here $\Delta_{\widetilde{K}}(t)$ is the Alexander polynomial (see \cite{Mur99}) of $\widetilde{K}$, where $\widetilde{K}$ is the closed knot obtained by closing the open component of $K$ trivially and $\doteq$ means equality up to multiplication by $\pm t^{n}$, $n\in\Z$.  
\end{prop} 

\begin{proof}
  By Alexander's Theorem (see \cite{KT08}) $\widetilde{K}$ is the closure of a braid $\beta$. Then the Alexander polynomial of $\widetilde{K}$ (see \cite{Mur99}) is given by 
   \[\Delta_{\widetilde{K}}(t)\doteq\mathrm{det}([I-f(\beta)]_1^1).\]
 Here $f(\beta)$ denotes the Burau representation of $\beta$, i.e. the Gassner representation when we set all the variables to $t$ and $[A]_i^j$ means the matrix obtained from $A$ by removing the $i$th row and the $j$th column. From Proposition \ref{gammabraid} we know that $f(\beta)$ agrees with (a permutation of) the matrix part of $\varphi(\beta)$. Now if we take the closure of $\beta$ except for the first strand, then we obtain a long knot $K_1$. Proposition \ref{stitchingbulk} says that the scalar part of $K_1$ is 
   \[\mathrm{det}([I-f(\beta)]_1^1).\]    
To finish off, we observe that $K_1$ is isotopic to $K$ because they are the results of cutting $\widetilde{K}$ at two different places. Therefore the scalar parts of $K_1$ and $K$ must agree since they are both invariants. In other words, 
  \[\omega\doteq \Delta_{\widetilde{K}}(t),\]
as required.      
\end{proof}

Thus we see that $\Gamma$-calculus gives us an extension of the Alexander polynomial to w-tangles, which include usual tangles. In the case of one component, we obtain an invariant of long w-knots, which contains the Alexander polynomials of usual knots. (Note that our theory yields a trivial invariant for closed w-knots.) We can compute the Alexander polynomial by taking the closure of an arbitrary tangle (not necessarily a braid). For instance,
consider the long knot $7_7$ in the \href{http://katlas.org/wiki/7_7}{Knot Atlas} obtained as the closure of a tangle 
  \begin{center}
    \includegraphics[scale=0.4]{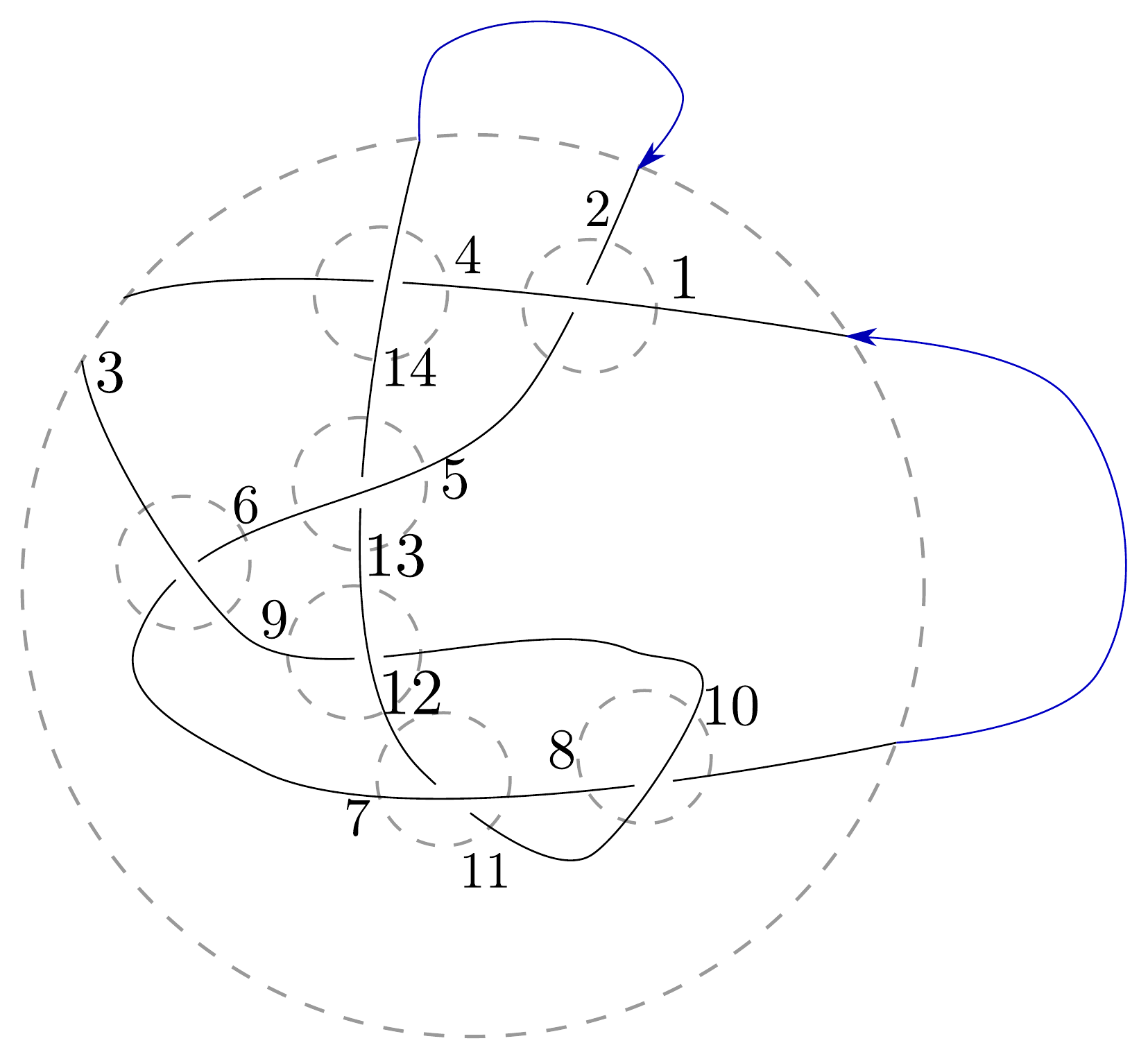}     
  \end{center}
In terms of meta-monoids the tangle is given by
  \begin{multline*}
     R_{1,2}^+R_{14,4}^+R_{5,13}^-R_{3,6}^-R_{12,9}^-R_{7,11}^+R_{10,8}^+\sslash m^{1,4}_1\sslash m^{2,5}_2\sslash m^{2,6}_2\sslash m^{2,7}_2\sslash m^{2,8}_2\sslash m^{3,9}_3\sslash m^{3,10}_3\sslash m^{3,11}_3\sslash m^{3,12}_3\sslash m^{3,13}_3\sslash m^{3,14}_3.
  \end{multline*} 
Suppose that its invariant in $\Gamma$-calculus has the form 
   \[
      \left(\begin{array}{c|ccc}
        \omega & x_1 & x_2 & x_3 \\ \hline 
        y_1 & \alpha_{11} & \alpha_{12} & \alpha_{13} \\
        y_2 & \alpha_{21} & \alpha_{22} & \alpha_{23} \\
        y_3 & \alpha_{31} & \alpha_{32} & \alpha_{33}
      \end{array}
      \right).
   \]
Then by stitching strand 2 to strand 1 and strand 3 to strand 2 the invariant of the long knot is given by 
  \[
    \left.\omega\det\left(
      I-\begin{pmatrix}
         \alpha_{12} & \alpha_{13} \\
         \alpha_{22} & \alpha_{23}
      \end{pmatrix}
    \right)\right|_{t_2,t_3\to t}.
  \]
Doing the calculation one obtain  
  \[
     t^{-2}-5t^{-1}+9-5t+t^2,
  \]
which one can check to be the Alexander polynomial of the knot.  
  
\subsection{Orientation Reversal} For subsequent sections, it is useful to have a formula to reverse the orientation of a strand of a w-tangle in $\Gamma$-calculus. 
    \begin{center}
      \includegraphics[scale=0.4]{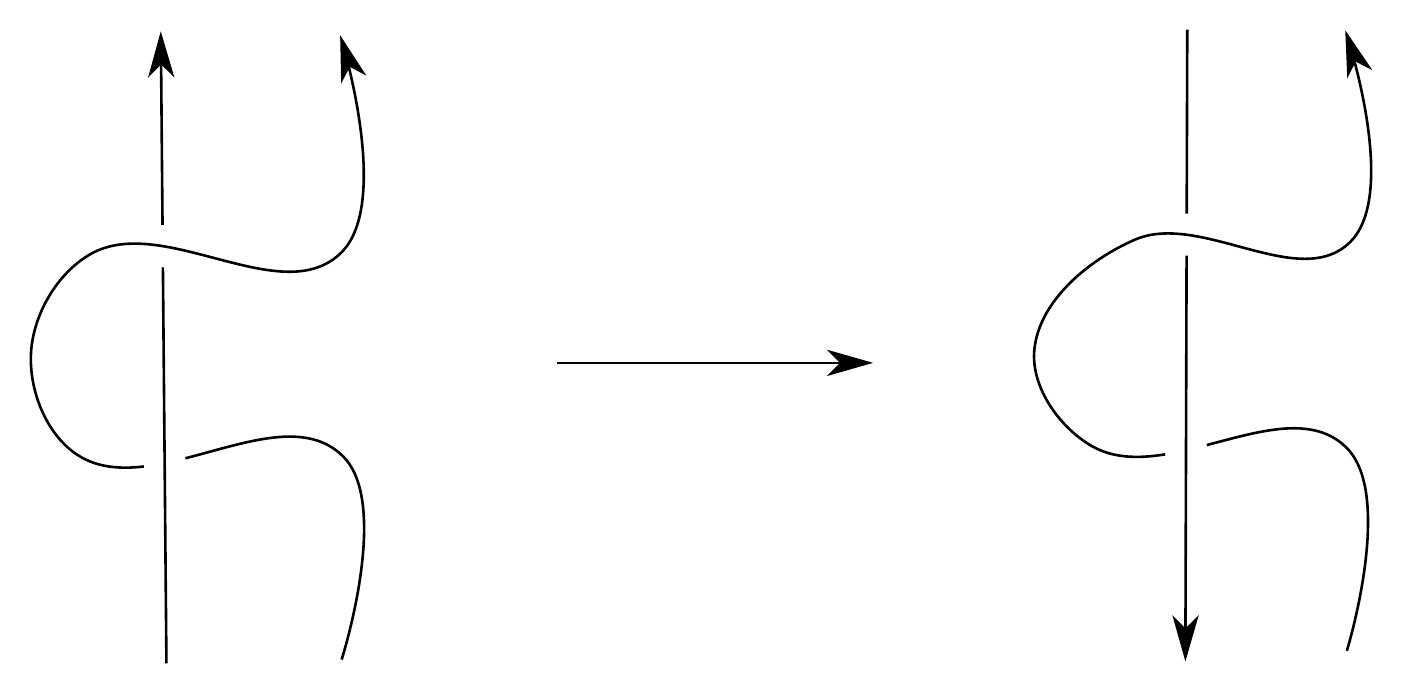}
    \end{center}   
To proceed, let us introduce another meta-monoid, called \emph{$\sigma$-calculus}, defined as follows. For a finite set $X$, let $\sigma^X$ be the collection of tuples of the form $(\sigma_x:x\in X)$, where $\sigma_x$ is a monomial in the variables $t_z$, $z\in X$. Let us record the operations below: 
\begin{itemize}
   \item[] identity $\quad(\sigma_x)\sslash e_a=(\sigma_x,\sigma_a=1)$,
   \item[] disjoint union $\quad(\sigma_x)\sqcup (\sigma_y)=(\sigma_x,\sigma_y)$, 
    \item[] deletion $(\sigma_x,\sigma_a)\sslash \eta_a=(\sigma_x)_{t_a\to 1}
    $, 
  \item[] renaming $(\sigma_a,\sigma_x)\sslash \sigma^a_b=(\sigma_b=\sigma_a,\sigma_x)_{t_a\to t_b}
   $, 
 \item[] stitching $(\sigma_a,\sigma_b,\sigma_x)\sslash m^{a,b}_c=
   (\sigma_c=\sigma_a\sigma_b,\sigma_x)_{t_a,t_b\to t_c}.$   
 \end{itemize}
It is easy to see that these operations satisfy the meta-monoid axioms. There is a meta-monoid homomorphism from w-tangles to $\sigma$-calculus, namely 
  \[R_{a,b}^{\pm }\mapsto \{\sigma_a=1,\sigma_b=t_a^{\pm 1}\}.\] 

One checks readily that the Reidemeister relations are satisfied. So we obtain a w-tangle invariant. Given a w-tangle, then $\sigma_a$ of the strand labeled $a$ is given by 
  \[\prod t_b^{\pm 1},\] 
where the product is over all crossings such that $a$ is the understrand and $b$ is the overstrand (including $a$ itself) and $\pm 1$ is the sign of the crossing. For example, the tangle given in Figure \ref{fig:tangledemo} has value  
 \[
    \sigma=(\sigma_1=t_3,\sigma_2=t_2t_3^{-1},\sigma_3=t_1t_2^{-1}t_3^{-1}).
  \]
With that we are ready to define the orientation reversal operation. Let $dS^a$ denote the operation of \emph{reversing the orientation} of the strand labeled $a$ of a w-tangle $T$. Note that $dS^a$ takes as input $(\varphi(T),\sigma_T)$. Although if we allow the scalar part to be determined up to a multiplication by $\pm \prod t_b^{\pm 1}$, then we can ignore $\sigma_T$.   
\begin{prop}   
The operation $dS^a$ is given by 
  \[\left(\varphi(T)=
     \left(
      \begin{array}{c|cc}
         \omega & x_a & x_S \\
         \hline
         y_a & \alpha &\theta\\
         y_S &\phi &\Xi 
       \end{array}          
    \right),\sigma\right) \xrightarrow{dS^a} 
   \left(\left( \begin{array}{c|cc}
        \alpha\omega/\sigma_a & x_a & x_S \\
        \hline 
        y_a & 1/\alpha & \theta/\alpha \\
        y_S & -\phi/\alpha & (\alpha\Xi-\phi\theta)/\alpha
    \end{array}\right)_{t_a\to t_a^{-1}},\left(\sigma_a^{-1},\sigma_S|_{t_a\to t_a^{-1}}\right)\right),
  \] 
where by $\sigma_S$ we mean $(\sigma_s)_{s\in S}$.  
\end{prop}

\begin{proof}
  First of all note that since $\left.\alpha\right|_{t_x\to 1}=1$, it makes sense to divide by $\alpha$. We first start with the positive crossing $R_{a,b}^{+}$. In this case we have $\sigma=(\sigma_a=1,\sigma_b=t_a)$. Note that 
  \[R_{a,b}^{+}\sslash dS^b\sslash \varphi=R_{a,b}^-\sslash \varphi=\left(\begin{array}{c|cc}
    1 & x_a & x_b \\ \hline
    y_a & 1 & 1-t_a^{1}\\
    y_b & 0 & t_a^{-1}
  \end{array}
  \right).\]    
On the other hand, 
  \[R_{a,b}^{+}\sslash \varphi\sslash dS^b=\left(
      \begin{array}{c|cc}
        1 & x_b & x_a\\ \hline
       y_b & t_a & 0\\
        y_a & 1-t_a & 1
      \end{array}
  \right) \sslash dS^b=\left(\begin{array}{c|cc}
       t_a/\sigma_b & x_b & x_a \\ \hline
       y_b & t_a^{-1} & 0\\
       y_a & 1-t_a^{-1} & 1
    \end{array}
  \right)_{t_b\to t_b^{-1}}=R_{a,b}^-\sslash \varphi,\]
as required. Now for the negative crossing $R_{a,b}^{-}$, we have $\sigma=(\sigma_a=1,\sigma_b=t_a^{-1})$ and
 \[R_{a,b}^{-}\sslash dS^b\sslash \varphi=R_{a,b}^+\sslash \varphi=\left(\begin{array}{c|cc}
    1 & x_a & x_b \\ \hline
    y_a & 1 & 1-t_a\\
    y_b & 0 & t_a
  \end{array}
  \right),\]      
whereas 
 \[R_{a,b}^{-}\sslash \varphi\sslash dS^b=\left(
      \begin{array}{c|cc}
        1 & x_b & x_a\\ \hline
        y_b & t_a^{-1} & 0\\
        y_a & 1-t_a^{-1} & 1
      \end{array}
  \right) \sslash dS^b=\left(\begin{array}{c|cc}
       t_a^{-1}/\sigma_b & x_b & x_a \\ \hline
       y_b & t_a & 0\\
       y_a & 1-t_a & 1
    \end{array}
  \right)_{t_b\to t_b^{-1}}=R_{a,b}^+\sslash \varphi.\] 
We leave the case $dS^a$ to the readers. For a general tangle, to reverse the orientation of a strand, our strategy is to break the tangle into a disjoint union of crossings, reverse the orientation of the crossings that contain the strand, and then stitch them again. Thus we need to show for a tangle $T$ that 
 \begin{equation}\label{stitchingreversal}
 (\varphi(T),\sigma_T)\sslash m^{a,b}_c\sslash dS^{c}=(\varphi(T),\sigma_T)\sslash dS^a\sslash dS^b\sslash m^{b,a}_c.
 \end{equation}
 Let 
    \[
    \varphi(T)=\left(
      \begin{array}{c|ccc}
         \omega & x_a & x_b & x_S \\
         \hline
         y_a & \alpha & \beta &\theta\\
         y_b & \gamma &\delta &\epsilon \\
         y_S &\phi &\psi &\Xi 
       \end{array}          
    \right),\quad \sigma_T=(\sigma_a,\sigma_b,\sigma_S).
    \] 
We leave it as an exercise to check that both sides of the above equality are equal to 
   \[
     \left(
\begin{array}{c|cc}
 \frac{(-\gamma  \beta +\beta +\alpha  \delta ) \omega}{\sigma_a\sigma_b}  & x_c & x_S \\ \hline
 y_c & \frac{\gamma -1}{\gamma  \beta -\beta -\alpha  \delta } & \frac{-\alpha  \epsilon
   +\gamma  \theta -\theta }{\gamma  \beta -\beta -\alpha  \delta } \\
 y_S & -\frac{-\delta  \phi +\gamma  \psi -\psi }{\gamma  \beta -\beta -\alpha  \delta } &
   \frac{-\beta  \Xi +\beta  \gamma  \Xi -\alpha  \delta  \Xi -\beta  \phi \epsilon
   +\delta   \phi \theta + \alpha \psi \epsilon  -\gamma   \psi\theta +\psi\theta
   }{\gamma  \beta -\beta -\alpha  \delta } \\
\end{array}
\right)_{t_a,t_b\to t_c^{-1}},
   \]   
 and the resulting $\sigma$ is $(\sigma_a^{-1}\sigma_b^{-1},\sigma_S|_{t_a,t_b\to t_c^{-1}})$, as required. 
\end{proof}

Again it is useful to have a formula to reverse the orientations of many strands at the same time. We record it in the next proposition

\begin{prop}\label{reverseinbulk}
  Let $T$ be a w-tangle. The operation of reversing the orientations of the strands labeled by $\vec{a}=(a_1,a_2,\dots,a_n)$ is given by 
   \[\left(\varphi(T)=
     \left(
      \begin{array}{c|cc}
         \omega & x_\vec{a} & x_S \\
         \hline
         y_\vec{a} & \alpha &\theta\\
         y_S &\phi &\Xi 
       \end{array}          
    \right),\sigma_T\right) \xrightarrow{dS^{\vec{a}}} 
   \left(\left( \begin{array}{c|cc}
        \frac{\omega\det(\alpha)}{\prod\sigma_{\vec{a}}} & x_\vec{a} & x_S \\
        \hline 
        y_\vec{a} & \alpha^{-1} & \alpha^{-1}\theta \\
        y_S & -\phi/\alpha^{-1} & \Xi-\phi\alpha^{-1}\theta
    \end{array}\right)_{t_\vec{a}\to t_\vec{a}^{-1}},\left(\sigma_{\vec{a}}^{-1},\sigma_S|_{t_{\vec{a}}\to t_{\vec{a}}^{-1}}\right)\right),
  \] 
where $\sigma_{\vec{a}}=(\sigma_{a_1},\cdots,\sigma_{a_n})$.  
\end{prop}

\begin{proof}
  We proceed by induction on $n$. The case when $n=1$ is precisely $dS^a$. Now for the induction step, we write $\vec{a}=(\vec{a}',a)$ and 
   \[
     \varphi(T)=\left(
        \begin{array}{c|ccc}
          \omega & x_\vec{a}' & x_a & x_S\\ \hline
          y_\vec{a}' & \alpha_1 & \alpha_2 &\theta_1\\
          y_a & \alpha_3 & \alpha_4 & \theta_2\\
          y_S & \phi_1 &\phi_2 & \Xi 
        \end{array}
     \right).
   \]
Then reversing the orientation of strands $\vec{a}'$, using the induction hypothesis, we obtain 
  \[\left(\begin{array}{c|ccc}
     \omega\det(\alpha_1)/\prod\sigma_{\vec{a}'}   &x_{\vec{a}'} & x_a & x_S \\ \hline
        y_{\vec{a}'} & \alpha_1^{-1} & \alpha_1^{-1}\alpha_2 & \alpha_1^{-1}\theta_1 \\
        y_a & -\alpha_3\alpha_1^{-1} & \alpha_4-\alpha_3\alpha_1^{-1}\alpha_2 & \theta_2-\alpha_3\alpha_1^{-1}\theta_1 \\
        y_S & -\phi_1\alpha_1^{-1} &\phi_2-\phi_1\alpha_1^{-1}\alpha_2 & \Xi-\phi_1\alpha_1^{-1}\theta_1
      \end{array}
     \right)_{t_{\vec{a}'}\to t_{\vec{a}'}^{-1}}.
  \] 
Now we reverse the orientation of strand $a$ to get 
 \[\left(\begin{array}{c|ccc}
    \widetilde{\omega} & x_a & x_{\vec{a}'} & x_S \\ \hline
      y_a & \frac{1}{\alpha_4-\alpha_3\alpha_1^{-1}\alpha_2} & -\frac{\alpha_3\alpha_1^{-1}}{\alpha_4-\alpha_3\alpha_1^{-1}\alpha_2} &\frac{\theta_2-\alpha_3\alpha_1^{-1}\theta_1}{\alpha_4-\alpha_3\alpha_1^{-1}\alpha_2} \\
      y_{\vec{a}'} & -\frac{\alpha_1^{-1}\alpha_2}{\alpha_4-\alpha_3\alpha_1^{-1}\alpha_2} & \alpha_1^{-1}+\frac{\alpha_1^{-1}\alpha_2\alpha_3\alpha_1^{-1}}{\alpha_4-\alpha_3\alpha_1^{-1}\alpha_2} & \alpha_1^{-1}\theta_1-\frac{\alpha_1^{-1}\alpha_2(\theta_2-\alpha_3\alpha_1^{-1}\theta_1)}{\alpha_4-\alpha_3\alpha_1^{-1}\alpha_2} \\
      y_S & \frac{-\phi_2+\phi_1\alpha_1^{-1}\alpha_2}{\alpha_4-\alpha_3\alpha_1^{-1}\alpha_2} & -\phi_1\alpha_1^{-1}+\frac{(\phi_2-\phi_1\alpha_1^{-1}\alpha_2)\alpha_3\alpha_1}{\alpha_4-\alpha_3\alpha_1^{-1}\alpha_2} & \Xi-\phi_1\alpha_1^{-1}\theta_1-\frac{(\phi_2-\alpha_1^{-1}\alpha_2\phi_1)(\theta_2-\alpha_3\alpha_1^{-1}\theta_1)}{\alpha_4-\alpha_3\alpha_1^{-1}\alpha_2}
    \end{array}
   \right)_{t_{\vec{a}}\to t_{\vec{a}}^{-1}},
 \]    
where 
 \[
   \widetilde{\omega}=\left. \frac{\omega(\alpha_4-\alpha_3\alpha_1^{-1}\alpha_2)\det(\alpha_1)}{\prod \sigma_{\vec{a}}}\right|_{t_{\vec{a}}\to t_{\vec{a}}^{-1}}.
 \] 
 Again by Lemma \ref{blockdeterminant} we have 
   \[
     \det\begin{pmatrix}
       \alpha_1 & \alpha_2 \\
       \alpha_3 & \alpha_4
     \end{pmatrix}=\det\begin{pmatrix}
        \alpha_4 & \alpha_3\\
        \alpha_2 & \alpha_1
     \end{pmatrix}=\det(\alpha_4-\alpha_3\alpha_1^{-1}\alpha_2)\det(\alpha_1).
   \] 
 To finish off, we just need to show that 
  \[\begin{pmatrix}
      \alpha_1 & \alpha_2 \\
      \alpha_3 & \alpha_4
    \end{pmatrix}^{-1}=\begin{pmatrix}
      \alpha_1^{-1}+\frac{\alpha_1^{-1}\alpha_2\alpha_3\alpha_1^{-1}}{\alpha_4-\alpha_3\alpha_1^{-1}\alpha_2} & -\frac{\alpha_1^{-1}\alpha_2}{\alpha_4-\alpha_3\alpha_1^{-1}\alpha_2} \\
      -\frac{\alpha_3\alpha_1^{-1}}{\alpha_4-\alpha_3\alpha_1^{-1}\alpha_2} & \frac{1}{\alpha_4-\alpha_3\alpha_1^{-1}\alpha_2}
    \end{pmatrix},
  \]  
which one can easily check by performing matrix multiplications. The formula for $\sigma$ is straightforward to verify.   
\end{proof}

\section{The Fox-Milnor Condition}\label{sec:foxmilnor}

\subsection{Ribon Knots} We first recall some basic terminologies and refer the readers to \cite{Kau87} for more details. A knot is called \emph{ribbon} if it can be written as the boundary of a 2-disk that is immersed into $S^3$ with \emph{ribbon singularities}. More precisely, if $\iota: D^2 \to S^3$ is the immersion and $C$ is a connected component of the singular set of $\iota$, then $\iota^{-1}(C)$ consists of a pair of closed intervals: one lies entirely in the interior of $D^2$ and one with endpoints on the boundary of $D^2$ as in the following figure.  
  \begin{center}    
     \includegraphics[scale=0.45]{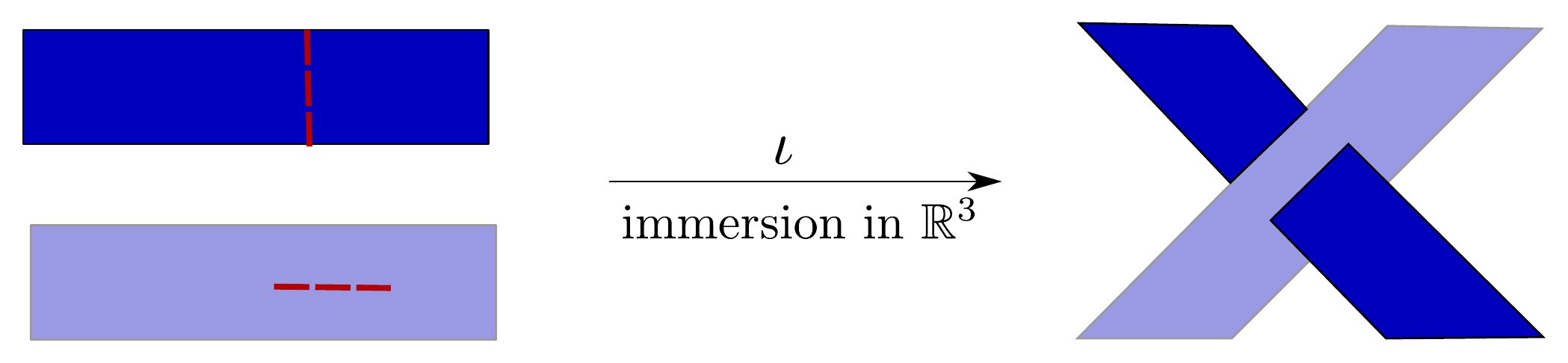}
  \end{center}
Here the dashed lines indicate the preimages of the singularity. For instance the following knot is ribbon. One sees that it can be written as the boundary of a 2-disk (the shaded part) with only ribbon singularities. 
   \begin{center}
      \includegraphics[scale=0.2]{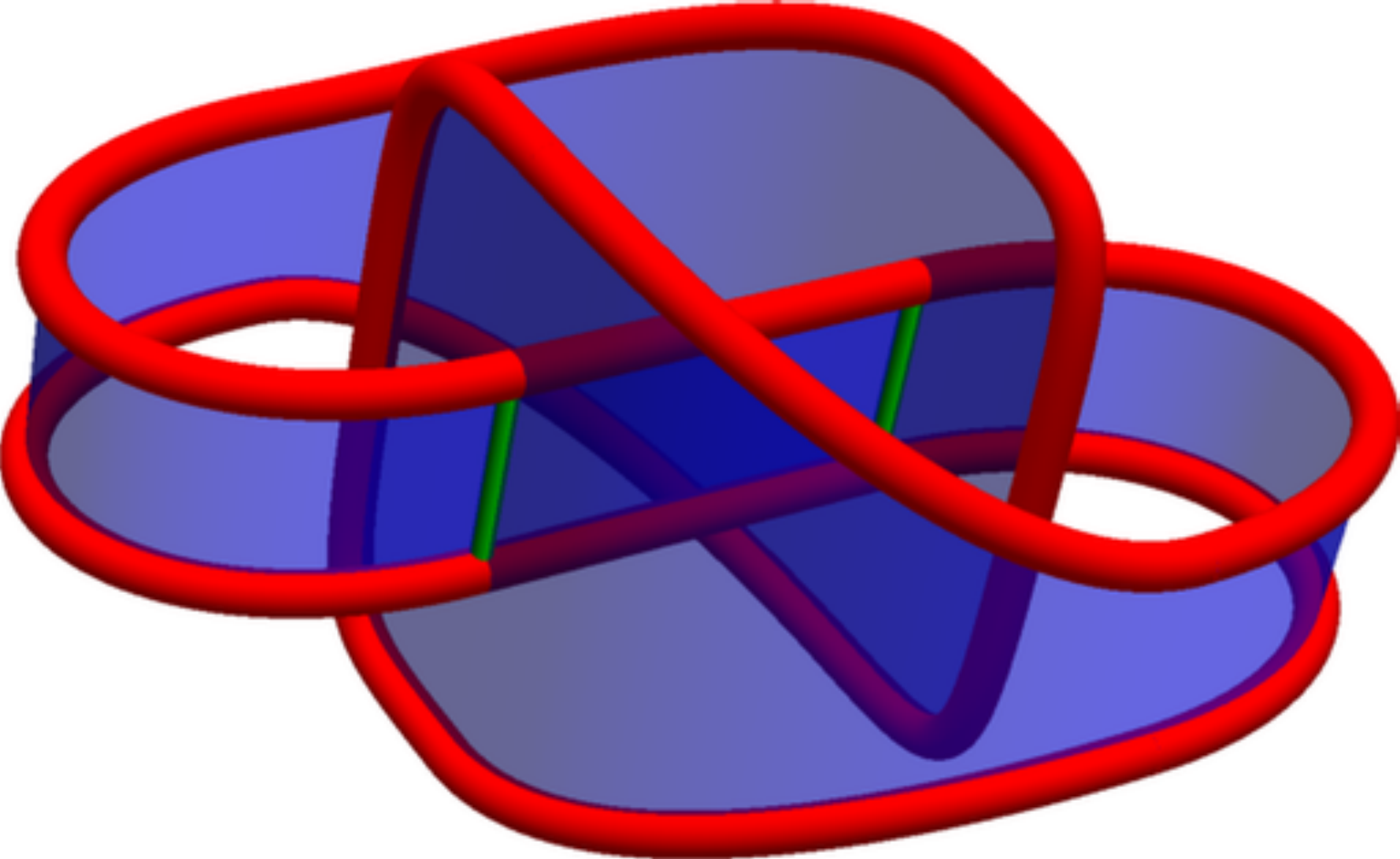}
   \end{center}    
A knot is called \emph{(smoothly) slice} if it is the boundary of a smoothly embedded 2-disk $D^2$ in the 4-dimensional disk $D^4$. (Here the boundary of $D^4$ is the 3-sphere $S^3$, which contains our knot.) It is clear that ribbon knots are slice because we can push the (ribbon) singularities into $D^4$, thereby obtaining an embedding of $D^2$. However the reverse direction, known as the \emph{slice-ribbon conjecture}, is one of the most famous open problems in classical knot theory. Our goal in this section is to prove the \emph{Fox-Milnor condition} using the framework of $\Gamma$-calculus. 

\begin{thm*}[Fox-Milnor \cite{Lic97}]
  If a knot $K$ is slice, and $\Delta_K(t)$ is the Alexander polynomial of $K$, then there exists a Laurent polynomial $f$ such that 
  \[\Delta_K(t)\doteq f(t)f(t^{-1}),\]
 where $\doteq$ means equality up to multiplication by $\pm t^n$, $n\in \Z$.  
\end{thm*}

Notice that the Fox-Milnor condition gives us a condition on slice knots, and since the class of slice knots contains ribbon knots, it cannot help resolve the slice-ribbon conjecture. The consensus is that it should be false, and we do have several potential counter-examples with a high number of crossings \cite{GSA10}. Below we are describing a characterization of ribbon knots in the language of meta-monoids. Therefore to help tackle the slice-ribbon conjecture we need an invariant that is polynomial-time computable, so that it can handle knots with a large number of crossings, and also behaves well with respect to the meta-monoid operations. We argue that $\Gamma$-calculus is one example of such an invariant (in fact the simplest of a series of invariants). In the remaining part of the paper we will investigate the ribbon property in $\Gamma$-calculus. Although in the end we just obtain the Fox-Milnor condition, our proof only uses the characterization for ribbon knots (as opposed to slice knots), thus it has the potential to answer the slice-ribbon conjecture when we generalize it in the context of a stronger invariant, which we are currently developing \cite{BN16}.    

Ribbon knots have the following characterization in terms of tangles. Consider a $2n$-component \emph{pure up-down} tangle. Here \emph{pure} means the permutation induced by the tangle is the identity permutation and \emph{up-down} means that the strands are oriented up and down alternately starting from the first strand, where we label the strands from left to right from 1 to $2n$. There are two special closure operations called \emph{knot closure} and \emph{tangle closure}, denoted by $\kappa$ and $\tau$, respectively. The $\tau$ closure connects strand $i$ to strand $i+1$, where $i$ runs over all odd labels $1,3,\dots,2n-1$, which yields an $n$-component tangle. The $\kappa$ closure connects strand $i+1$ to strand $i$, where $i$ runs over the labels $1,2,\dots,2n-1$, which yields a long knot.
  \begin{center}
  \includegraphics[scale=0.4]{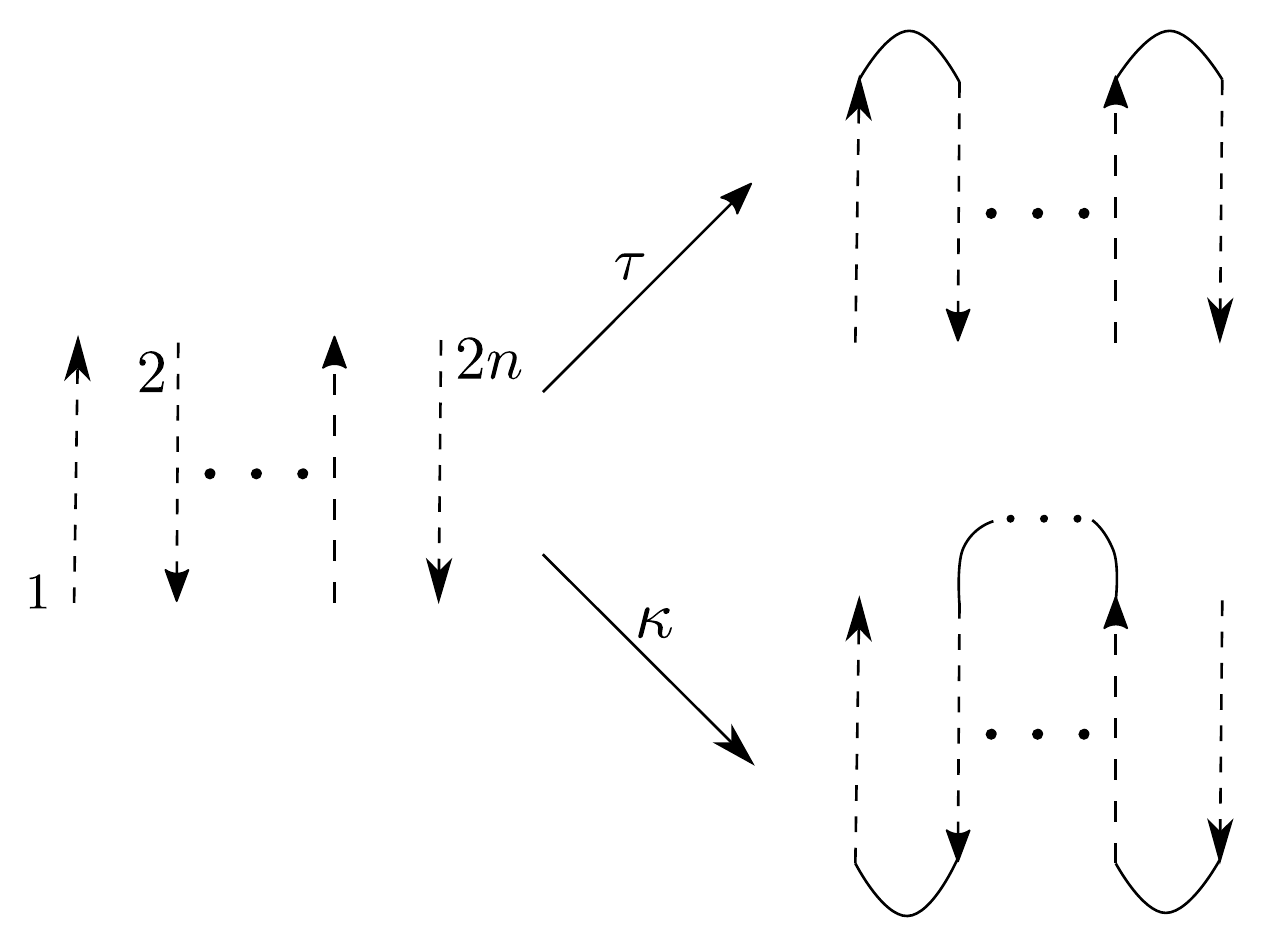}
  \end{center}  

\begin{prop}\label{ribbon}
 A long knot $K$ is ribbon if and only if there exists a $2n$-component pure up-down tangle $T$ such that $\kappa(T)$ is the knot $K$ and $\tau(T)$ is the trivial $n$-component tangle, i.e. it bounds $n$ disjoint embedded half-disks in $\R^2$ as in the following figure. 
    \begin{center}
    \includegraphics[scale=0.4]{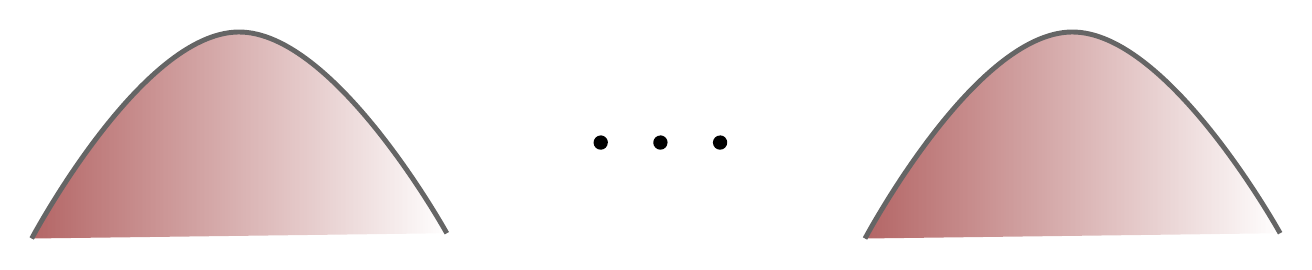}
    \end{center}  
\end{prop}

\begin{proof}
Let us sketch a proof of the proposition, (see also \cite{Khe17}). For the only if direction, note that a ribbon knot can be presented in a special form, known as a \emph{ribbon presentation} (see \cite{Kaw96}). Namely, every ribbon knot can be obtained from an embedding of a disjoint union of rings and strings between consecutive rings 
  \begin{center}
    \includegraphics[scale=0.35]{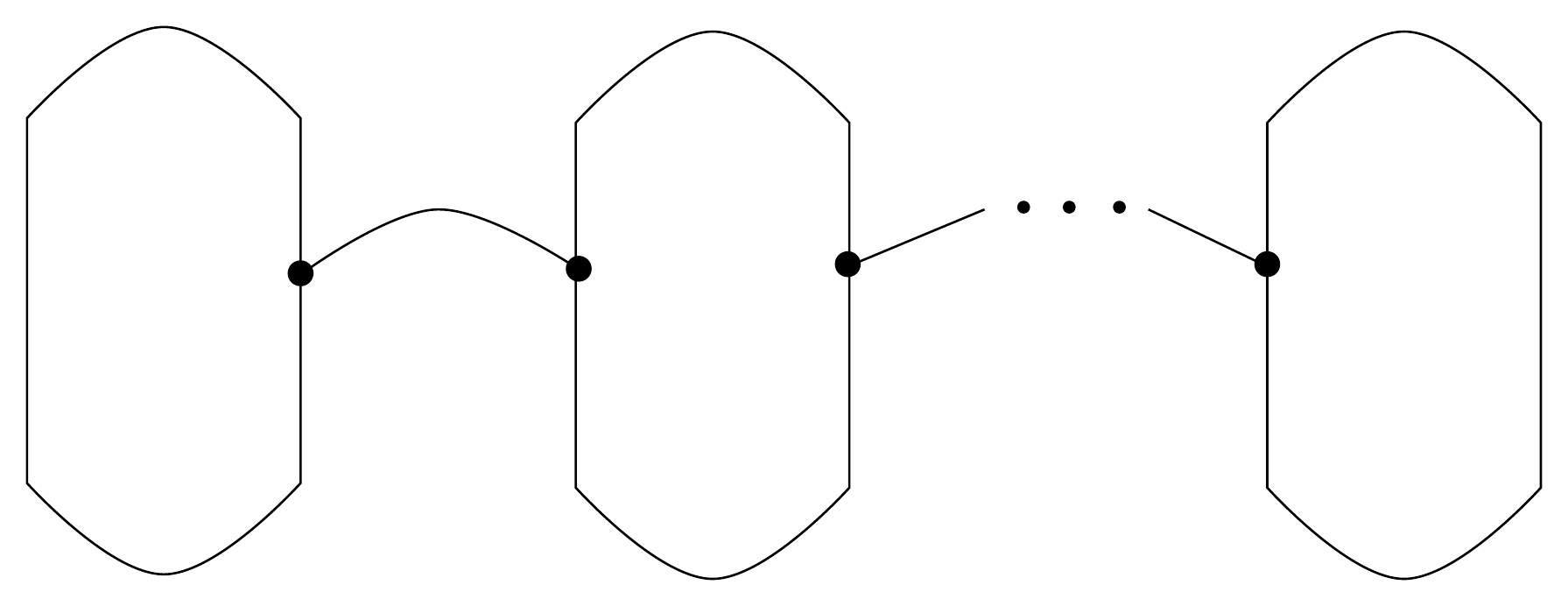}
  \end{center}  
where we require that the rings are embedded trivially, i.e. each bounds a 2-disk, and we use dots to denote the ends of a string, for instance 
  \begin{center}
     \includegraphics[scale=0.35]{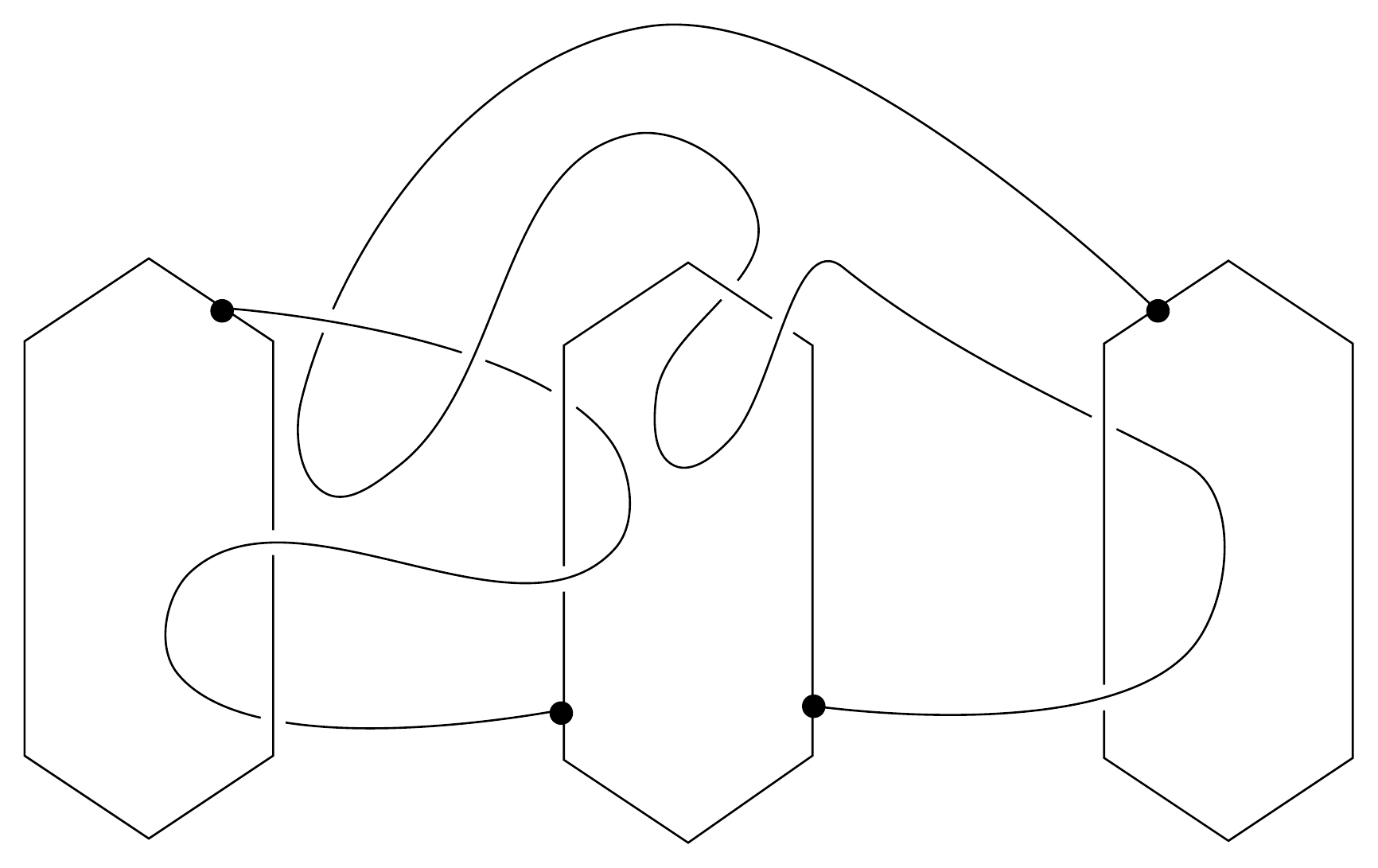}
  \end{center}
To obtain the ribbon knot, we simply unzip the strings between the rings
  \begin{center}
     \includegraphics[scale=0.35]{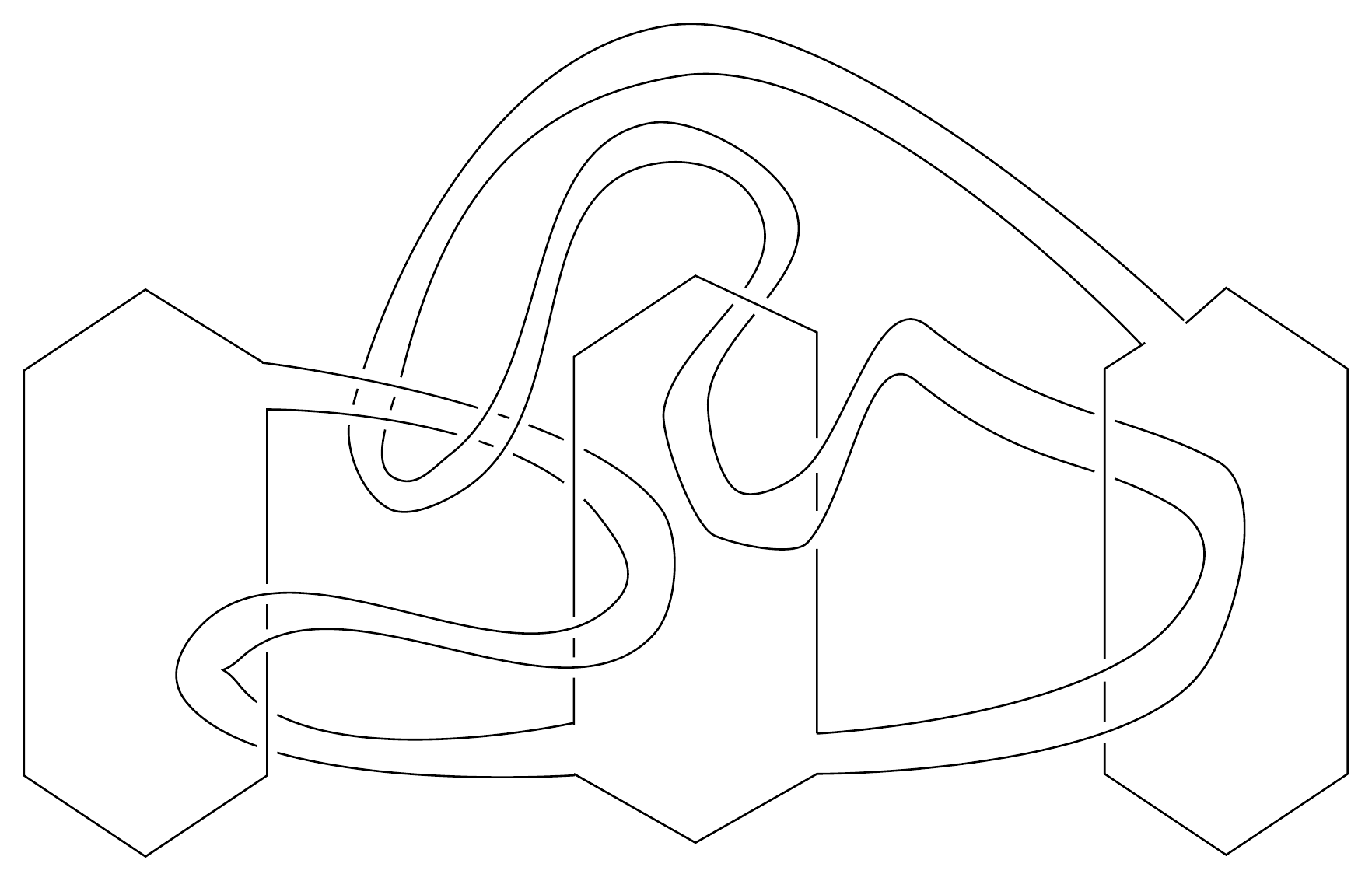}
  \end{center} 
Now given a ribbon presentation of a ribbon knot, observe that if we can deform it into the following form 
  \begin{center}
    \includegraphics[scale=0.35]{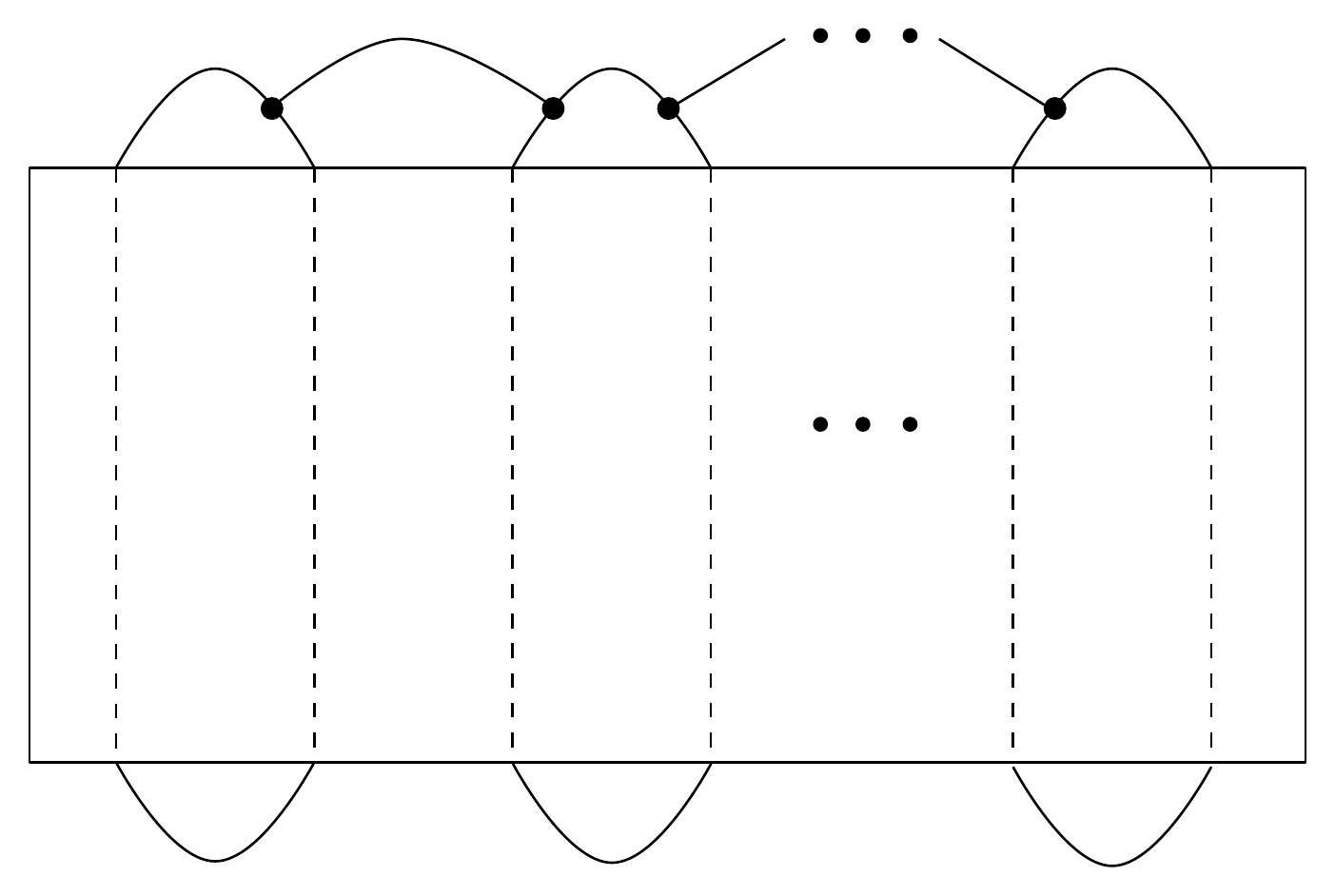}
  \end{center} 
then the tangle inside the rectangle satisfies our requirements (here again dashed lines mean they can be knotted in any manner). To see why, note that the $\tau$ closure amounts to removing the strings, which results in a trivial tangle, and the $\kappa$ closure is equivalent to unzipping to strings, which results in the knot. 

Therefore, it suffices to show that given any ribbon presentation, we can deform it to the above form. For that, we need to make two cuts to the ribbon presentation, the bottom cut and the top cut. The bottom cut is easy to perform. Namely, for each ring, we can pull the bottom part down below so that it does not interact with any string. Then we cut all the bottom parts. For the top cut, we first need to deform the ribbon presentation as follows. We describe the method for a particular example, but it is representative of a general case. Our strategy would be to move the dots along the strings, which will drag parts of the rings along in the process. For our example, we first move the dot from the third ring along the string, which pulls along a part of the third ring
  \begin{center}
     \includegraphics[scale=0.28]{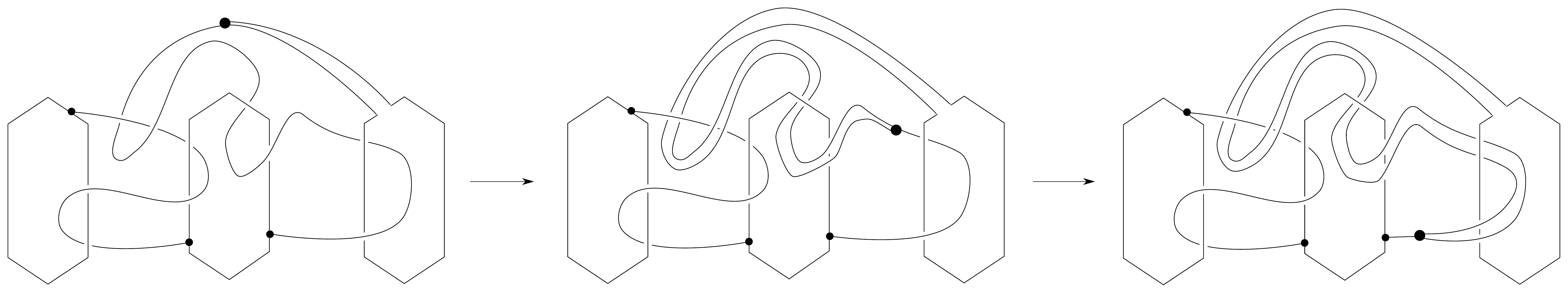}
  \end{center}
When we get close to the end of the string, we move both dots along the second ring to the other dot on the second ring.
  \begin{center}
     \includegraphics[scale=0.3]{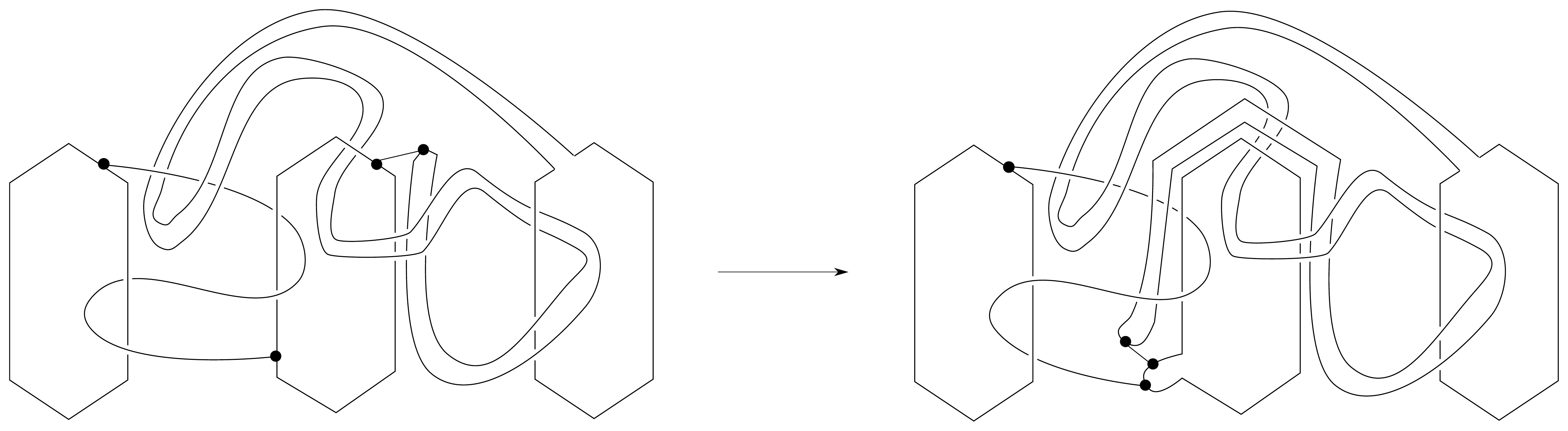}
  \end{center}
Then we pull all three dots along the remaining strings, which pull along parts of the second and the third rings
  \begin{center}
    \includegraphics[scale=0.3]{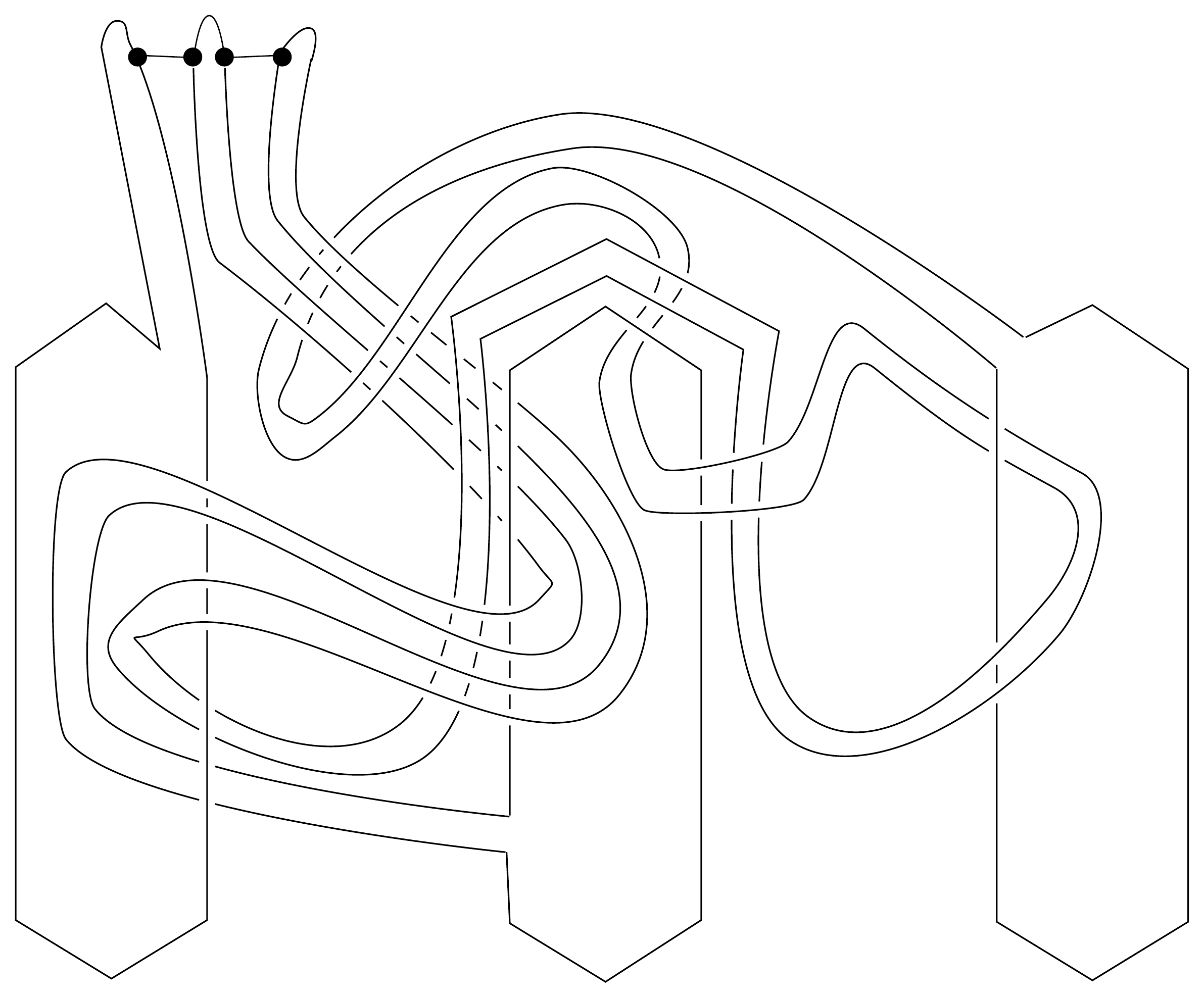}
  \end{center} 
This procedure allows us to pull all the dots and the strings above all rings, then we can easily make the top cut as follows. 
   \begin{center}
     \includegraphics[scale=0.3]{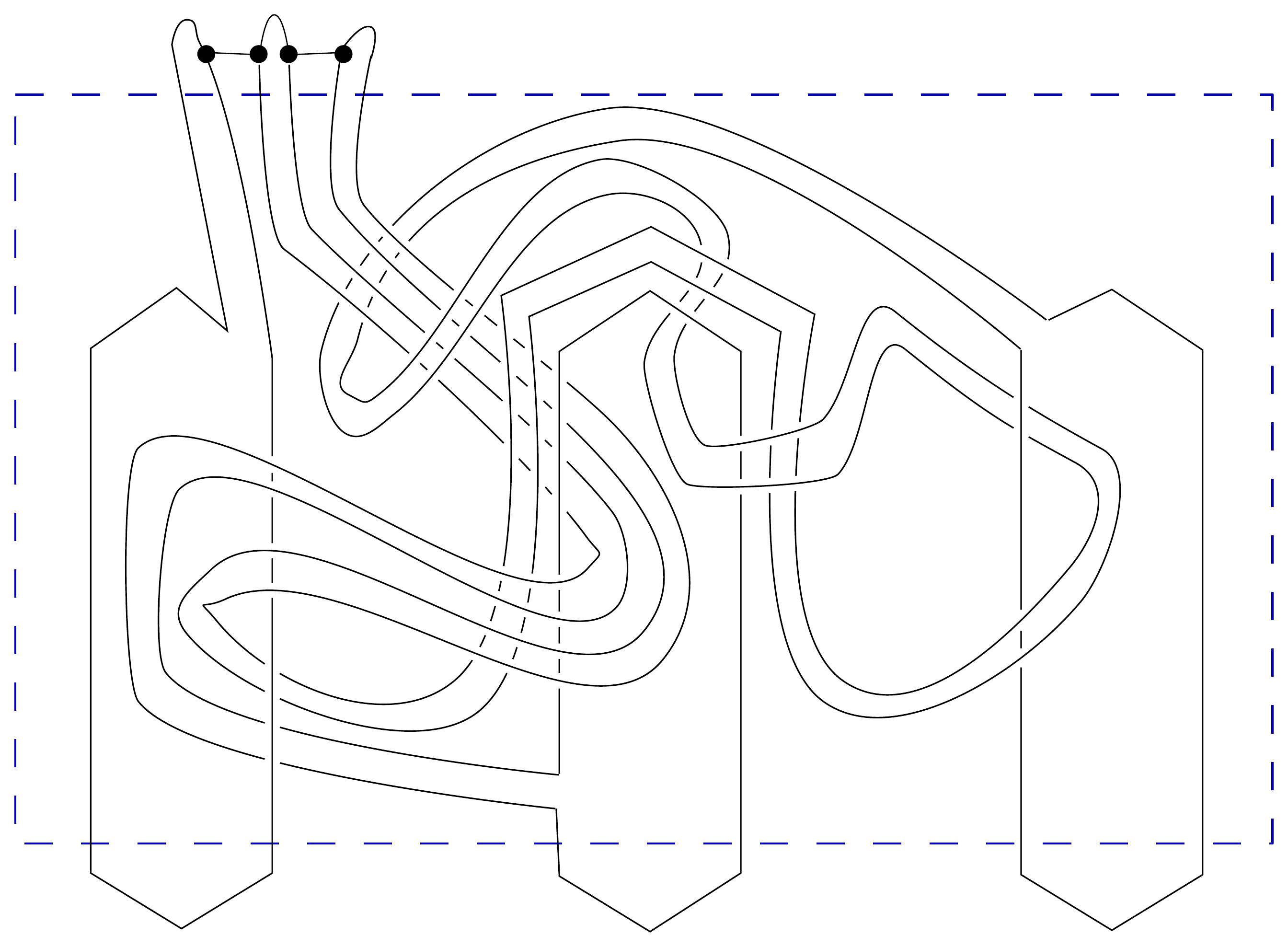}
   \end{center} 
Our required tangle is contained in the rectangle. This completes the only if direction.

For the if direction, we need to show that if a tangle $T$ satisfies the condition, then its $\kappa$ closure is ribbon. By assumption, when we take the $\tau$ closure, we can deform the link to a trivial position. In the process, we can make sure that the bands in the $\kappa$ closure intersect the interior of the disks transversely, i.e. ribbon singularities. 
   \begin{center}
     \includegraphics[scale=0.4]{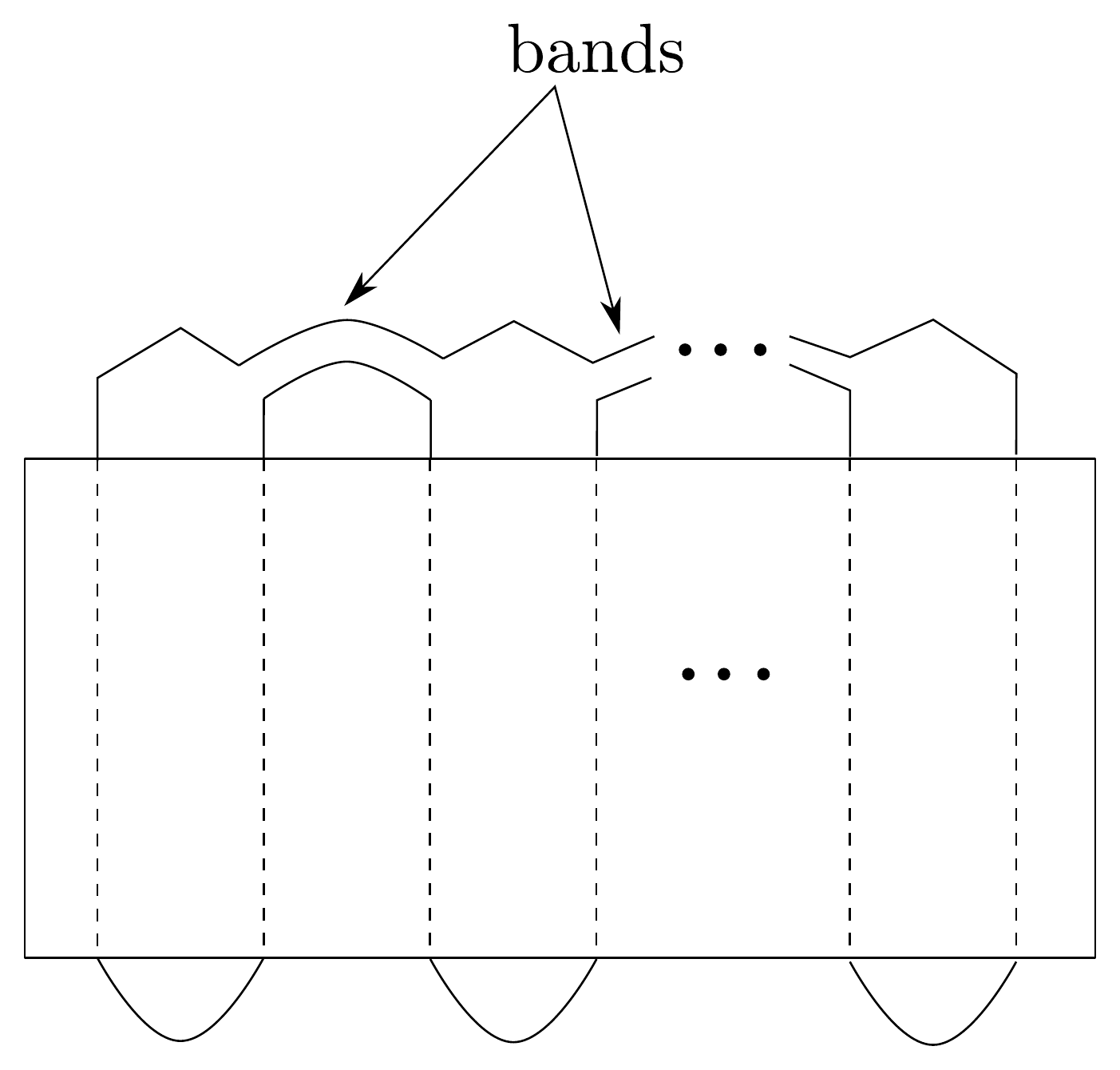}
   \end{center}
The result is a ribbon presentation and hence the knot is ribbon. Again let us look at a concrete example. Consider the following tangle
  \begin{center}
    \includegraphics[scale=0.32]{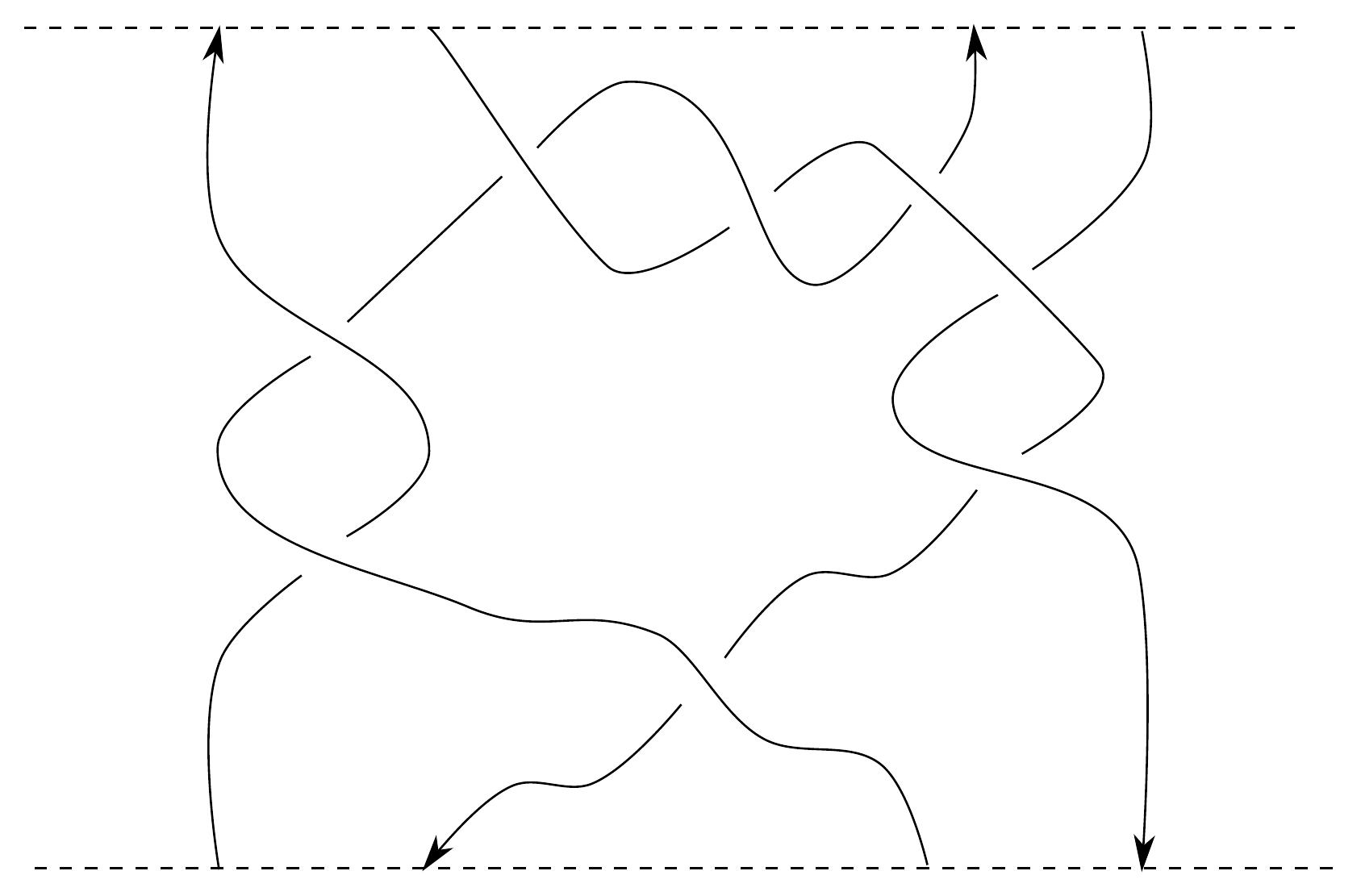}
  \end{center} 
Taking the $\tau$ closure we obtain 
   \begin{center}
     \includegraphics[scale=0.32]{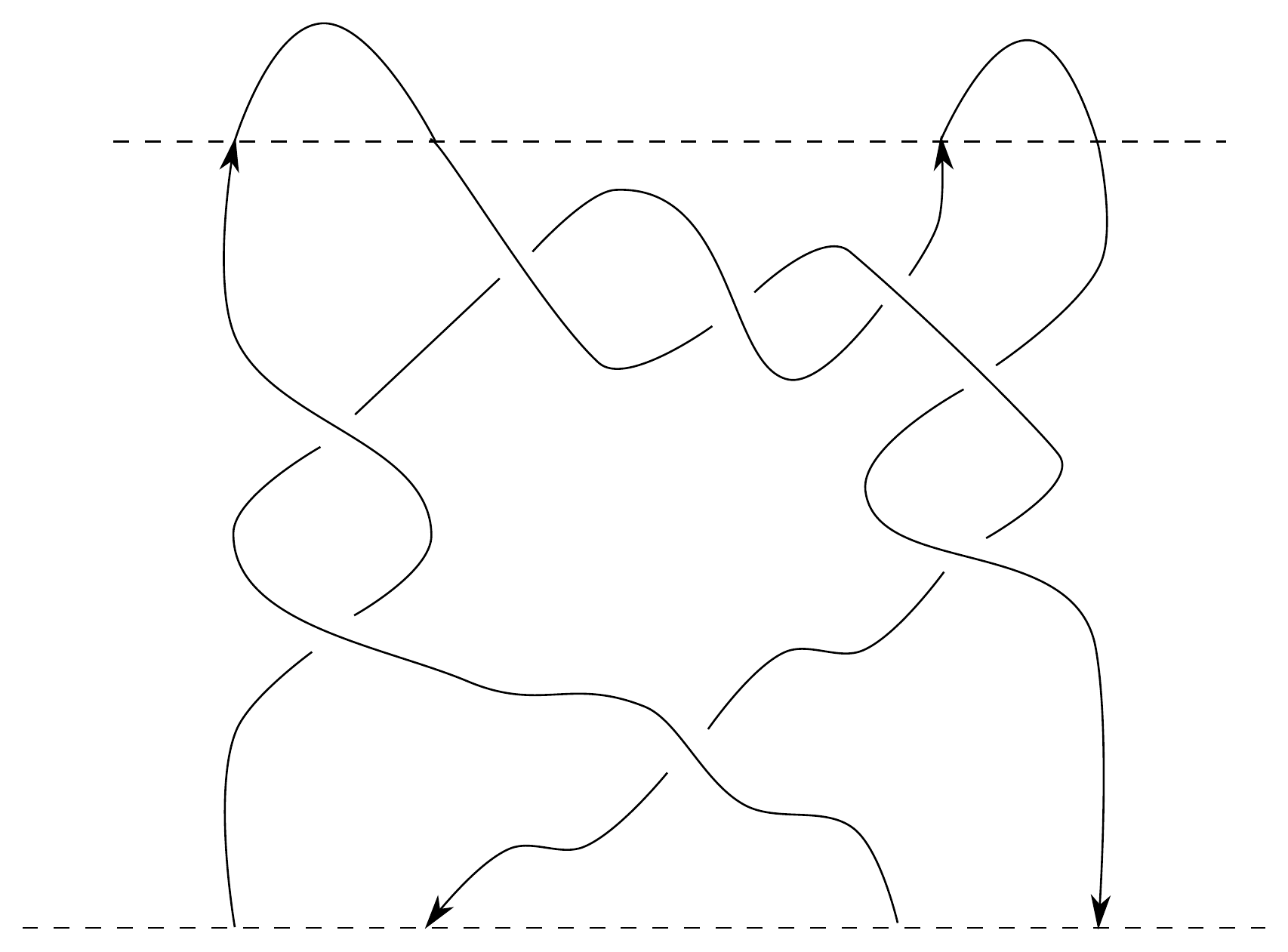}
   \end{center}  
which one can check to be the trivial tangle. The tangle satisfies the condition of the proposition, therefore it represents a ribbon knot. To see which one it is we look at the $\kappa$ closure, whereas here we also connect the first and the last strand to obtain a closed knot 
  \begin{center}
     \includegraphics[scale=0.32]{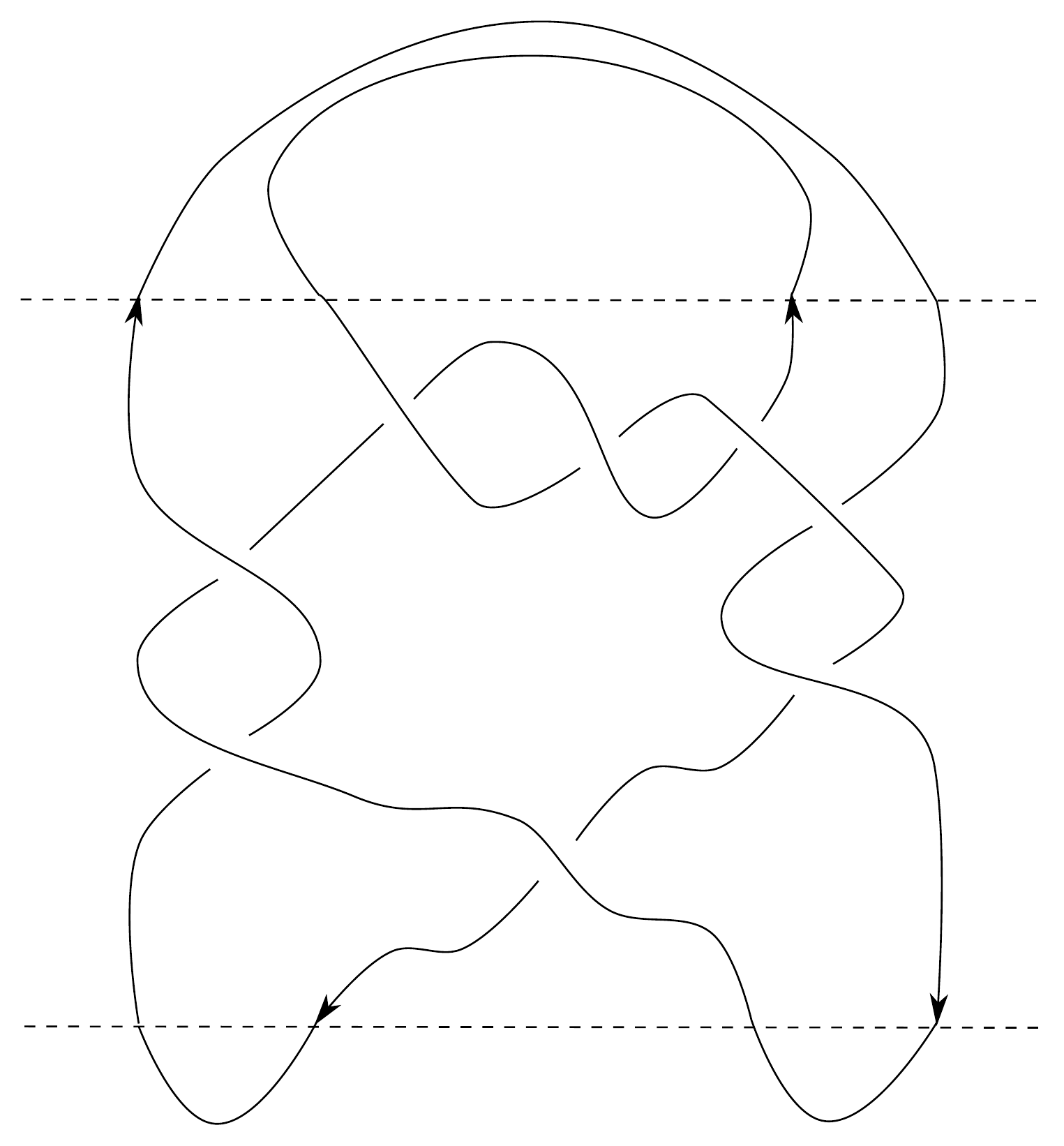}
  \end{center} 
which one can deform into the following form 
   \begin{center}
      \includegraphics[scale=0.3]{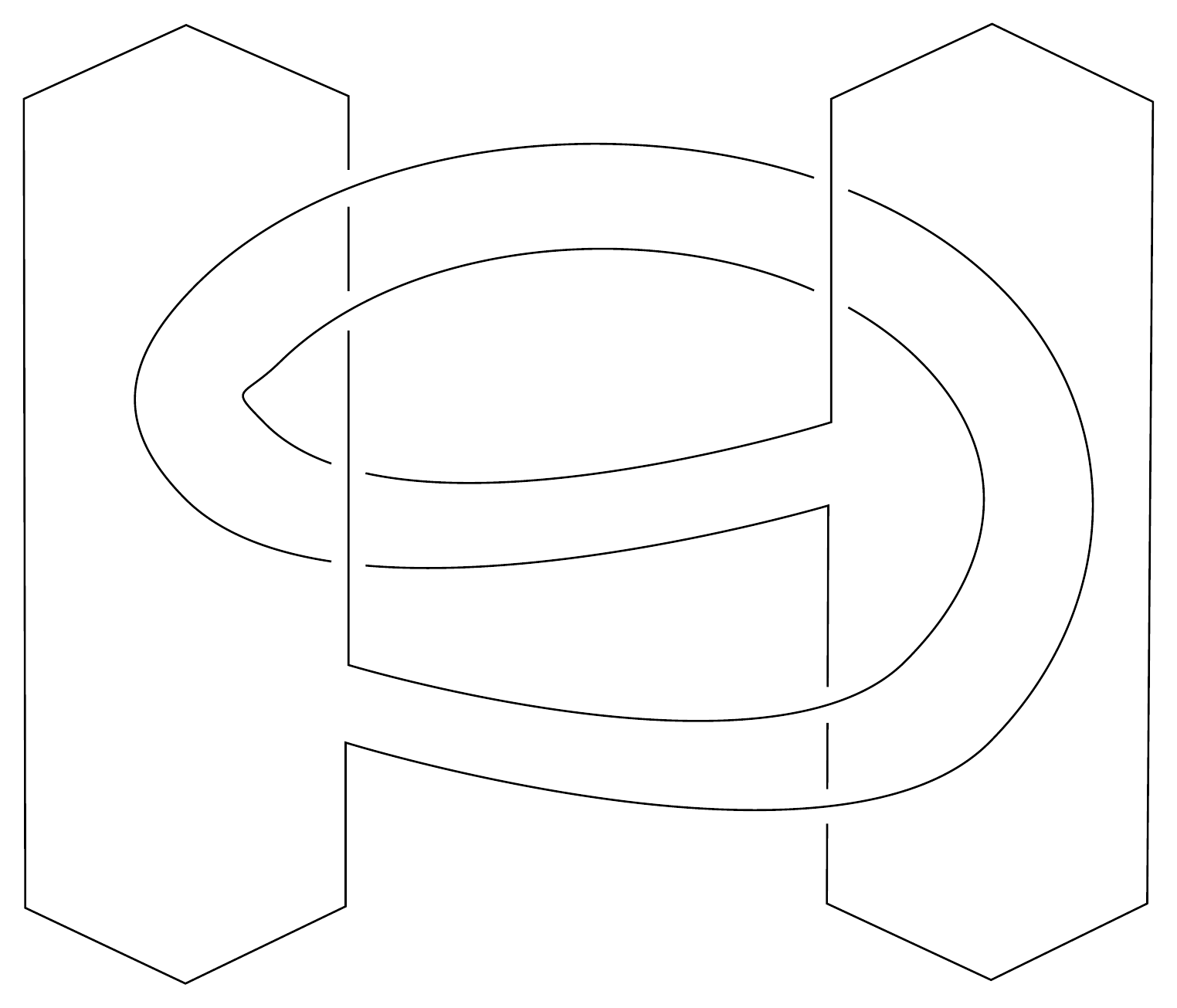}
   \end{center}
In this form one easily sees that the knot is ribbon.              
\end{proof}

\subsection{Unitary Property} To prove the Fox-Milnor condition, we first need to establish the ``unitary property'' of tangles (which may not hold for w-tangles). For our purpose however, we only need the case of string links. A key topological fact is given in the following lemma. 

\begin{lem}\label{stringlinkbraid}
 Every string link can be obtained from a braid by connecting the right-most outgoing strand with the right-most incoming strand successively finitely many times. 
\end{lem}

\begin{proof}
 First we deform the string link to a Morse position. If the string link contains no downward arcs, then it is a braid and there is nothing to do. Otherwise, because each strand goes from bottom to top, the cups and caps will occur consecutively in pairs, and each downward arc will occur between a consecutive pair of cup and cap. Our strategy will be to transform each downward arc into a closing of the last strand as follows. Look at a particular downward arc which occurs between a pair of cup and cap. There will generally be a number of arcs between them, which go either over or under the downward arc. By introducing new cups and caps we can make sure that between a cup and a cap there is only one arc which goes either over or under the downward arc. 
   \begin{center}
    \includegraphics[scale=0.3]{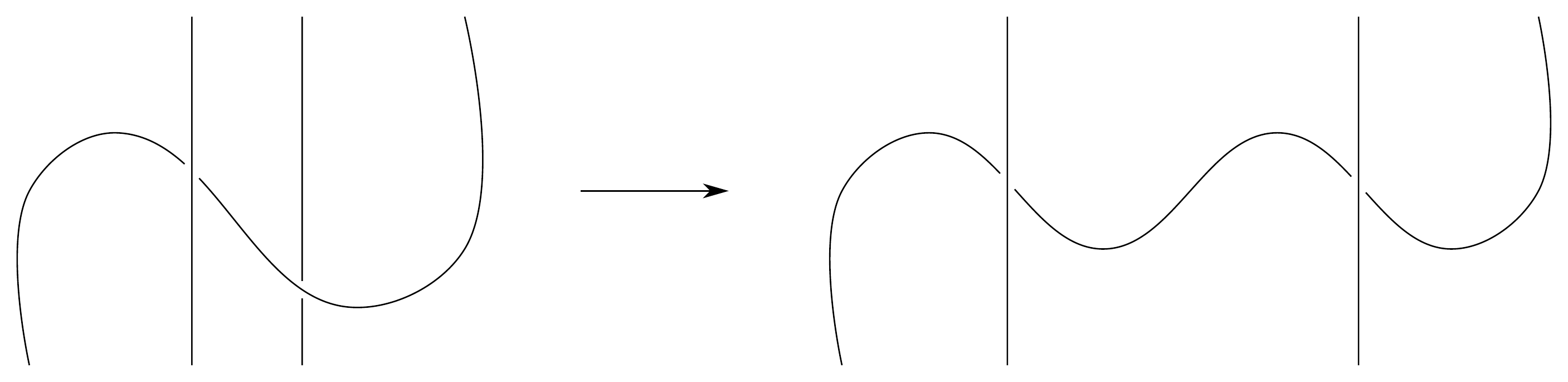}
    \end{center}  
So it suffices to consider the following cases   
    \begin{center}
    \includegraphics[scale=0.3]{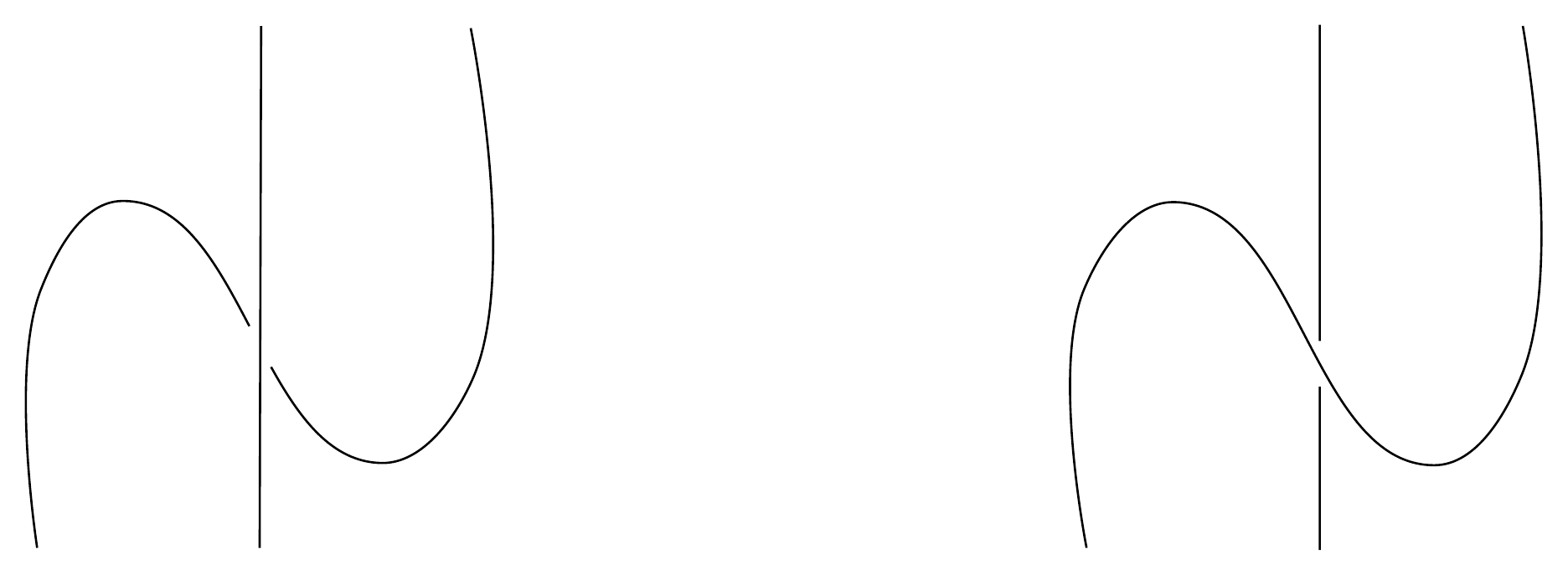}
    \end{center}
For the case where the arc goes over, we create a ``finger'' at the cup and and a ``finger'' at the cap and bring them to the right-most position going under the remaining strands and then pull part of the arc. 
   \begin{center}
    \centering
    \includegraphics[scale=0.3]{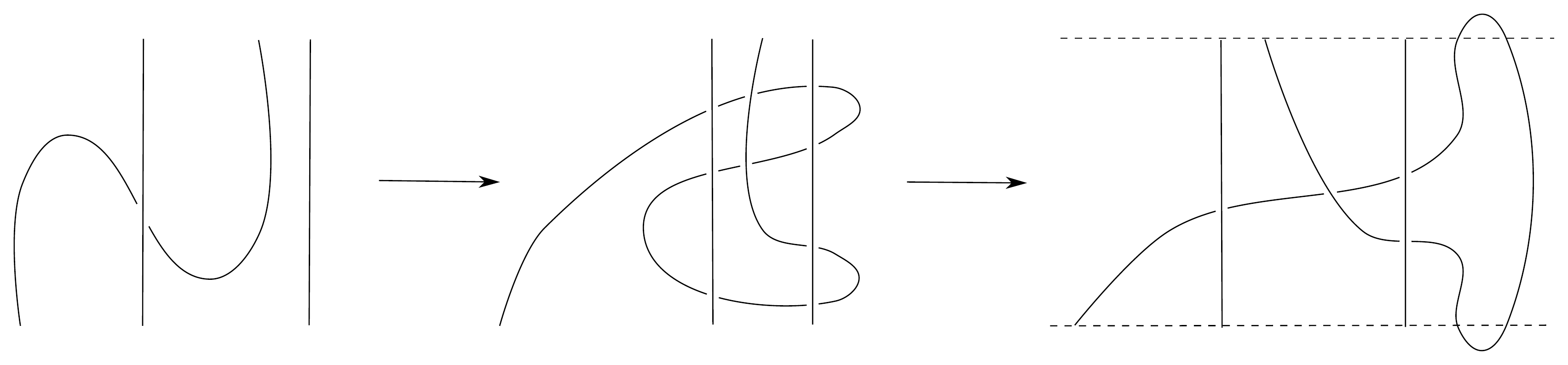}
    \end{center}      
This procedure will turn a pair of cup and cap into a stitching (or closing) of the last strand and does not introduce any new downward arc. Therefore we can repeatedly use it obtain our desired form. The case where the arc goes under is similar, we just have to pull the cup and cap to the right going over the remaining strands. 
\end{proof} 

For a matrix $A$ with entries rational functions in $t_i$, let $A^*$ be $\overline{A^{t}}$, where $A^t$ is the transpose of $A$ and $\overline{A}$ is the operation sending all variables $t_i$ to $t_i^{-1}$ applied to each entry of $A$. Recall also that for an $n\times n$ matrix $M$ and a permutation $\rho=(\rho_1,\rho_2,\dots,\rho_n)$ we let $M^{\rho}$ be the matrix obtained by permuting the columns of $M$ according to $\rho$, i.e. the $j$th column of $M^{\rho}$ is the $\rho_j$th column of $M$ (see \eqref{permutecolumn}). Now we are ready to state the unitary property:

\begin{thm*}[Unitary Property]
  Let $T$ be a string link and suppose that for simplicity the bottom endpoints of $T$ are labeled by $(1,2,\dots,n)$ and the top endpoints of $T$ are labeled by $(\rho_1,\dots,\rho_n)$ and
    \[
       \varphi(T)=\left(\begin{array}{c|ccc}
            \omega & x_1 & \cdots & x_n \\ \hline  
            y_1 & & & \\
            \vdots & & \scalebox{2}{$M$} & \\
            y_n & & & 
          \end{array}
       \right).
    \] 
Then we have 
   \[(M^{\rho})^{*}\Omega M^{\rho}=\Omega(\rho),\]
 and 
   \[\overline{\omega}\doteq\omega\det(M^{\rho}).\] 
 Here $\rho$ is the permutation induced by the skeleton of $T$. The matrix $\Omega$ is given by 
  \[\Omega=\begin{pmatrix}
      (1-t_1)^{-1} & 0 &\cdots & 0\\
      1 & (1-t_2)^{-1} &\cdots & 0\\
      \vdots & \vdots & \ddots & \vdots\\
      1 & 1 & \cdots & (1-t_n)^{-1}
  \end{pmatrix} \]
and $\Omega(\rho)$ is the action of $\rho$ on $\Omega$ by permuting the diagonal entries, i.e. 
     \[\Omega(\rho)=\begin{pmatrix}
      (1-t_{\rho_1})^{-1} & 0 &\cdots & 0\\
      1 & (1-t_{\rho_2})^{-1} &\cdots & 0\\
      \vdots & \vdots & \ddots & \vdots\\
      1 & 1 & \cdots & (1-t_{\rho_n})^{-1}
  \end{pmatrix}. \]  
\end{thm*}

\begin{remark}
  Before presenting the proof let us explain the name unitary property. In the case where $\rho$ is the identity matrix, i.e. pure string links, or when we identify all the variables $t_i$, i.e. the Burau representation, we obtain
  \[
    M^*\Omega M=\Omega.
  \]
Taking the conjugate transpose of both sides we obtain 
  \[
    M^*\Omega^*M=\Omega^*.
  \]  
Therefore if we let $\Psi=i\Omega-i\Omega^*$, then 
  \[
     M^*\Psi M=\Psi.
  \]
Note that the matrix $\Psi$ is Hermitian since 
 \[
   \Psi^*=(i\Omega-i\Omega^*)^*=i\Omega-i\Omega^*=\Psi,
 \]  
hence the matrix $M$ is unitary with respect to the Hermitian form $\Psi$. 
\end{remark}

\begin{proof}
  By Lemma \ref{stringlinkbraid} we just need to show that the property holds for braids and is invariant under stitching the right-most outgoing strand with the right-most incoming strand. To streamline the proof, we separate the matrix part and the scalar part. 
  
  \underline{\textbf{The matrix part:}} Let us first check the crossings 
    \[
     \begin{pmatrix}
       1-t_a^{-1} & t_a^{-1} \\
       1 & 0
     \end{pmatrix}\begin{pmatrix}
        (1-t_a)^{-1} & 0 \\
        1 & (1-t_b)^{-1}
     \end{pmatrix}\begin{pmatrix}
       1-t_a & 1\\
       t_a & 0
     \end{pmatrix}=\begin{pmatrix}
        (1-t_b)^{-1} & 0 \\
        1 & (1-t_a)^{-1}
     \end{pmatrix},
   \]
 and 
    \[
     \begin{pmatrix}
       0 & 1 \\
       t_a & 1-t_a
     \end{pmatrix}\begin{pmatrix}
        (1-t_b)^{-1} & 0 \\
        1 & (1-t_a)^{-1}
     \end{pmatrix}\begin{pmatrix}
       0 & t_a^{-1}\\
       1 & 1-t_a^{-1}
     \end{pmatrix}=\begin{pmatrix}
        (1-t_a)^{-1} & 0 \\
        1 & (1-t_b)^{-1}
     \end{pmatrix}.
   \]    
Let us also remark here that the above property does not hold for w-string links, simply because it does not hold for a virtual crossing (recall that in $\Gamma$-calculus a virtual crossing is sent to the identity matrix, so in a sense it is ``not even there''):
    \[
     \begin{pmatrix}
       0 & 1 \\
       1 & 0
     \end{pmatrix}\begin{pmatrix}
        (1-t_a)^{-1} & 0 \\
        1 & (1-t_b)^{-1}
     \end{pmatrix}\begin{pmatrix}
       0 & 1\\
       1 & 0
     \end{pmatrix}\neq \begin{pmatrix}
        (1-t_b)^{-1} & 0 \\
        1 & (1-t_a)^{-1}
     \end{pmatrix}.
   \] 
The computation clearly extends to generators of the braid groups (extend by block identity matrix). Next observe that the unitary property is invariant under composition of string links (or braids in particular). Indeed, consider two string links $\beta_1$ and $\beta_2$ with induced permutations $\rho_1$ and $\rho_2$, respectively: 
    \[\varphi(\beta_1)=\left(\begin{array}{c|c}
      \omega_1 & x_{\vec{a}\rho_1} \\
      \hline
      y_{\vec{a}} & M_1^{\rho_1}
    \end{array}  
  \right)\quad\text{ and }\quad \varphi(\beta_2)=\left(\begin{array}{c|c}
      \omega_2 & x_{\vec{b}\rho_2} \\
      \hline
      y_{\vec{b}} & M_2^{\rho_2}
    \end{array}  
  \right)
  \] 
and suppose that we have 
  \[(M_1^{\rho_1})^*\Omega(\vec{a}) M_1^{\rho_1}=\Omega(\vec{a}\rho_1)\quad{\text{ and }}\quad (M_2^{\rho_2})^*\Omega(\vec{b})M_2^{\rho_2}=\Omega(\vec{b}\rho_2).\]   
 Recall that the result of composing $\beta_1$ and $\beta_2$ is 
  \[
    \varphi(\beta_1\cdot\beta_2)=\left(\begin{array}{c|c}
         \omega_1\omega_2 & x_{\vec{a}\rho_1\rho_2} \\
         \hline
         y_{\vec{a}} & M_1^{\rho_1}M_2^{\rho_2}
      \end{array}
    \right)_{t_{\vec{b}}\to t_{\vec{a}\tau}}.
  \] 
Thus with $t_{\vec{b}}\to t_{\vec{a}\tau}$ we have 
 \begin{align*}
     (M_1^{\rho_1}M_2^{\rho_2})^{*}\Omega(\vec{a}) (M_1^{\rho_1}M_2^{\rho_2})&=(M_2^{\rho_2})^*(M_1^{\rho_1})^*\Omega(\vec{a})M_1^{\rho_1}M_2^{\rho_2}\\
            &=(M_2^{\rho_2})^*\Omega(\vec{a}\rho_1) M_2^{\rho_2} \\
            &=(M_2^{\rho_2})^*\Omega(\vec{b}) M_2^{\rho_2} \\
            &=\Omega(\vec{b}\rho_2)\\
            &=\Omega(\vec{a}\rho_1\rho_2), 
 \end{align*}
as required. So the property holds for the case of braids (compare with \cite{BNT14}). 
        
 Now given a string link $\beta$ with induced permutation $\rho=(\rho_1,\rho_2,\dots,\rho_n)$ and $\rho_n\neq n$ and suppose we want to stitch the right-most outgoing strand to the right-most incoming strand. Note that by composing the top and bottom of $\beta$ with appropriate permutation braids we can bring $\beta$ to a standard form where the induced permutation is $(1,2,\dots,n-2,n,n-1)$, i.e. the transposition $(n-1,n)$ and we stitch strand $n-1$ to strand $n$. For example, 
     \begin{center}
    \includegraphics[scale=0.5]{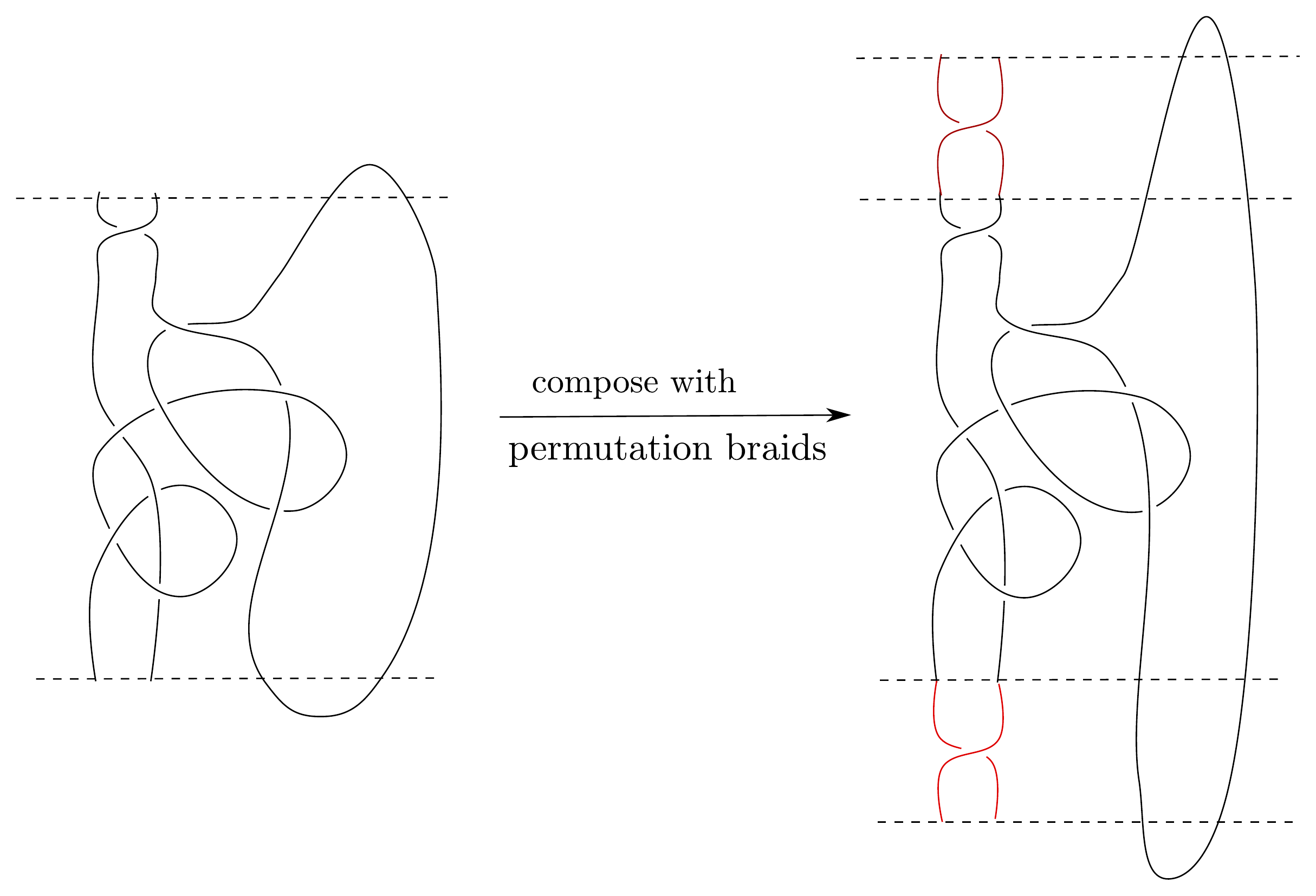}
    \end{center}      
 Since we have shown unitarity for braids and composition, it suffices to consider the string link $\beta$ with the above configuration. To that end, let 
 \[
   \varphi(\beta)=\left(\begin{array}{c|ccc}
       \omega & x_{n-1} & x_n & x_S \\
       \hline
       y_{n-1} & \alpha & \beta & \theta \\
       y_n & \gamma & \delta & \epsilon\\
       y_S & \phi & \psi & \Xi
     \end{array}
   \right)\xrightarrow[]{m^{n-1,n}_{n-1}} 
      \left(\begin{array}{c|cc}
          (1-\gamma)\omega & x_{n-1} & x_S \\ \hline
          y_{n-1} & \beta+\frac{\alpha\delta}{1-\gamma} & \theta+\frac{\alpha\epsilon}{1-\gamma} \\
          y_S & \psi+\frac{\delta\phi}{1-\gamma} & \Xi+\frac{\phi\epsilon}{1-\gamma}
        \end{array}
      \right)_{t_n\to t_{n-1}},
 \]      
where $S=\{1,\dots,n\}\setminus\{n-1,n\}$. Assume $\beta$ satisfies the unitary property, for that we need to rearrange the matrix part as follows 
  \[
    \left(\begin{array}{c|ccc}
       \omega & x_S & x_n & x_{n-1} \\ \hline
       y_S & \Xi & \psi & \phi \\
       y_{n-1} &\theta &\beta &\alpha \\
       y_n & \epsilon &\delta &\gamma
      \end{array}    
    \right).
  \]
Let us denote
  \[
     M=\begin{pmatrix}
        \Xi & \psi & \phi \\
        \theta & \beta & \alpha \\
        \epsilon & \delta & \gamma
     \end{pmatrix}.
  \]
Then the unitary statement is 
 \begin{equation}\label{unitarymatrix}
 M^*\Omega M=\Omega(\rho),
 \end{equation}
where to simplify notation we put
\[
    \Omega=\begin{pmatrix}
       \Omega_{n-2} & \vec{0} &\vec{0}\\
       \vec{1} & (1-t_{n-1})^{-1} & 0\\
       \vec{1} & 1 & (1-t_n)^{-1}
    \end{pmatrix},
 \]
and 
\[
    \Omega(\rho)=\begin{pmatrix}
       \Omega_{n-2} & \vec{0} &\vec{0}\\
       \vec{1} & (1-t_{n})^{-1} & 0\\
       \vec{1} & 1 & (1-t_{n-1})^{-1}
    \end{pmatrix},
 \]  
and 
  \[
    \Omega_{n-2}=\begin{pmatrix}
      (1-t_1)^{-1} & 0 &\cdots & 0\\
      1 & (1-t_2)^{-1} &\cdots & 0\\
      \vdots & \vdots & \ddots & \vdots\\
      1 & 1 & \cdots & (1-t_{n-2})^{-1}
  \end{pmatrix}.
  \]
Here $\vec{1}$ denotes either a row or a column or a square matrix (the size of which depends on the context) consists entirely of 1's and similarly for $\vec{0}$. Now to show that the unitary property is invariant under stitching we first need to decompose the stitching operation into a sequence of elementary operations as follows:
  \begin{align*}
     \begin{pmatrix}
       \Xi & \psi & \phi \\
       \theta & \beta & \alpha \\
       \epsilon & \delta & \gamma
     \end{pmatrix}\xrightarrow[]{} \begin{pmatrix}
       \Xi & \psi & \phi \\
       \theta & \beta &\alpha\\
       \epsilon & \delta & \gamma-1
     \end{pmatrix}\xrightarrow[]{} \begin{pmatrix}
        \Xi & \psi & \phi \\
        \theta &\beta &\alpha\\
        \frac{\epsilon}{\gamma-1} & \frac{\delta}{\gamma-1} & 1
     \end{pmatrix}\xrightarrow[]{} \begin{pmatrix}
        \Xi+\frac{\phi\epsilon}{1-\gamma} & \psi+\frac{\delta \phi}{1-\gamma} & \vec{0}\\
        \theta+\frac{\alpha\epsilon}{1-\gamma} & \beta+\frac{\alpha\delta}{1-\gamma} & 0\\
        \frac{\epsilon}{\gamma-1} & \frac{\delta}{\gamma-1} & 1
     \end{pmatrix}.
  \end{align*}  
Note that except for the first one, all the operations are simply elementary row operations. Now under stitching, we identify $t_{n-1}$ and $t_n$. In what follows, we set $t_n$ to be $t_{n-1}$. Then $\left.\Omega\right|_{t_n\to t_{n-1}}=\left.\Omega(\sigma)\right|_{t_{n}\to t_{n-1}}$ and again to avoid cumbersome notations we will denote both of them by $\Omega$. We then write \eqref{unitarymatrix} as 
  \begin{equation}\label{unitarystitching}
   \left[\begin{pmatrix}
      \Xi^* &\theta^* &\epsilon^*\\
      \psi^* &\beta^* & \delta^* \\
      \phi^* &\alpha^* &\gamma^*-1
   \end{pmatrix}+\begin{pmatrix}
      \vec{0} & \vec{0} & \vec{0}\\
      \vec{0} & 0 & 0\\
      \vec{0} & 0 & 1
   \end{pmatrix}\right]\Omega \left[\begin{pmatrix}
     \Xi & \psi &\phi \\
     \theta & \beta & \alpha\\
     \epsilon &\delta & \gamma-1
   \end{pmatrix}+\begin{pmatrix}
      \vec{0} & \vec{0} & \vec{0}\\
      \vec{0} & 0 & 0\\
      \vec{0} & 0 & 1
   \end{pmatrix}\right]=\Omega. 
  \end{equation}
Observe that 
 \[
    \begin{pmatrix}
      \Xi^* &\theta^* &\epsilon^*\\
      \psi^* &\beta^* & \delta^* \\
      \phi^* &\alpha^* &\gamma^*-1
   \end{pmatrix} \Omega \begin{pmatrix}
      \vec{0} & \vec{0} & \vec{0}\\
      \vec{0} & 0 & 0\\
      \vec{0} & 0 & 1
   \end{pmatrix}=\begin{pmatrix}
      \vec{0} & \vec{0} & \frac{\epsilon^*}{1-t_{n-1}} \\
      \vec{0} & 0 & \frac{\delta^*}{1-t_{n-1}} \\
      \vec{0} & 0 & \frac{\gamma^*-1}{1-t_{n-1}} 
   \end{pmatrix}
 \]  
and 
 \begin{align*}
    &\begin{pmatrix}
      \vec{0} & \vec{0} & \vec{0}\\
      \vec{0} & 0 & 0\\
      \vec{0} & 0 & 1
   \end{pmatrix}\Omega \begin{pmatrix}
     \Xi & \psi &\phi \\
     \theta & \beta & \alpha\\
     \epsilon &\delta & \gamma-1
   \end{pmatrix}\\
    &=\begin{pmatrix}
      \vec{0} & \vec{0} & \vec{0}\\
      \vec{0} & 0 & 0\\
      \theta+\left\langle \Xi\right\rangle +\frac{\epsilon}{1-t_{n-1}} & \beta+\left\langle\psi\right\rangle+\frac{\delta}{1-t_{n-1}} & \alpha+\left\langle\phi\right\rangle+\frac{\gamma-1}{1-t_{n-1}}
   \end{pmatrix}.
 \end{align*}
We also have 
 \[
    \begin{pmatrix}
      \vec{0} & \vec{0} & \vec{0}\\
      \vec{0} & 0 & 0\\
      \vec{0} & 0 & 1
   \end{pmatrix} \Omega \begin{pmatrix}
      \vec{0} & \vec{0} & \vec{0}\\
      \vec{0} & 0 & 0\\
      \vec{0} & 0 & 1
   \end{pmatrix}=\begin{pmatrix}
      \vec{0} & \vec{0} & \vec{0}\\
      \vec{0} & 0 & 0\\
      \vec{0} & 0 & \frac{1}{1-t_{n-1}}
   \end{pmatrix}.
 \]   
Therefore \eqref{unitarystitching} becomes 
\begin{multline}
 \label{unitarystitching2}
     \begin{pmatrix}
      \Xi^* &\theta^* &\epsilon^*\\
      \psi^* &\beta^* & \delta^* \\
      \phi^* &\alpha^* &\gamma^*-1
   \end{pmatrix}\Omega \begin{pmatrix}
     \Xi & \psi &\phi \\
     \theta & \beta & \alpha\\
     \epsilon &\delta & \gamma-1
   \end{pmatrix}=\\
   \begin{pmatrix}
       \Omega_{n-2} & \vec{0} &\frac{\epsilon^*}{-1+t_{n-1}}\\
       \vec{1} & (1-t_{n-1})^{-1} & \frac{\delta^*}{-1+t_{n-1}}\\
       \vec{1}-\theta-\left\langle\Xi\right\rangle+\frac{\epsilon}{-1+t_{n-1}} & 1-\beta-\left\langle\psi\right\rangle+\frac{\delta}{-1+t_{n-1}} & \frac{-2+\alpha+\gamma+\left\langle\phi\right\rangle-(\alpha+\left\langle\phi\right\rangle)t_{n-1}+\gamma^*}{-1+t_{n-1}}
    \end{pmatrix}.
\end{multline} 
By Lemma \ref{columnsum} we can rewrite the above as 
\begin{equation}\label{unitarystitching3}
     \begin{pmatrix}
      \Xi^* &\theta^* &\epsilon^*\\
      \psi^* &\beta^* & \delta^* \\
      \phi^* &\alpha^* &\gamma^*-1
   \end{pmatrix}\Omega \begin{pmatrix}
     \Xi & \psi &\phi \\
     \theta & \beta & \alpha\\
     \epsilon &\delta & \gamma-1
   \end{pmatrix}=\begin{pmatrix}
       \Omega_{n-2} & \vec{0} &\frac{\epsilon^*}{-1+t_{n-1}}\\
       \vec{1} & (1-t_{n-1})^{-1} & \frac{\delta^*}{-1+t_{n-1}}\\
       \frac{t_{n-1}\epsilon}{-1+t_{n-1}} & \frac{t_{n-1}\delta}{-1+t_{n-1}} & \frac{-1+\gamma^*-(1-\gamma)t_{n-1}}{-1+t_{n-1}}
    \end{pmatrix}.
\end{equation} 
Consider the left hand side of the above identity, we can obtain the stitching formula by a sequence of elementary row and column operations. By employing elementary matrices, we can rewrite it as  
   \[\begin{pmatrix}
      \Xi^*+\frac{\epsilon^*\phi^*}{1-\gamma^*} &\theta^*+\frac{\alpha^*\epsilon^*}{1-\gamma^*} &\frac{\epsilon^*}{\gamma^*-1}\\
      \psi^*+\frac{\delta^*\phi^*}{1-\gamma^*} &\beta^*+\frac{\alpha^*\delta^*}{1-\gamma^*} &\frac{\delta^*}{\gamma^*-1}\\
      \vec{0} & 0 & 1 
  \end{pmatrix}\widetilde{\Omega}\begin{pmatrix}
     \Xi+\frac{\phi\epsilon}{1-\gamma} & \psi+\frac{\delta\phi}{1-\gamma} & \vec{0}\\
     \theta+\frac{\alpha\epsilon}{1-\gamma} &\beta+\frac{\alpha\delta}{1-\gamma} & 0\\
     \frac{\epsilon}{\gamma-1} & \frac{\delta}{\gamma-1} & 1
  \end{pmatrix},
  \]
where 
  \begin{align*}
     \widetilde{\Omega}&=\begin{pmatrix}
        I & \vec{0} & \vec{0}\\
        \vec{0} & 1 & 0\\
        \phi^* & \alpha^* & 1
     \end{pmatrix}\begin{pmatrix}
        I & \vec{0} & \vec{0}\\
        \vec{0} & 1 & 0\\
        \vec{0} & 0 & \gamma^*-1
     \end{pmatrix}\Omega\begin{pmatrix}
        I & \vec{0} & \vec{0}\\
        \vec{0} & 1 & 0\\
        \vec{0} & 0 & \gamma-1
     \end{pmatrix}\begin{pmatrix}
        I & \vec{0} & \phi\\
        \vec{0} & 1 & \alpha\\
        \vec{0} & 0 &  1
     \end{pmatrix}\\
      &=\begin{pmatrix}
       \Omega_{n-2} & \vec{0} &\bullet\\
       \vec{1} & (1-t_{n-1})^{-1} & \bullet\\
       \bullet & \bullet & \bullet
    \end{pmatrix}.
  \end{align*} 
Here a $\bullet$ denotes an entry we do not care about. Notice that the row and column operations only affect the last row and the last column. Finally, we apply column operations to the right-most matrix and row operations to the left-most matrix to obtain 
    \[\begin{pmatrix}
      \Xi^*+\frac{\epsilon^*\phi^*}{1-\gamma^*} &\theta^*+\frac{\alpha^*\epsilon^*}{1-\gamma^*} &\vec{0}\\
      \psi^*+\frac{\delta^*\phi^*}{1-\gamma^*} &\beta^*+\frac{\alpha^*\delta^*}{1-\gamma^*} &0\\
      \vec{0} & 0 & 1 
  \end{pmatrix}\widetilde{\Omega}\begin{pmatrix}
     \Xi+\frac{\phi\epsilon}{1-\gamma} & \psi+\frac{\delta\phi}{1-\gamma} & \vec{0}\\
     \theta+\frac{\alpha\epsilon}{1-\gamma} &\beta+\frac{\alpha\delta}{1-\gamma} & 0\\
     \vec{0} & 0 & 1
  \end{pmatrix}.
  \] 
We can encode these operations as multiplying with the matrix 
  \[\begin{pmatrix}
     I & \vec{0} & \vec{0} \\
     \vec{0} & 1 & 0\\
     -\frac{\epsilon}{\gamma-1} & -\frac{\delta}{\gamma-1} & 1
  \end{pmatrix}
  \]  
on the right and its conjugate transpose 
    \[\begin{pmatrix}
     I & \vec{0} & -\frac{\epsilon^*}{\gamma^*-1} \\
     \vec{0} & 1 & -\frac{\delta^*}{\gamma^*-1}\\
     \vec{0} & 0 & 1
  \end{pmatrix}
  \]
on the left. Therefore the right hand side of \eqref{unitarystitching3} becomes 
  \[ 
    \begin{pmatrix}
     I & \vec{0} & -\frac{\epsilon^*}{\gamma^*-1} \\
     \vec{0} & 1 & -\frac{\delta^*}{\gamma^*-1}\\
     \vec{0} & 0 & 1
  \end{pmatrix}\begin{pmatrix}
       \Omega_{n-2} & \vec{0} &\frac{\epsilon^*}{-1+t_{n-1}}\\
       \vec{1} & (1-t_{n-1})^{-1} & \frac{\delta^*}{-1+t_{n-1}}\\
       \frac{t_{n-1}\epsilon}{-1+t_{n-1}} & \frac{t_{n-1}\delta}{-1+t_{n-1}} & \frac{-1+\gamma^*-(1-\gamma)t_{n-1}}{-1+t_{n-1}}
    \end{pmatrix}\begin{pmatrix}
     I & \vec{0} & \vec{0} \\
     \vec{0} & 1 & 0\\
     -\frac{\epsilon}{\gamma-1} & -\frac{\delta}{\gamma-1} & 1
  \end{pmatrix}.
   \]    
For our purpose we only need to look at the first $n-1$ rows and the first $n-1$ columns. We record these changes below 
 \[
   \Omega_{n-2}-\frac{\epsilon^*\epsilon}{(-1+t_{n-1})(\gamma-1)}+\frac{(-1+\gamma^*)\epsilon^*\epsilon}{(-1+t_{n-1})(\gamma^*-1)(-1+\gamma)}=\Omega_{n-2},
 \]   
 \[
   \vec{1}-\frac{\delta^*\epsilon}{(-1+t_{n-1})(-1+\gamma)}+\frac{(-1+\gamma^*)\delta^*\epsilon}{(-1+t_{n-1})(\gamma^*-1)(-1+\gamma)}=\vec{1},
 \]
  \[
    \vec{0}-\frac{\delta\epsilon^*}{(-1+t_{n-1})(-1+\gamma)}+\frac{(-1+\gamma^*)\delta\epsilon^*}{(-1+t_{n-1})(\gamma^*-1)(-1+\gamma)}=\vec{0},
  \]
 \[
    -\frac{-1+\gamma+\delta\delta^*}{(-1+t_{n-1})(-1+\gamma)}+\frac{(-1+\gamma^*)\delta\delta^*}{(1-t_{n-1})(\gamma^*-1)(-1+\gamma)}=\frac{1}{1-t_{n-1}}.
 \] 
Thus we see that the first $n-1$ rows and the first $n-1$ columns stay unchanged. In summary, we obtain the following identity 
 \begin{align*}
     \begin{pmatrix}
      \Xi^*+\frac{\epsilon^*\phi^*}{1-\gamma^*} &\theta^*+\frac{\alpha^*\epsilon^*}{1-\gamma^*} &\vec{0}\\
      \psi^*+\frac{\delta^*\phi^*}{1-\gamma^*} &\beta^*+\frac{\alpha^*\delta^*}{1-\gamma^*} &0\\
      \vec{0} & 0 & 1 
  \end{pmatrix}\begin{pmatrix}
       \Omega_{n-2} & \vec{0} &\bullet\\
       \vec{1} & (1-t_{n-1})^{-1} & \bullet\\
       \bullet & \bullet & \bullet
    \end{pmatrix}\begin{pmatrix}
     \Xi+\frac{\phi\epsilon}{1-\gamma} & \psi+\frac{\delta\phi}{1-\gamma} & \vec{0}\\
     \theta+\frac{\alpha\epsilon}{1-\gamma} &\beta+\frac{\alpha\delta}{1-\gamma} & 0\\
     \vec{0} & 0 & 1
  \end{pmatrix}=\\
  \begin{pmatrix}
       \Omega_{n-2} & \vec{0} &\bullet\\
       \vec{1} & (1-t_{n-1})^{-1} & \bullet\\
       \bullet & \bullet & \bullet
    \end{pmatrix}.
 \end{align*}
It then follows that 
  \[
     \begin{pmatrix}
      \Xi^*+\frac{\epsilon^*\phi^*}{1-\gamma^*} &\theta^*+\frac{\alpha^*\epsilon^*}{1-\gamma^*} \\
      \psi^*+\frac{\delta^*\phi^*}{1-\gamma^*} &\beta^*+\frac{\alpha^*\delta^*}{1-\gamma^*}  
  \end{pmatrix}\Omega_{n-1}\begin{pmatrix}
     \Xi+\frac{\phi\epsilon}{1-\gamma} & \psi+\frac{\delta\phi}{1-\gamma} \\
     \theta+\frac{\alpha\epsilon}{1-\gamma} &\beta+\frac{\alpha\delta}{1-\gamma} 
  \end{pmatrix}=\Omega_{n-1},
  \]
which is precisely the unitary statement after stitching, and the unitary property for the matrix part is proved.

\underline{\textbf{The scalar part:}} Next let us show the unitary property for the scalar part. The initial setup will be exactly the same as in the proof for the matrix part. Again we first verify the crossings. For the positive crossing:    
  \[  
     1\cdot\det\begin{pmatrix}
       1-t_a & 1\\
       t_a & 0
     \end{pmatrix}=-t_a\doteq 1,
  \]
and for the negative crossing 
  \[
     1\cdot\det\begin{pmatrix}
       0 & t_a^{-1}\\
       1 & 1-t_a^{-1}
     \end{pmatrix}=-t_a^{-1}\doteq 1,
  \]  
as required. It is easy to verify that the property is invariant under disjoint union (the determinant of the direct sum of two matrices is the the product of the determinants) and under composition (the determinant of the product of two square matrices is the product of the determinants). So again we only need to check the property under stitching strand $n-1$ to strand $n$. Using the same notation as in the proof for the matrix part, we let 
 \[
    M=\begin{pmatrix}
        \Xi & \psi & \phi \\
        \theta &\beta &\alpha \\
        \epsilon &\delta &\gamma
     \end{pmatrix},
 \]
and $M'$ is the matrix part after stitching strand $n-1$ to strand $n$ and calling the resulting strand $n-1$
 \[
   M'= \begin{pmatrix}
        \Xi+\frac{\phi\epsilon}{1-\gamma} & \psi+\frac{\delta \phi}{1-\gamma} \\
        \theta+\frac{\alpha\epsilon}{1-\gamma} & \beta+\frac{\alpha\delta}{1-\gamma} 
     \end{pmatrix}_{t_n\to t_{n-1}}.
 \] 
Suppose that we have 
 \begin{equation}\label{unitaryomega}
 \overline{\omega}\doteq \omega\det(M).
 \end{equation}
After stitching strand $n-1$ to strand $n$ we want to show that 
 \[\big.(1-\overline{\gamma})\overline{\omega}\big|_{t_{n}\to t_{n-1}}\doteq \big.(1-\gamma)\omega\det(M')\big|_{t_{n}\to t_{n-1}}.\] 
Again to simplify notation we assume $t_{n}\to t_{n-1}$ from now on. Using \eqref{unitaryomega} we can rewrite the above as 
 \[(1-\overline{\gamma})\omega\det(M)\doteq (1-\gamma)\omega\det(M').\]
If $\omega\equiv 0$, then unitarity holds trivially. Otherwise, we can divide both sides by $\omega$ to get
 \begin{equation}\label{unitaryomega2}
 (1-\overline{\gamma})\det(M)\doteq (1-\gamma)\det(M').
 \end{equation}
 Now from the unitary property of $M'$
  \[(M')^*\Omega_{n-1}M'=\Omega_{n-1},\]
 taking the determinant of both sides we obtain 
  \[\overline{\det(M')}\det(M')=1.\] 
 Thus \eqref{unitaryomega2} becomes 
   \[\overline{\det(M')}\det(M)\doteq\frac{1-\gamma}{1-\overline{\gamma}}.\]
  It follows that we just need to prove the above identity. We see that it only involves the matrix part, so starting with the unitary property for the matrix part:
  \[M^*\begin{pmatrix}
     \Omega_{n-2} & \vec{0} &\vec{0}\\
     \vec{1} & (1-t_{n-1})^{-1} & 0\\
     \vec{1} & 1 & (1-t_{n-1})^{-1}
  \end{pmatrix}M=\begin{pmatrix}
     \Omega_{n-2} & \vec{0} &\vec{0}\\
     \vec{1} & (1-t_{n-1})^{-1} & 0\\
     \vec{1} & 1 & (1-t_{n-1})^{-1}
  \end{pmatrix}.  \]  
We can rewrite the above as   
  \begin{equation}\label{unitaryscalar}
     \left[\begin{pmatrix}
      \Xi^* &\theta^* &\epsilon^*\\
      \psi^* &\beta^* & \delta^* \\
      \phi^* &\alpha^* &\gamma^*-1
   \end{pmatrix}+\begin{pmatrix}
      \vec{0} & \vec{0} & \vec{0}\\
      \vec{0} & 0 & 0\\
      \vec{0} & 0 & 1
   \end{pmatrix}\right]\Omega \begin{pmatrix}
     \Xi & \psi &\phi \\
     \theta & \beta & \alpha\\
     \epsilon &\delta & \gamma
   \end{pmatrix}=\Omega,
  \end{equation}
 where 
   \[
    \begin{pmatrix}
      \vec{0} & \vec{0} & \vec{0}\\
      \vec{0} & 0 & 0\\
      \vec{0} & 0 & 1
   \end{pmatrix}\Omega \begin{pmatrix}
     \Xi & \psi &\phi \\
     \theta & \beta & \alpha\\
     \epsilon &\delta & \gamma
   \end{pmatrix}=\begin{pmatrix}
      \vec{0} & \vec{0} & \vec{0}\\
      \vec{0} & 0 & 0\\
      \theta+\left\langle \Xi\right\rangle +\frac{\epsilon}{1-t_{n-1}} & \beta+\left\langle\psi\right\rangle+\frac{\delta}{1-t_{n-1}} & \alpha+\left\langle\phi\right\rangle+\frac{\gamma}{1-t_{n-1}}
   \end{pmatrix}.
 \]
Then \eqref{unitaryscalar} becomes 
  \begin{align*}
    &\begin{pmatrix}
      \Xi^* &\theta^* &\epsilon^*\\
      \psi^* &\beta^* & \delta^* \\
      \phi^* &\alpha^* &\gamma^*-1
   \end{pmatrix}\Omega \begin{pmatrix}
     \Xi & \psi &\phi \\
     \theta & \beta & \alpha\\
     \epsilon &\delta & \gamma
   \end{pmatrix}\\&\hspace{1in}=
   \begin{pmatrix}
       \Omega_{n-2} & \vec{0} &\vec{0}\\
       \vec{1} & (1-t_{n-1})^{-1} & 0\\
       \vec{1}-\theta-\left\langle\Xi\right\rangle+\frac{\epsilon}{-1+t_{n-1}} & 1-\beta-\left\langle\psi\right\rangle+\frac{\delta}{-1+t_{n-1}} & -\alpha-\left\langle\phi\right\rangle+\frac{1-\gamma}{1-t_{n-1}}
    \end{pmatrix}\\
    &\hspace{1in}=\begin{pmatrix}
       \Omega_{n-2} & \vec{0} &\vec{0}\\
       \vec{1} & (1-t_{n-1})^{-1} & 0\\
       \frac{t_{n-1}\epsilon}{-1+t_{n-1}}&\frac{t_{n-1}\delta}{-1+t_{n-1}}& \frac{t_{n-1}(1-\gamma)}{1-t_{n-1}}
    \end{pmatrix},
 \end{align*} 
where we use lemma \ref{columnsum}. Now for the left hand side, we can perform column operations via elementary matrices to get 
  \begin{multline*}
  \begin{pmatrix}
    \Xi^*+\frac{\epsilon^*\phi^*}{1-\gamma^*} &\theta^*+\frac{\alpha^*\epsilon^*}{1-\gamma^*} & \frac{\epsilon^*}{\gamma^*-1}\\
    \psi^*+\frac{\delta^*\phi^*}{1-\gamma^*} &\beta^*+\frac{\alpha^*\delta^*}{1-\gamma^*} & \frac{\delta^*}{\gamma^*-1}\\
    \vec{0} & 0 & 1
  \end{pmatrix}\begin{pmatrix}
    I & \vec{0} &\vec{0}\\
    \vec{0} & 1 & 0\\
    \phi^* & \alpha^* & 1
  \end{pmatrix}\begin{pmatrix}
    I & \vec{0} & \vec{0}\\
    \vec{0} & 1 & 0\\
    \vec{0} & 0 & \gamma^*-1
  \end{pmatrix}\Omega M\\
   =\begin{pmatrix}
       \Omega_{n-2} & \vec{0} &\vec{0}\\
       \vec{1} & (1-t_{n-1})^{-1} & 0\\
       \frac{t_{n-1}\epsilon}{-1+t_{n-1}}&\frac{t_{n-1}\delta}{-1+t_{n-1}}& \frac{t_{n-1}(1-\gamma)}{1-t_{n-1}}
    \end{pmatrix}.
  \end{multline*}
Finally taking the determinant of both sides we obtain 
 \[\overline{\det(M')}(\overline{\gamma}-1)\det(M)=t_{n-1}(1-\gamma)    .\]
 Thus 
  \[\overline{\det(M')}\det(M)\doteq \frac{1-\gamma}{1-\overline{\gamma}},\]
 which completes the proof. 
\end{proof}

\begin{remark}
   As a consequence of the unitary property for the scalar part, for the case of long knots, the matrix part is 1 and we have 
  \[\overline{\omega}\doteq \omega,\]
which is the usual fact that the Alexander polynomial is palindromic. (This is not true for w-knots, see Example \ref{longvsclosed}.)
\end{remark}

\subsection{The Fox-Milnor Condition} Now we are ready to tackle the Fox-Milnor condition

\begin{thm*}
  If a knot $K$ is ribbon, then the Alexander polynomial of $K$, $\Delta_K(t)$ satisfies 
  \[\Delta_K(t)\doteq f(t)f(t^{-1}),\]
 where $\doteq$ means equality up to multiplication by $\pm t^n$, $n\in \Z$ and $f$ is a Laurent polynomial.
\end{thm*}

\begin{proof}
   All we have to do is to make Proposition \ref{ribbon} more concrete in $\Gamma$-calculus. To that end, consider a pure up-down tangle $T$ with strands labeled by $1,2,\dots, 2n$, which satisfies the condition of Proposition \ref{ribbon}. We let $\vec{odd}$ denote the tuple $(1,3,\dots, 2n-1)$ and $\vec{even}$ denote the tuple $(2,4,\dots,2n)$. For convenience, we write the matrix part of $\varphi(T)$ as 
 \[\begin{cases}
    y_{\vec{odd}}=\alpha x_{\odd}+\beta x_{\even}\\
    y_{\even}=\gamma x_{\odd}+\delta x_{\even}
 \end{cases}
 \] 
where each $\alpha$, $\beta$, $\gamma$, $\delta$ is an $n\times n$ matrix. For the $\tau$ closure, we stitch the odd strands to the even strands and label the resulting strands odd. Then it follows from Proposition \ref{stitchingbulk} that 
  \[\left(\begin{array}{c|cc}
      \omega & x_{\odd} & x_{\even}\\ \hline
      y_{\even} & \gamma & \delta\\
      y_{\odd} & \alpha & \beta
  \end{array}\right)\xrightarrow[]{\tau=m^{\odd,\even}_{\odd}} \left(\begin{array}{c|c}
     \omega\det(I-\gamma) & x_{\odd} \\ \hline
     y_{\odd} & \beta+\alpha(I-\gamma)^{-1}\delta
  \end{array}
  \right)_{t_{\even}\to t_{\odd}}.
 \]       
Since the $\tau$ closure yields a trivial tangle we have $\omega\det(I-\gamma)=1$ and $\beta+\alpha(I-\gamma)^{-1}\delta=I$. Now for the $\kappa$ closure, the stitching instructions are given by $y_1=x_2,y_2=x_3,\dots,y_{2n-1}=x_{2n}$, or in short, $y_{[1,2n-1]}=x_{[2,2n]}$. After we perform all the stitchings, the end result is the original knot $K$. From Proposition \ref{omegaalexander} we know that the scalar part is the Alexander polynomial of the knot $K$ (same as its closure $\widetilde{K}$). On the other hand, from the stitching formula (Proposition \ref{stitchingbulk}) the scalar part is given by 
  \[\big.\omega\det(I-N)\big|_{t_i\to t},\]
where $N$ is the submatrix of the matrix part of $\varphi(T)$ specified by 
  \[\left(\begin{array}{c|ccc}
       \bullet & x_2 &\cdots & x_{2n} \\ \hline
      y_1& \multicolumn{3}{c}{\multirow{3}{*}{\raisebox{-3mm}{\scalebox{2}{$N$}}}}  \\
     \vdots & & &\\
      y_{2n-1} & & &
   \end{array}
  \right).
  \]    
(Again $\bullet$ denotes an entry we do not care about.) Now it is a simple exercise in linear algebra that 
 \begin{equation}\label{determinant}
 \det(I-N)=\det(P-M),
 \end{equation}
where $M$ is the matrix part of $\varphi(T)$
   \[\left(\begin{array}{c|ccc}
       \bullet & x_1 &\cdots & x_{2n} \\ \hline
      y_1& \multicolumn{3}{c}{\multirow{3}{*}{\raisebox{-3mm}{\scalebox{2}{$M$}}}}  \\
     \vdots & & &\\
      y_{2n} & & &
   \end{array}
  \right),
  \]    
and $P$ is the matrix given by 
 \[
   P= \left(\begin{array}{c|c|ccc}
      \bullet & x_1 & x_2 &\cdots & x_{2n} \\ \hline
       y_1 & \multirow{3}{*}{\raisebox{-3mm}{\scalebox{2}{$\vec{0}$}}} &\multicolumn{3}{c}{\multirow{3}{*}{\raisebox{-3mm}{\scalebox{2}{$I$}}}} \\
       y_2 & & & & \\
       \vdots & & & &\\ \hline
       y_{2n} & 0 &\multicolumn{3}{c}{\vec{0}}
     \end{array}
    \right).
 \]
To see why \eqref{determinant} is true, observe that if we replace the last row of $P-M$ by the the sum of all the rows, which does not change the value of the determinant, then we obtain the row $\left(-1,0,\dots,0\right)$ by Lemma \ref{columnsum}. We then compute the determinant by expansion along the last row and the result follows. 

Now it is useful to rearrange the rows and columns of $P-M$ into $\odd$ and $\even$, which only changes the determinant up to $\pm 1$, in order to relate to the $\tau$ closure: 
  \[
     \left(\begin{array}{c|c|c}
      \bullet & x_{\odd} & x_{\even} \\ \hline
       y_{\odd} & -\alpha & I-\beta\\ \hline
       y_{\even} & \begin{pmatrix}
          \vec{0} & I_{n-1}\\
          0 & \vec{0}
       \end{pmatrix}-\gamma & -\delta
     \end{array}
     \right).
  \] 
Then by Lemma \ref{blockdeterminant} we have    
  \begin{align*}
     \det\left(\begin{array}{c|c}
       \alpha & \beta-I\\ \hline
        \gamma-\begin{pmatrix}
          \vec{0} & I_{n-1}\\
          0 & \vec{0}
       \end{pmatrix} & \delta
     \end{array}
     \right)=\det\left(\alpha+(I-\beta)\delta^{-1}\left(\gamma-\begin{pmatrix}
          \vec{0} & I_{n-1}\\
          0 & \vec{0}
       \end{pmatrix}\right)\right)\det(\delta).  
  \end{align*}   
From $\beta+\alpha(I-\gamma)^{-1}\delta=I$ we get 
 \[\alpha(I-\gamma)^{-1}\delta=I-\beta.\]
Therefore 
  \begin{align*}
   & \det\left(\alpha+(I-\beta)\delta^{-1}\left(\gamma-\begin{pmatrix}
          \vec{0} & I_{n-1}\\
          0 & \vec{0}
       \end{pmatrix}\right)\right)\det(\delta)\\
       &\hspace{1.5in}=\det\left(\alpha+\alpha(I-\gamma)^{-1}\left(\gamma-\begin{pmatrix}
          \vec{0} & I_{n-1}\\
          0 & \vec{0}
       \end{pmatrix}\right)\right)\det(\delta)\\
       &\hspace{1.5in}=\det(\alpha)\det\left(I+(I-\gamma)^{-1}\left(\gamma-\begin{pmatrix}
          \vec{0} & I_{n-1}\\
          0 & \vec{0}
       \end{pmatrix}\right)\right)\det(\delta)\\
       &\hspace{1.5in}=\det(\alpha)\det[(I-\gamma)^{-1}]\det\left(I-\begin{pmatrix}
          \vec{0} & I_{n-1}\\
          0 & \vec{0}
       \end{pmatrix}\right)\det(\delta)\\
       &\hspace{1.5in}=\frac{\det(\alpha)\det(\delta)}{\det(I-\gamma)}\\
       &\hspace{1.5in}=\omega\det(\alpha)\det(\delta),
  \end{align*}  
 where we use $\omega\det(I-\gamma)=1$ in the last equality. It follows that 
   \begin{equation}\label{foxmilnor}
   \Delta_K(t)\doteq \big.\omega\det(\alpha)\omega\det(\delta)\big|_{t_i\to t}.
   \end{equation}
  To finish off, we will employ the unitary property of $\varphi(T)$. But since we only have the unitary property for string links, we first need to reverse the orientations of all the even strands of $T$. The orientation reversal formula (Proposition \ref{reverseinbulk}) yields
    \[\left(\begin{array}{c|cc}
      \omega & x_{\even} & x_{\odd}\\ \hline
      y_{\even} & \delta & \gamma\\
      y_{\odd} & \beta & \alpha
  \end{array}\right)\xrightarrow[]{dS^{\even}} \left(\begin{array}{c|cc}
     \omega\det(\delta) & x_{\even} & x_{\odd}\\ \hline
     y_{\even} & \delta^{-1} & \delta^{-1}\gamma \\
     y_{\odd} & -\beta\delta^{-1} & \alpha-\beta\delta^{-1}\gamma 
  \end{array}
  \right)_{t_{\even}\to t_{\even}^{-1}}.
 \]  
Now the unitary property of the scalar part tells us that 
 \begin{align*}
    \Big.\omega\det(\delta)\Big|_{t_{\odd}\to t_{\odd}^{-1}} & \doteq \left.\omega\det(\delta)\det\begin{pmatrix}
        \alpha-\beta\delta^{-1}\gamma & -\beta\delta^{-1} \\
        \delta^{-1}\gamma & \delta^{-1}
    \end{pmatrix}\right|_{t_{\even}\to t_{\even}^{-1}}.
 \end{align*}      
Taking $t_{\even}\to t_{\even}^{-1}$ in both sides we obtain 
 \begin{align*}
 \overline{\omega\det(\delta)}&\doteq \omega\det(\delta)\det\begin{pmatrix}
        \alpha-\beta\delta^{-1}\gamma & -\beta\delta^{-1} \\
        \delta^{-1}\gamma & \delta^{-1}
    \end{pmatrix} \\
    &=\omega\det(\delta)\det(\alpha-\beta\delta^{-1}\gamma+\beta\delta^{-1}\delta\delta^{-1}\gamma)\det(\delta^{-1})\\
    &=\omega\det(\alpha).
 \end{align*} 
Again we use Lemma \ref{blockdeterminant} in the second equality. Then setting all $t_i$ to $t$, \eqref{foxmilnor} becomes 
  \[\Delta_K(t)\doteq \omega\det(\delta)\overline{\omega\det(\delta)}\doteq\omega\det(\alpha)\overline{\omega\det(\alpha)},\]
which is precisely the Fox-Milnor condition.

Note that in our proof we can choose the function $f$ to be $\omega\det(\delta)$ or $\omega\det(\alpha)$. In the first case $f$ is the invariant of the tangle obtained by reversing the orientations of the odd strands of $T$, and in the second case $f$ is the invariant of the tangle obtained by reversing the orientations of the even strands of $T$. By Proposition \ref{polynomial} we see that $f$ is a Laurent polynomial.  
\end{proof}

\section{Extension to w-Links}\label{sec:link}
\subsection{The Trace Map} In this section we would like to extend our invariant to links. So far our invariant in $\Gamma$-calculus only works for tangles and long knots, since we do not allow closed components. Notice that our stitching formula involves division by $1-\gamma$, and it only makes sense when $\gamma$ is an off-diagonal term. In other words, we can only stitch strands with distinct labels. When we try to stitch strands of the same label, we may encounter division by zero. Nevertheless, the formula for the scalar part $\omega$ only requires multiplication by $1-\gamma$ and so we expect to be able to extend it to links, or more precisely \emph{long w-links}, i.e. w-links with only one open component. The matrix part is no longer well-defined for links. For instance, if a tangle contains a trivial open component, then to stitch the component to itself we will have to divide by $1-1=0$.

As a first step, we need to describe closed components within the framework of meta-monoids. Let $\W^{X\cup\{c\}}_{cl}$ be the collection of w-tangles with a closed component labeled by $c$. Note that we cannot generate $\W^{X\cup\{c\}}_{cl}$ from crossings using the meta-monoid operations because we cannot stitch the same strand to itself. Let $\W^{X\cup \{c\}}$ be the usual collection of w-tangles with an open component labeled $c$. Then we have a \emph{trace map} 
 \[\tr_c: \W^{X\cup\{c\}}\to \W^{X\cup\{c\}}_{cl}\]
given simply by closing the component $c$ in a trivial way (i.e. only through virtual crossings). We have the following key topological result.

\begin{prop}
  Two w-tangles $T_1$ and $T_2$ have isotopic (same) images in $\W^{X\cup\{c\}}_{cl}$ under the map $\tr_c$ if and only if there is a w-tangle $T\in \W^{X\cup\{a,b\}}$ such that $T_1=m^{a,b}_c(T)$ and $T_2=m^{b,a}_c(T)$. 
\end{prop}  

\begin{proof}
  The if direction is quite clear from the following diagram.
     \begin{center}
    \includegraphics[scale=0.45]{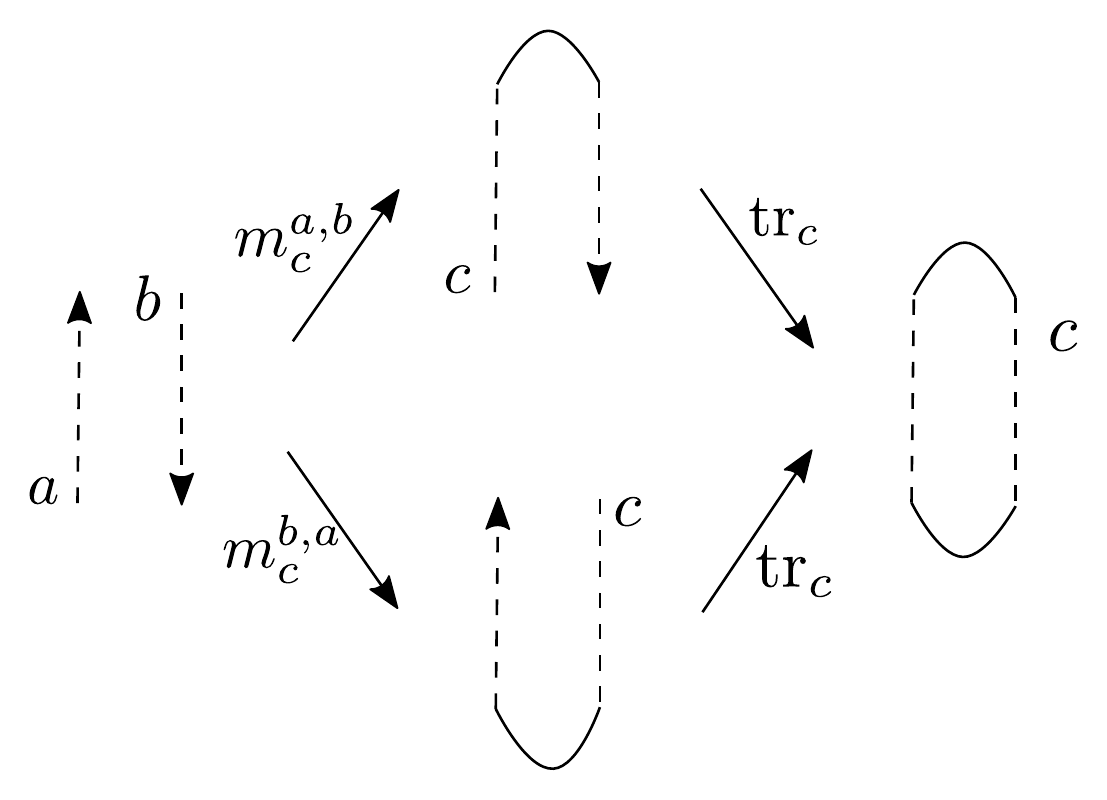}
    \end{center} 
Now for the only if direction, let $T_1$ and $T_2$ have isotopic images under the trace map. We can view the image as a closed component $c$ with two beads on it that represent the two positions where we take the trace and two strands that connect the beads to a fixed base as in the following figure (in general the two strands can be knotted). 
     \begin{center}
   \includegraphics[scale=0.3]{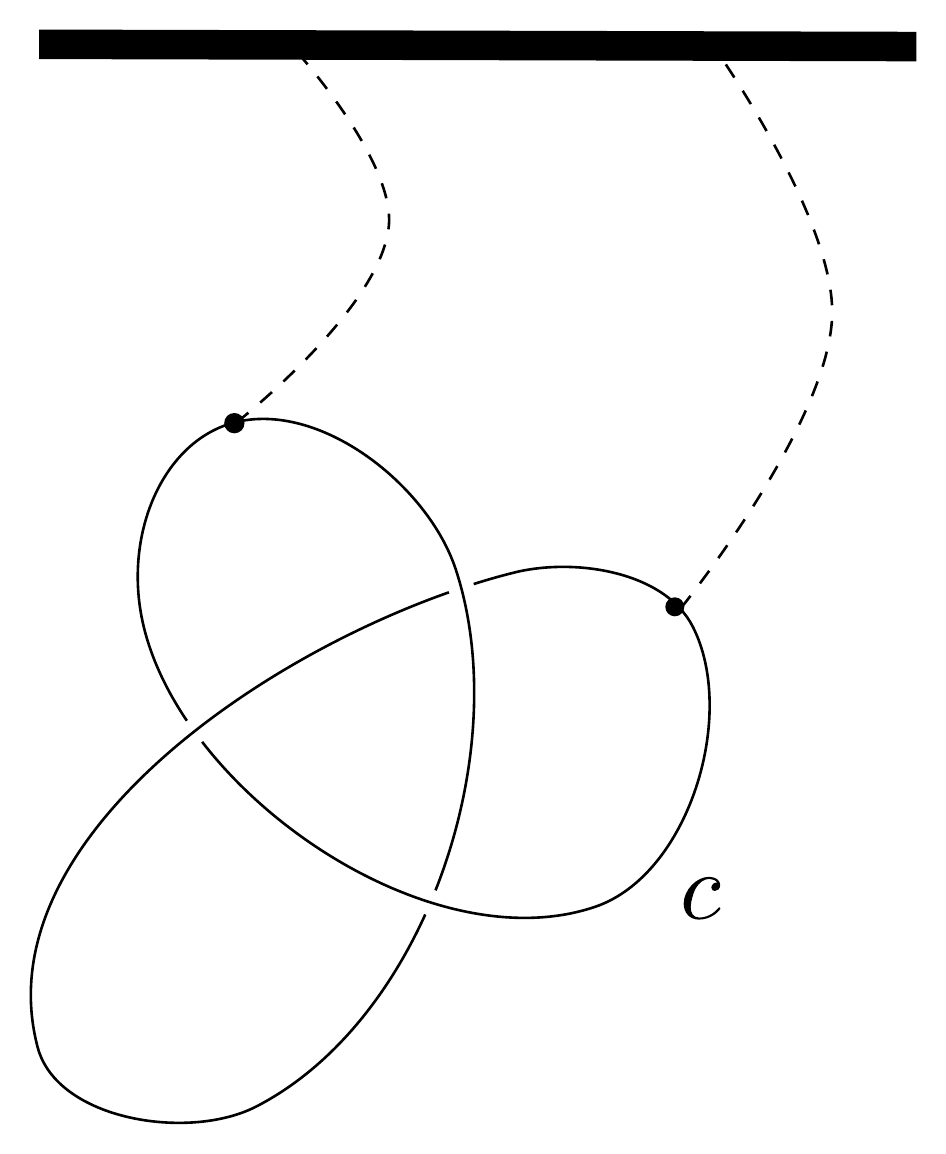}
    \end{center} 
For each position, to take the trace, we unzip the strand, and then cap off the ends.
     \begin{center}
   \includegraphics[scale=0.3]{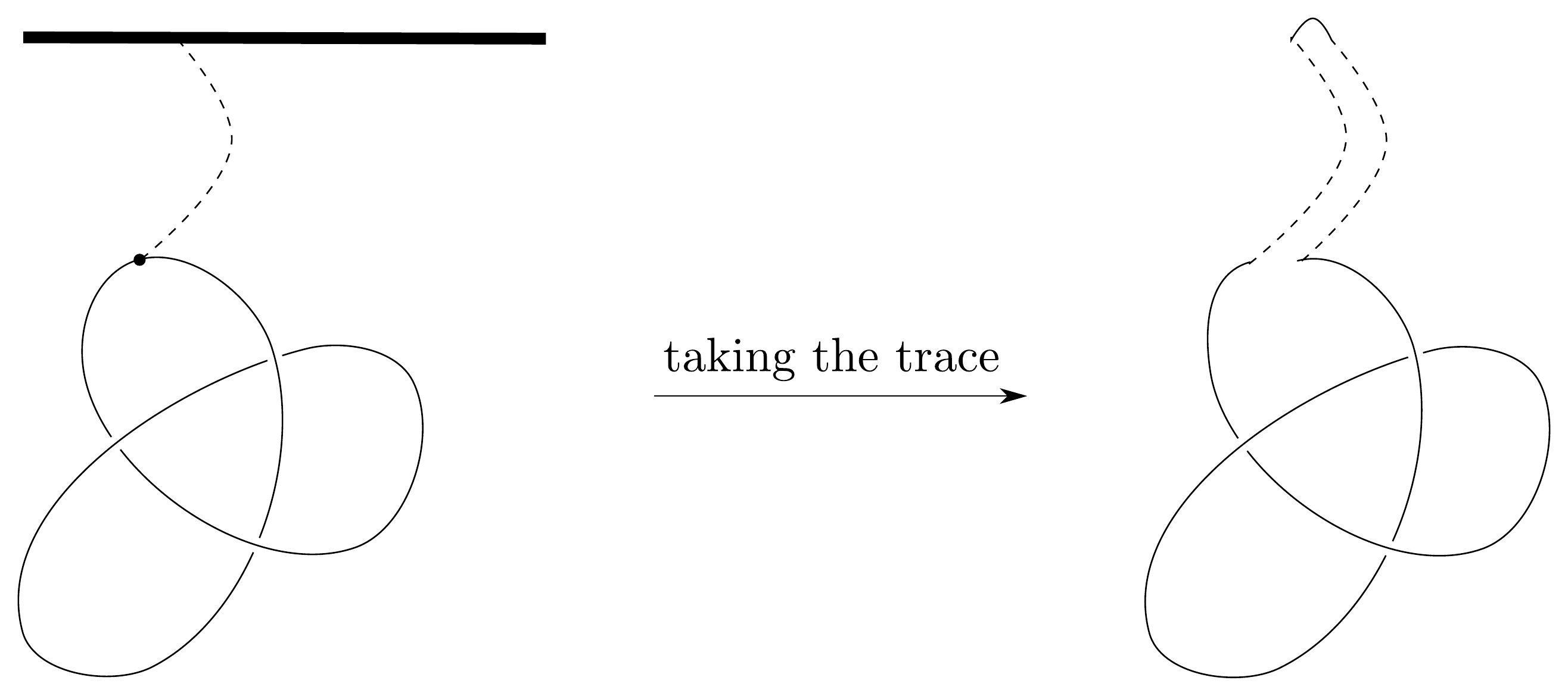}
    \end{center} 
Then to find a tangle $T$ such that $T_1=m_c^{a,b}(T)$ and $T_2=m^{b,a}_c(T)$, we simply unzip the two strands to obtain a tangle $T$ with two components $a$ and $b$. 
    \begin{center}
    \includegraphics[scale=0.3]{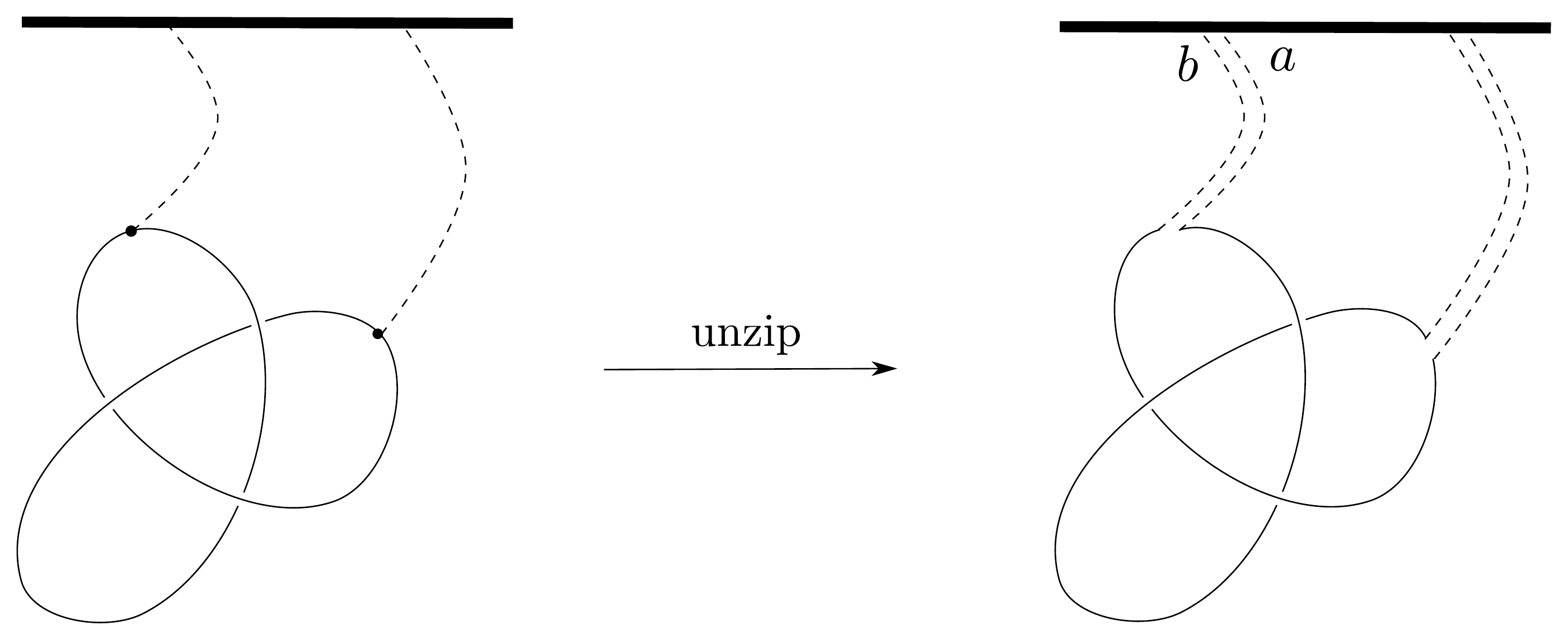}
    \end{center} 
 It is straightforward to check that $T$ satisfies our requirement.                
\end{proof}  

Therefore we can think of an element of $\W^{X\cup\{c\}}_{cl}$ as an equivalence class, namely 
 \[\W^{X\cup\{c\}}_{cl}\simeq \W^{X\cup \{c\}}/(m^{a,b}_c(T)\sim m^{b,a}_c(T)),\]
where $T\in \W^{X\cup\{a,b\}}$. So in particular an invariant $\omega$ on $\W^{X\cup \{c\}}$ will descend to an invariant on $\W^{X\cup\{c\}}_{cl}$ if if satisfies the condition 
 \[\omega(m^{a,b}_c(T))=\omega(m^{b,a}_c(T))\]
for all w-tangles $T\in \W^{X\cup\{a,b\}}$. In general we would want to include links with more than one closed component, and the above discussion can be generalized in a straightforward manner. For a vector $\vec{c}=(c_1,c_2,\dots,c_n)$ we also have a \emph{trace map}
  \[\tr_{\vec{c}}: \W^{X\cup \{\vec{c}\}}\to \W^{X\cup\{\vec{c}\}}_{cl},\]
and an invariant $\omega$ on $\W^{X\cup\{\vec{c}\}}$ will descend to an invariant on $\W^{X\cup\{\vec{c}\}}_{cl}$ if it fulfills the condition     
  \[\omega(m^{\vec{a},\vec{b}}_{\vec{c}}(T))=\omega(m^{\vec{b},\vec{a}}_{\vec{c}}(T))\]
 for two vectors $\vec{a}$, $\vec{b}$ such that $a_i\neq b_j$, and $T$ is a w-tangle in $\W^{X\cup\{\vec{a},\vec{b}\}}$. 
 
\begin{prop}[The Trace Map]\label{tracemap}
 Let $T$ be a w-tangle in $\W^{X\cup \{\vec{c}\}}$, then the following map, which we denote by $\tr$ 
 \[
   \varphi(T)=\left(\begin{array}{c|cc}
      \omega & x_{\vec{c}} & x_S \\ \hline
      y_{\vec{c}} & \alpha & \theta \\
      y_S & \phi & \Xi
    \end{array}
   \right)\xrightarrow[]{\tr} \omega\det(I-\alpha)
 \]
yields an invariant on $\W^{X\cup \{\vec{c}\}}_{cl}$. For an element $L\in \W^{X\cup \{\vec{c}\}}_{cl}$, we denote its value by $\omega_L$.   
\end{prop}

\begin{proof}
  We just have to check that 
    \[\tr(m^{\vec{a},\vec{b}}_{\vec{c}}(\varphi(T_1)))=\tr(m^{\vec{b},\vec{a}}_{\vec{c}}(\varphi(T_1)))\]
    for all tangles $T_1\in \W^{X\cup\{\vec{a},\vec{b}\}}$. Suppose that 
   \[
      \varphi(T_1)=\left(\begin{array}{c|ccc}
         \omega & x_{\vec{a}} & x_{\vec{b}} & x_S\\ \hline
         y_{\vec{a}} & \alpha & \beta & \theta \\
         y_{\vec{b}} & \gamma & \delta &\epsilon\\
         y_S & \phi &\psi & \Xi
       \end{array}
      \right).
   \] 
 Then $m^{\vec{a},\vec{b}}_{\vec{c}}$ gives 
    \[\left(\begin{array}{c|cc}
         \det(I-\gamma)\omega & x_{\vec{c}} & x_S \\
         \hline
         y_{\vec{c}} & \beta + \alpha(I-\gamma)^{-1}\delta & \theta +\alpha(I-\gamma)^{-1}\epsilon \\
         y_S & \psi+\phi(I-\gamma)^{-1}\delta & \Xi+\phi(I-\gamma)^{-1}\epsilon 
      \end{array}
      \right)_{t_{\vec{a}},t_{\vec{b}}\to t_{\vec{c}}}.\]
Taking $\tr$ one obtains 
  \[\big.\det(I-\beta-\alpha(I-\gamma)^{-1}\delta)\det(I-\gamma)\omega\big|_{t_{\vec{a}},t_{\vec{b}}\to t_{\vec{c}}}.\]
Now the other stitching $m^{\vec{b},\vec{a}}_{\vec{c}}$ yields   
    \[\left(\begin{array}{c|cc}
         \det(I-\beta)\omega & x_{\vec{c}} & x_S \\
         \hline
         y_{\vec{c}} & \gamma+\delta(I-\beta)^{-1}\alpha & \epsilon+\delta(I-\beta)^{-1}\theta \\
         y_S & \phi+\psi(I-\beta)^{-1}\alpha & Xi+\psi(I-\beta)^{-1}\theta 
      \end{array}
      \right)_{t_{\vec{a}},t_{\vec{b}}\to t_{\vec{c}}}.\]      
Taking $\tr$ one obtains 
 \[\big.\det(I-\gamma-\delta(I-\beta)^{-1}\alpha)\det(I-\beta)\omega\big|_{t_{\vec{a}},t_{\vec{b}}\to t_{\vec{c}}}.\]
 Finally we invoke Lemma \ref{blockdeterminant} and obverse that 
  \[
    \det\begin{pmatrix}
       I-\beta & \alpha \\
       \delta &  I-\gamma
    \end{pmatrix}=\det\begin{pmatrix}
       I-\gamma & \delta \\
       \alpha & I-\beta
    \end{pmatrix},
  \]      
 which completes the proof.  
\end{proof}

\begin{remark}
  Notice that the trace map agrees with the scalar part of the stitching formula in Proposition \ref{stitchingbulk} when we allow $a_i=b_i$. The matrix part is no longer well-defined because the matrix $I-\gamma$ may not always be invertible. 
\end{remark}

\subsection{The Alexander-Conway Skein Relation} In this section we derive the Alexander-Conway skein relation for long w-links. First of all we have the following analog of Alexander Theorem (see \cite{KT08}).

\begin{prop}
  Every long w-link can be expressed as a partial closure, except the first strand, of a \emph{w-braid} \cite{BND16}, i.e. a braid with virtual crossings modulo the OC relation.
    \begin{center}
    \includegraphics[scale=0.5]{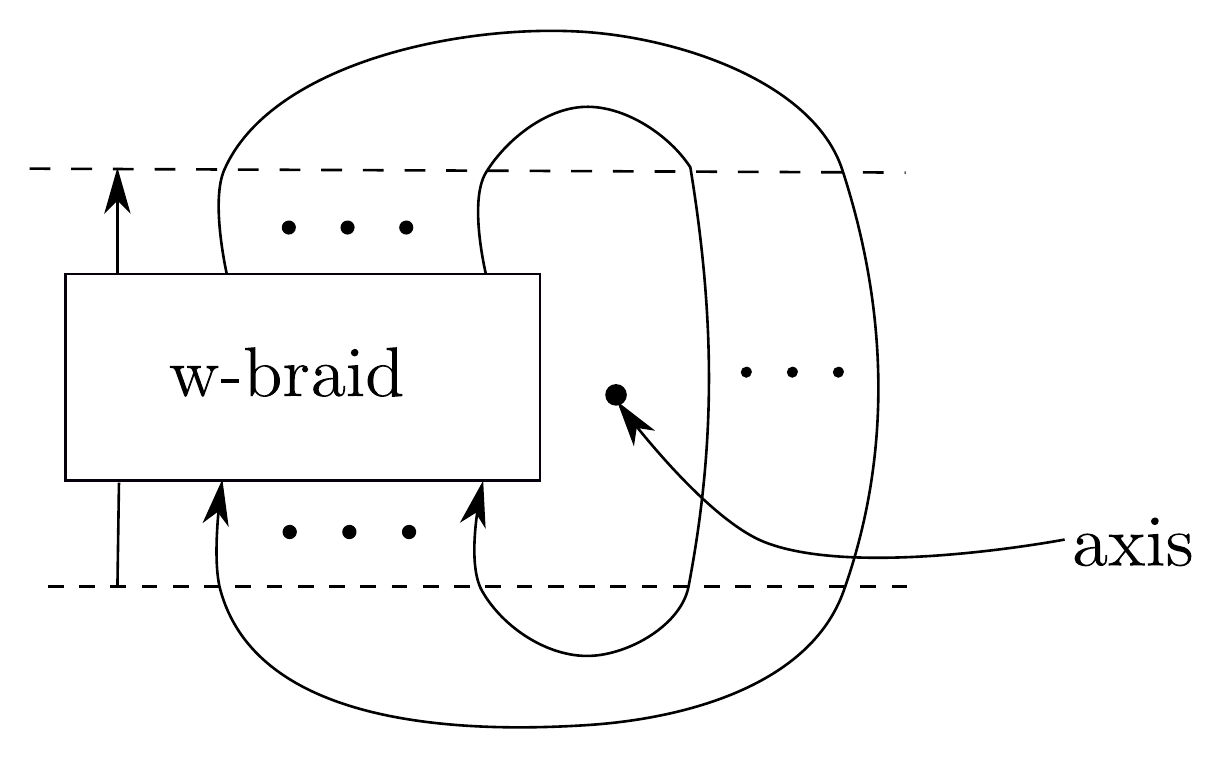}
    \end{center}
\end{prop}    

\begin{proof}
  When we allow virtual crossings, the proof simplifies greatly. Namely we just need to decompose the long w-link into a disjoint union of crossings, put all the crossings in a row and then stitch them. As an example, let us look at the long trefoil.
  \begin{center}
    \includegraphics[scale=0.4]{longtrefoil.pdf}
   \end{center}
 Putting all the crossings of the long trefoil horizontally and the stitching the strands appropriately we obtain the desired form.
   \begin{center}
    \includegraphics[scale=0.3]{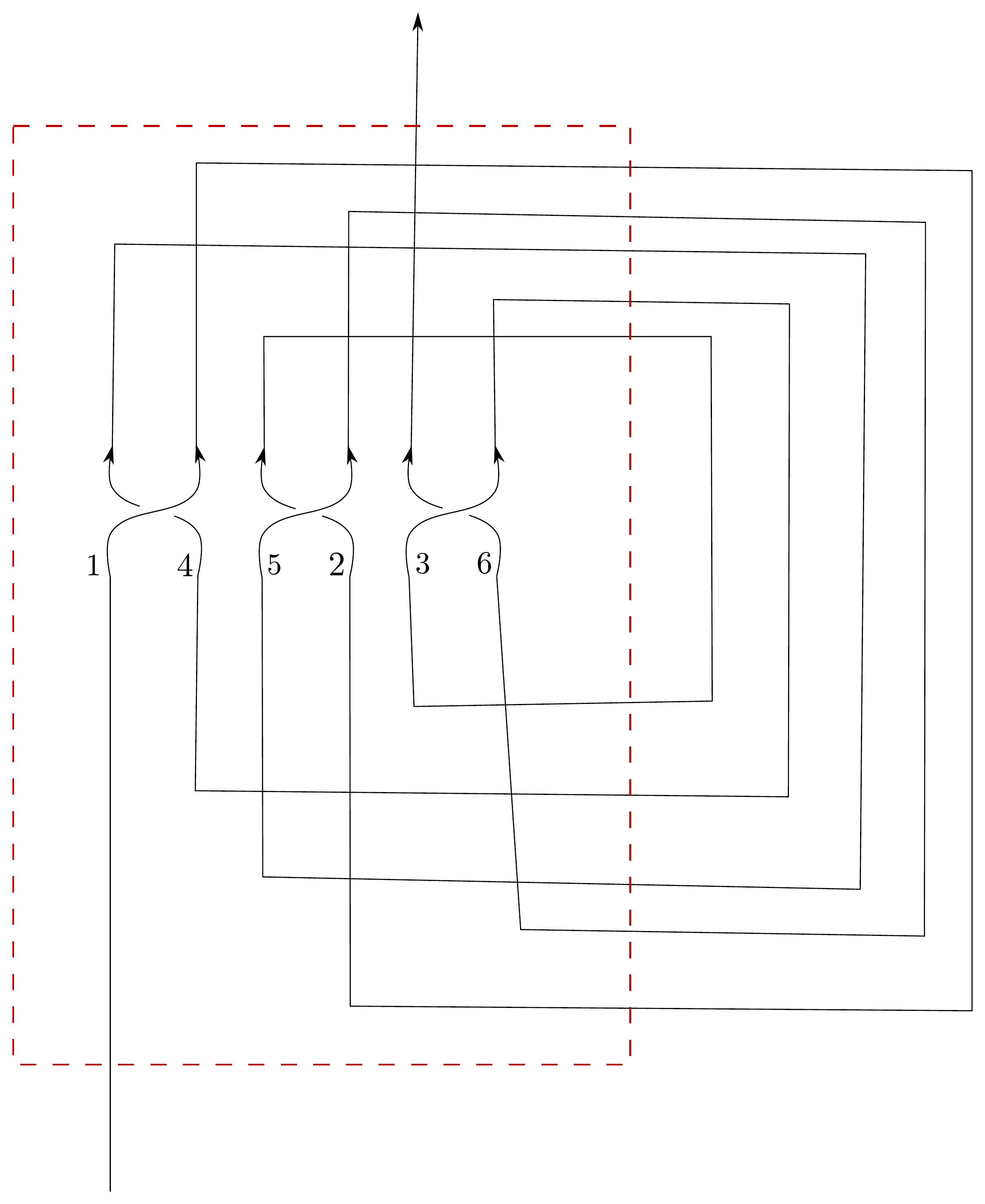}
    \end{center} 
The w-braid is enclosed in the dashed rectangle.   
\end{proof}

For a long w-link $L$, let 
  \[\Delta_L(t)=t^{-w(L)/2}\omega_L(t),\]
where $\omega_L(t)$ is the invariant as defined in Proposition \ref{tracemap} (we identify all the variables $t_i$ to $t$) and 
  \[w(L)=\sum_{\text{crossings}}\pm 1,\]
with $+1$ for a positive crossing and $-1$ for a negative crossing. We record here a simple property of $\Delta_L$.

\begin{prop}
   Let $L$ be a long w-link and suppose $L$ contains a closed trivial component, i.e. bounds an embedded 2-disk that is disjoint from the rest of $L$. Then 
   \[
     \Delta_L(t)=0.
   \]
\end{prop}

\begin{proof}
  Suppose the closed component is labeled $c$. The link $L$ can be obtained by closing a tangle of the form $T\sqcup U_c$, where $U_c$ denotes the trivial strand. 
    \begin{center}
       \includegraphics[scale=0.4]{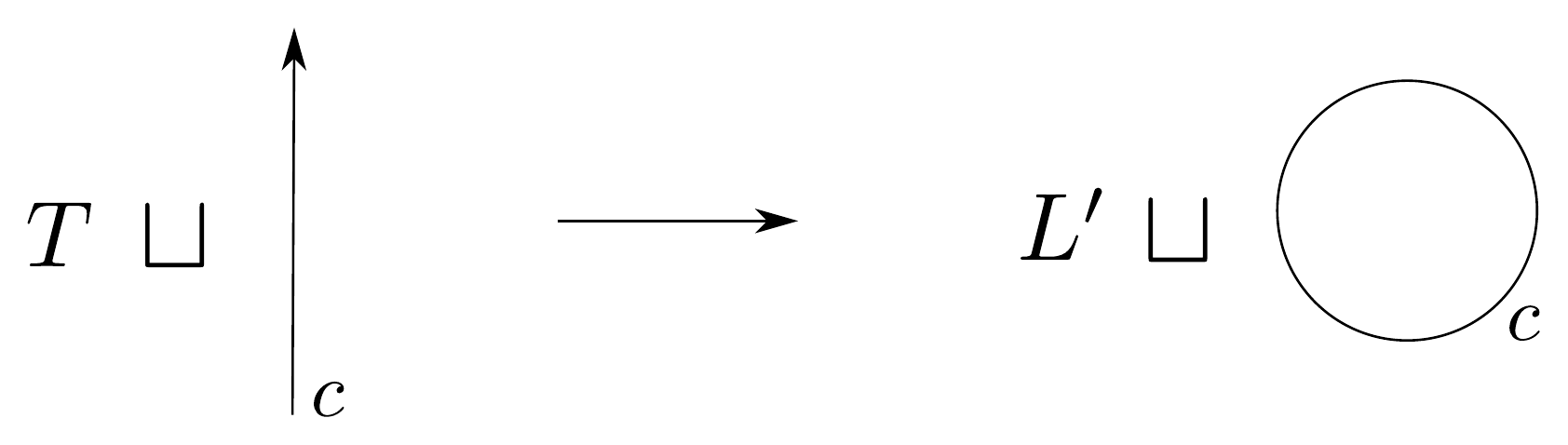}
    \end{center} 
 Then the matrix part of $\varphi(T)$ will contain a column of the form $(1,\vec{0})^t$. From Proposition \ref{tracemap} we observe that $I-\alpha$ will contain a column of zeros. Therefore the determinant vanishes.    
\end{proof}

\begin{thm*}[Alexander-Conway Skein Relation]
  Let $L_+$, $L_{-}$ and $L_0$ be three long w-links which are identical except at a neighborhood of a crossing where they are given by, 
   \begin{center}
    \includegraphics[width=0.5\textwidth]{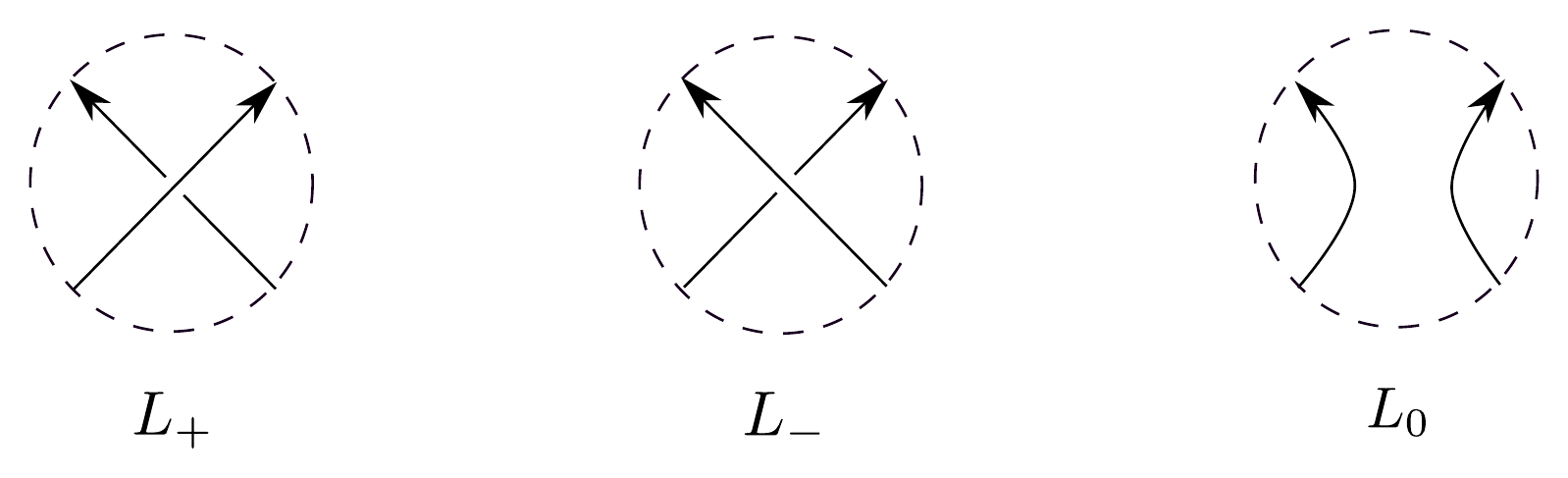}
   \end{center}     
then we have 
  \begin{equation}\label{alexanderconway}
  \Delta_{L_+}(t)-\Delta_{L_{-}}(t)=(t^{-1/2}-t^{1/2})\Delta_{L_0}(t).
  \end{equation}
\end{thm*} 

\begin{proof}
  First of all we prove the following special case $L_+=\widehat{\beta\sigma_{n-1}}$, $L_{-}=\widehat{\beta\sigma_{n-1}^{-1}}$, $L_0=\widehat{\beta}$. Here $\beta$ is a w-braid and $\sigma_{n-1}$ is a standard generator of the braid group, $n$ is the number of strands, and $\ \widehat{}\ $ denotes the partial closure (except the first strand). Observe that 
   \[w(L_+)=w(L_0)+1,\quad w(L_-)=w(L_0)-1.\]
   Thus the skein relation becomes
   \begin{equation}\label{conwayskein}   
    t^{-1/2}\omega_{L_+}(t)-t^{1/2}\omega_{L_{-}}(t)=(t^{-1/2}-t^{1/2})\omega_{L_0}(t).
    \end{equation}
  From Proposition \ref{tracemap}  we have 
   \[\omega_{L_+}(t)=\mathrm{det}([I-\beta\sigma_{n-1}]_1^1),\quad \omega_{L_{-}}(t)=\mathrm\det([I-\beta\sigma_{n-1}^{-1}]_1^1),\quad \omega_{L_0}=\mathrm\det([I-\beta]_1^1),\]
  where we identify the braid with its Burau representation by abuse of notations. Let 
   \[
     \beta=\left(\begin{array}{c|cc}
           M_1 & \phi_1 & \psi_1 \\ \hline
           \theta_1 & a & b \\
           \epsilon_1 & c & d 
        \end{array}     
     \right),
   \] 
where $M_1$ is an $(n-2)\times (n-2)$ matrix, $\phi_1$, $\psi_1$ are column vectors and $\theta_1$, $\epsilon_1$ are row vectors. Then 
  \[\beta\sigma_{n-1}=\left(\begin{array}{ccc}
     M_1 & \phi_1 & \psi_1\\ 
     \theta_1 & a & b\\
     \epsilon_1 & c & d
  \end{array}
  \right)\left(\begin{array}{ccc}
      I & \vec{0} & \vec{0} \\ 
      \vec{0} & 1-t & 1\\
      \vec{0} & t & 0
    \end{array}
  \right)=\left(\begin{array}{ccc}
      M_1 & (1-t)\phi_1+t\psi_1 & \phi_1 \\
      \theta_1 & (1-t)a+tb & a\\
      \epsilon_1 & (1-t)c+td & c
    \end{array}
  \right),
  \]
and 
  \[\beta\sigma_{n-1}^{-1}=\left(\begin{array}{ccc}
     M_1 & \phi_1 & \psi_1\\ 
     \theta_1 & a & b\\
     \epsilon_1 & c & d
  \end{array}
  \right)\left(\begin{array}{ccc}
      I & \vec{0} & \vec{0} \\ 
      \vec{0} & 0 & t^{-1}\\
      \vec{0} & 1 & 1-t^{-1}
    \end{array}
  \right)=\left(\begin{array}{ccc}
      M_1 & \psi_1 & t^{-1}\phi_1+(1-t^{-1})\psi_1 \\ 
      \theta_1 & b & t^{-1}a+(1-t^{-1})b \\
      \epsilon_1 & d & t^{-1}c+(1-t^{-1})d
    \end{array}
  \right).
  \] 
Removing the first column and the first row (correspondingly, we remove the subscript 1 in the notations) we can rewrite \eqref{conwayskein} as 
  \begin{align*}
    t^{-1/2}\det\begin{pmatrix}
        I-M & -(1-t)\phi-t\psi& -\phi \\
      -\theta & 1-(1-t)a-tb & -a\\
      -\epsilon & -(1-t)c-td & 1-c
    \end{pmatrix}&-t^{1/2}\det\begin{pmatrix}
        I-M & -\psi & -t^{-1}\phi-(1-t^{-1})\psi \\ 
      -\theta & 1-b & -t^{-1}a-(1-t^{-1})b \\
      -\epsilon & -d & 1-t^{-1}c-(1-t^{-1})d
    \end{pmatrix}\\&=(t^{-1/2}-t^{1/2})\det\begin{pmatrix}
        I-M & -\phi & -\psi\\ 
     -\theta & 1-a & -b\\
     -\epsilon & -c & 1-d
    \end{pmatrix}.
  \end{align*} 
Now for the first matrix, multiply the third column with $(1-t)$ and subtract it from the second column we obtain 
 \begin{align*}
   &t^{-1/2}\det\begin{pmatrix}
        I-M & -t\psi& -\phi \\
      -\theta & 1-tb & -a\\
      -\epsilon & -1+t(1-d) & 1-c
    \end{pmatrix} \\
   &=t^{-1/2}\det\begin{pmatrix}
        I-M & -t\psi& -\phi \\
      -\theta & -tb & -a\\
      -\epsilon & t(1-d) & 1-c
    \end{pmatrix} +t^{-1/2}\det\begin{pmatrix}
        I-M & \vec{0}& -\phi \\
      -\theta & 1 & -a\\
      -\epsilon & -1 & 1-c
    \end{pmatrix}\\
    &=t^{1/2}\det\begin{pmatrix}
        I-M & -\psi& -\phi \\
      -\theta & -b & -a\\
      -\epsilon & 1-d & 1-c
    \end{pmatrix}+t^{-1/2}\det\begin{pmatrix}
        I-M & \vec{0}& -\phi \\
      -\theta & 1 & -a\\
      -\epsilon & -1 & 1-c
    \end{pmatrix}\\
    &=-t^{1/2}\det\begin{pmatrix}
        I-M & -\phi& -\psi \\
      -\theta & 1-a & -b\\
      -\epsilon & -c & 1-d
    \end{pmatrix}-t^{1/2}\det\begin{pmatrix}
        I-M & \vec{0}& -\psi \\
      -\theta & -1& -b\\
      -\epsilon & 1 & 1-d
    \end{pmatrix}\\
    &\hspace{1.5in}+t^{-1/2}\det\begin{pmatrix}
        I-M & \vec{0}& -\phi \\
      -\theta & 1 & -a\\
      -\epsilon & -1 & 1-c
    \end{pmatrix}.
 \end{align*} 
Similarly for the second matrix, multiply the second column with $(1-t^{-1})$ and subtract it from the third column we have 
 \begin{align*}
   &t^{1/2}\det\begin{pmatrix}
        I-M & -\psi& -t^{-1}\phi \\
      -\theta & 1-b & -1+t^{-1}(1-a)\\
      -\epsilon & -d & 1-t^{-1}c
    \end{pmatrix} \\
   &=t^{1/2}\det\begin{pmatrix}
        I-M & -\psi& -t^{-1}\phi \\
      -\theta & 1-b & t^{-1}(1-a)\\
      -\epsilon & -d & -t^{-1}c
    \end{pmatrix} +t^{1/2}\det\begin{pmatrix}
        I-M & -\psi& \vec{0} \\
      -\theta & 1-b & -1\\
      -\epsilon & -d & 1
    \end{pmatrix}\\
    &=t^{-1/2}\det\begin{pmatrix}
        I-M & -\psi& -\phi \\
      -\theta & 1-b & 1-a\\
      -\epsilon & -d & -c
    \end{pmatrix}+t^{-1/2}\det\begin{pmatrix}
        I-M & -\psi & -\vec{0} \\
      -\theta & 1-b & -1\\
      -\epsilon & -d & 1
    \end{pmatrix}\\
    &=-t^{-1/2}\det\begin{pmatrix}
        I-M & -\phi& -\psi \\
      -\theta & 1-a & -b\\
      -\epsilon & -c & 1-d
    \end{pmatrix}-t^{-1/2}\det\begin{pmatrix}
        I-M & -\phi& \vec{0} \\
      -\theta & 1-a & 1\\
      -\epsilon & -c & -1
    \end{pmatrix}\\
    &\hspace{1.5in}+t^{1/2}\det\begin{pmatrix}
        I-M & -\psi & \vec{0} \\
      -\theta & 1-b & -1\\
      -\epsilon & -d & 1
    \end{pmatrix}. 
 \end{align*}
Subtracting the above identities give us the skein relation. 

Finally we show that a general case can be reduced to the special case as follows. Given a long w-link $L_{+}$, we first express it as the partial closure of a braid $\beta_1\sigma_i\beta_2$. As a first step, observe that $\sigma_i$ can be written as a conjugate of $\sigma_{n-1}$ (simply pulling the crossing $\sigma_i$ to the right-most position. Thus we can assume that $L_{+}$ is the partial closure of $\alpha_1\sigma_{n-1}\alpha_2$. 
   \begin{center}
   \includegraphics[scale=0.4]{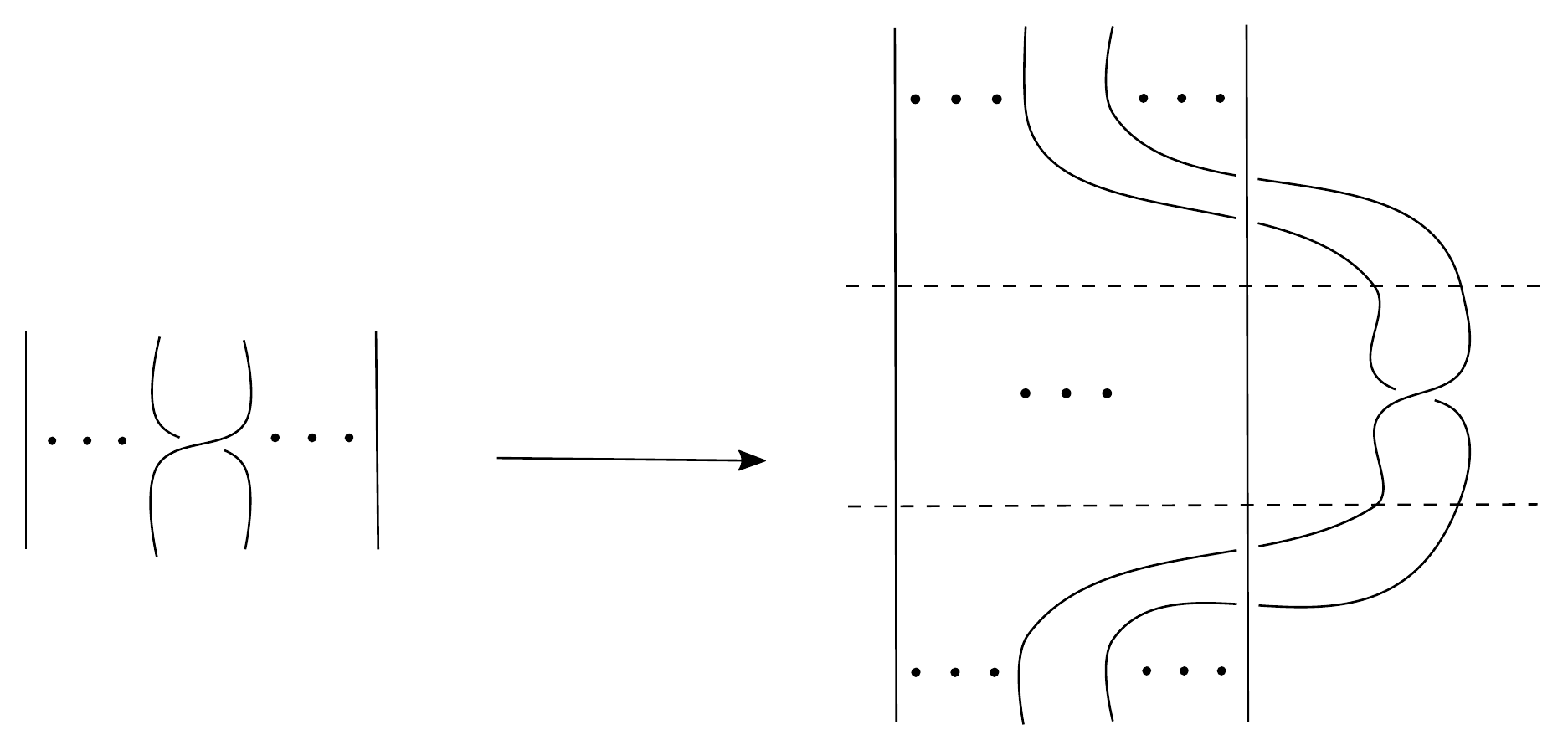}
    \end{center}   
Now to proceed we can push $\alpha_2$ along the closure to the bottom of $\alpha_1$ and then move the open component to the left. 
   \begin{center}
    \includegraphics[scale=0.3]{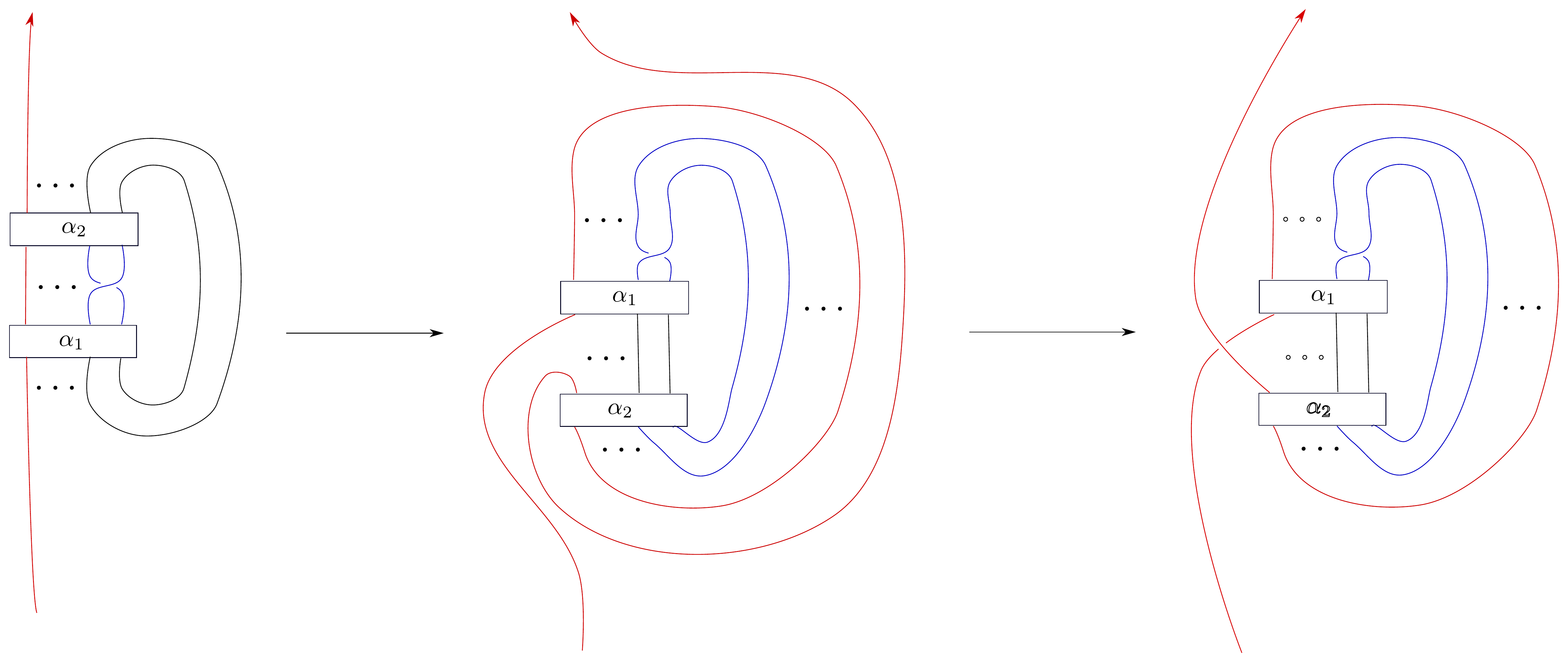}
    \end{center}   
The end result now is the partial closure of a w-braid of the form $\beta\sigma_{n-1}$, as required. 
\end{proof}  

\section{Appendix: Computer Programs}\label{sec:appendix}     
In this section we include the Mathematica code for $\Gamma$-calculus. As mentioned previously, $\Gamma$-calculus has a rather simple implementation on a computer. A reader with Mathematica can just copy the code from \href{http://www.math.toronto.edu/vohuan/}{here} and run it directly. The version of $\Gamma$-calculus that we present is a slightly modified form of the original program, which can be found \href{http://drorbn.net/AcademicPensieve/2015-07/PolyPoly/nb/Demo.pdf}{here}. First we write a subroutine that will display $\Gamma$-calculus in the correct format, this is mostly for aesthetic purpose.
  \begin{center}  
  \includegraphics[scale=0.7]{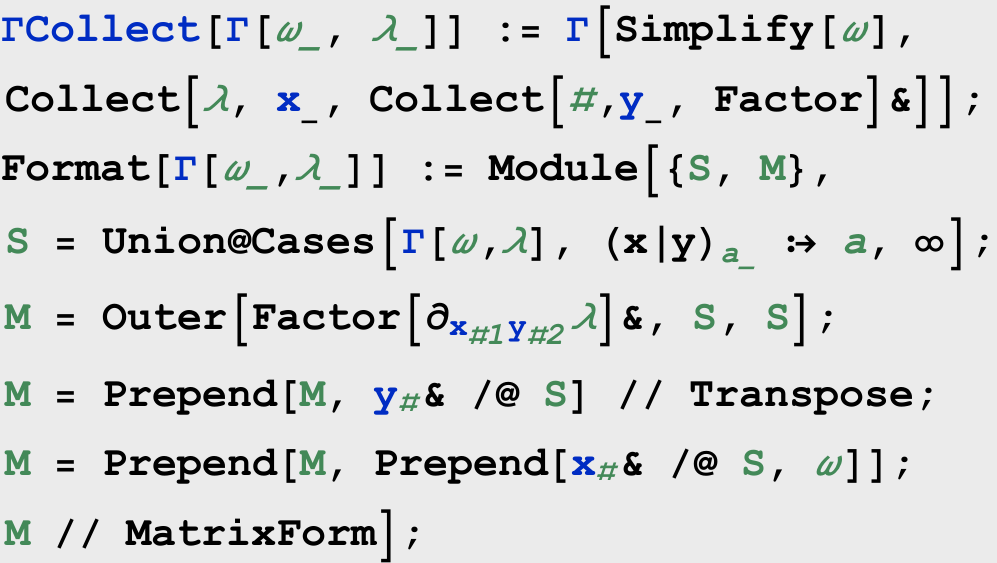}
  \end{center}
The subroutine takes as input a rational function $\omega$ and a matrix $\lambda$. Here $\lambda$ is given in the form 
  \[\lambda=\{y_{\vec{a}}\}\cdot\mathrm{matrix}\cdot \{x_{\vec{a}}\}\]
  where $\vec{a}$ is the labels of the strands. So for instance, the following input 
   \begin{center}  
  \includegraphics[scale=0.7]{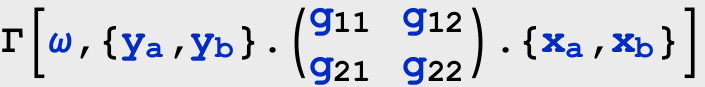}
  \end{center}
produces 
  \begin{center}  
  \includegraphics[scale=0.7]{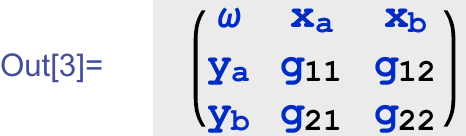}
  \end{center} 
Now we include the main bulk of the program, which is the subroutine that executes stitching
    \begin{center}  
  \includegraphics[scale=0.7]{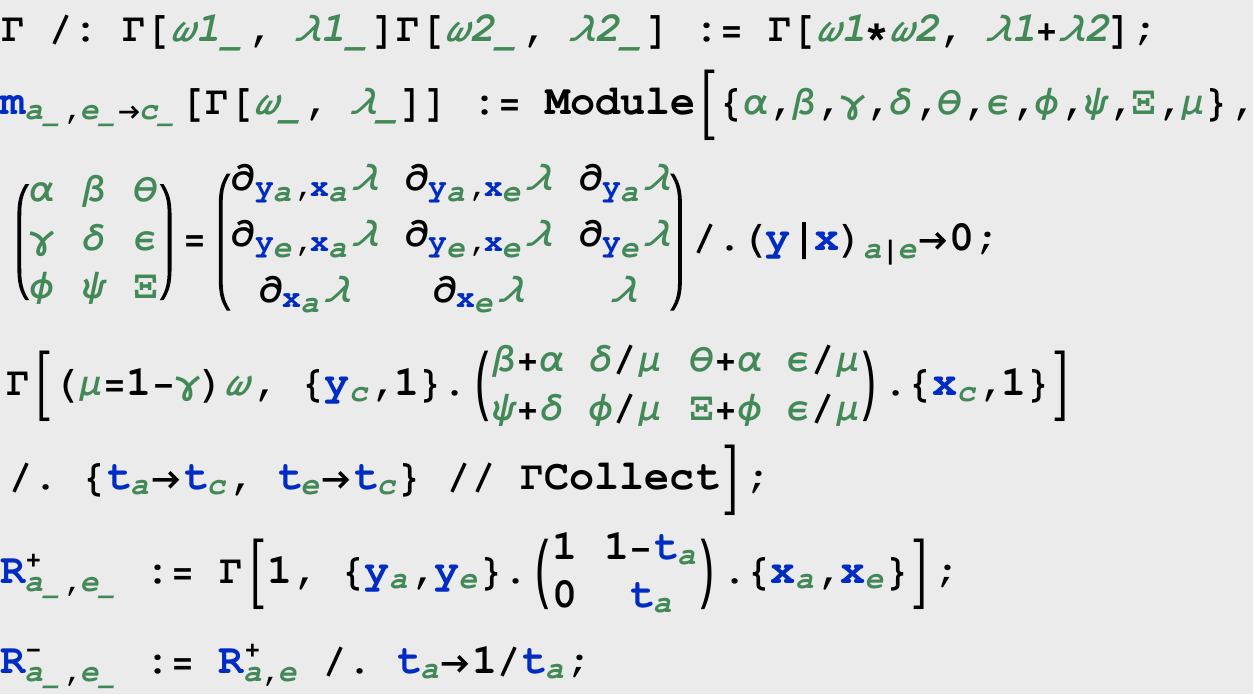}
  \end{center}
First let us check the meta-associativity condition. Meta-associativity involves three strands in a tangle, so we input a matrix with a $3\times 3$ minor singled out together with the meta-associativity equation
   \begin{center}  
  \includegraphics[scale=0.7]{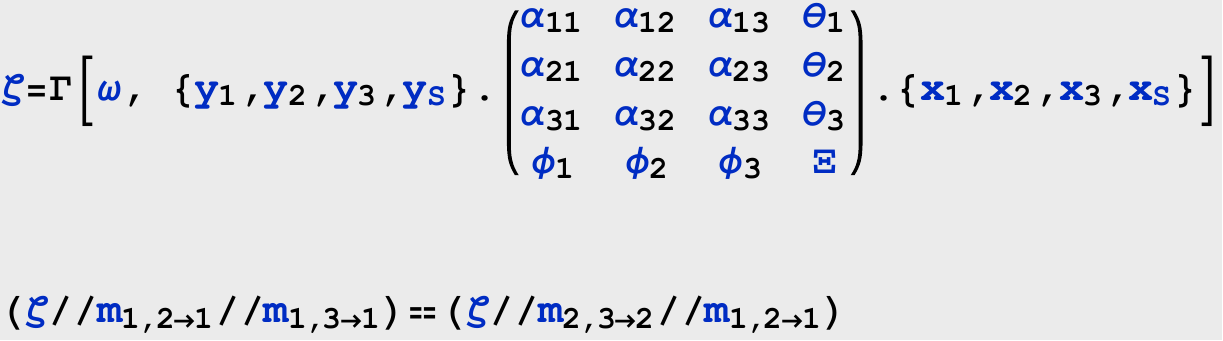}
  \end{center} 
The output is     
   \begin{center}  
  \includegraphics[scale=0.7]{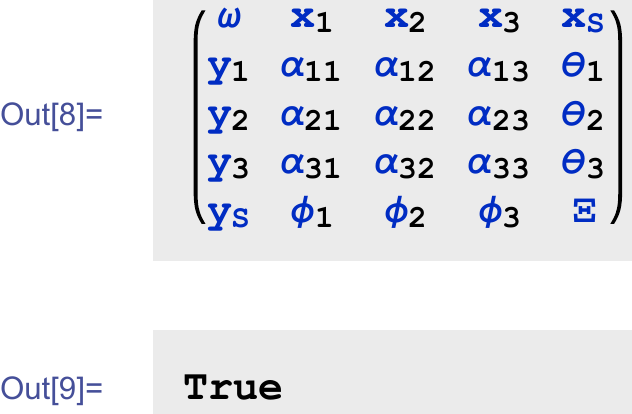}
  \end{center} 
as expected. Next we check the Reidemeister III relation. Its left hand side is 
     \begin{center}  
  \includegraphics[scale=0.7]{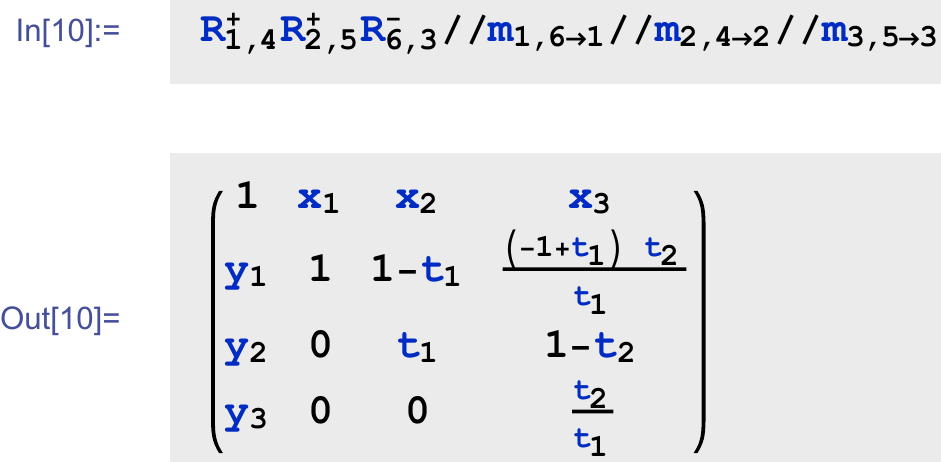}
  \end{center}
and its right hand side is 
    \begin{center}  
  \includegraphics[scale=0.7]{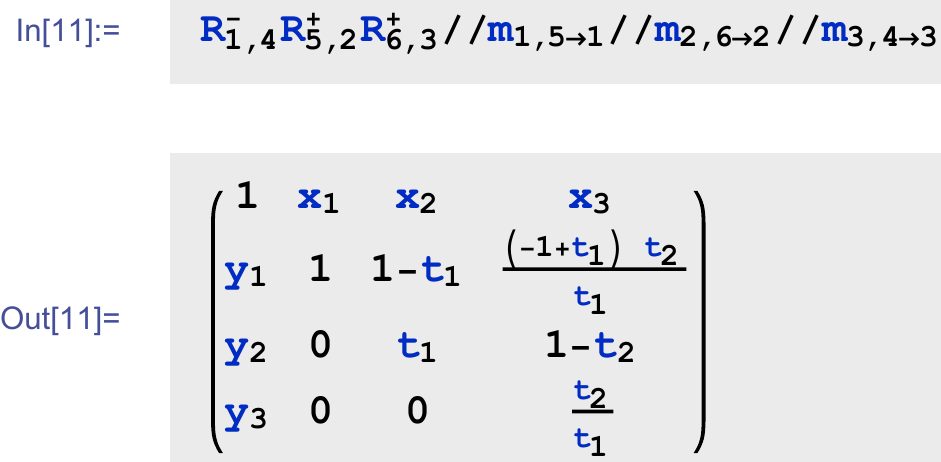}
  \end{center}
as expected. The Reidemeister II relation has a simple verification
    \begin{center}  
  \includegraphics[scale=0.7]{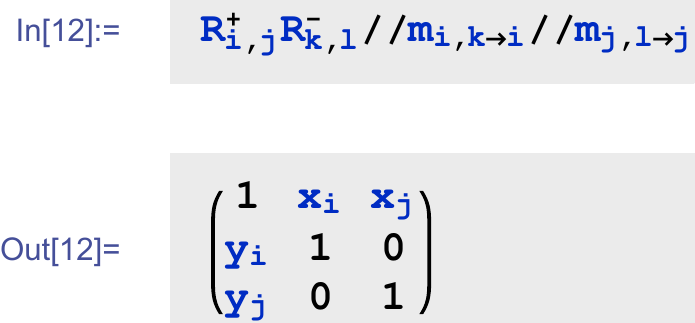}
  \end{center}  
The OC relation yields 
    \begin{center}  
  \includegraphics[scale=0.7]{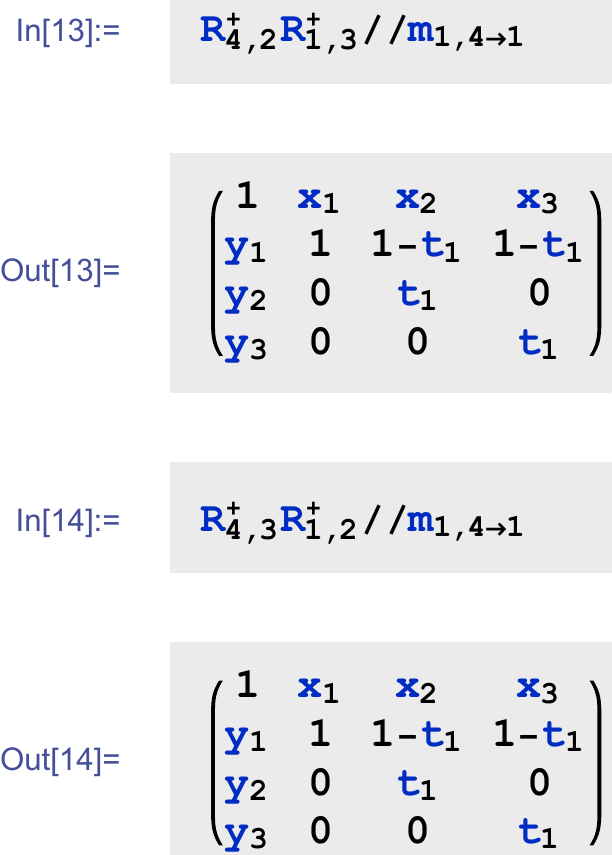}
  \end{center}         
To conclude, let us compute the invariant for the figure-eight knot
    \begin{center}  
  \includegraphics[scale=0.7]{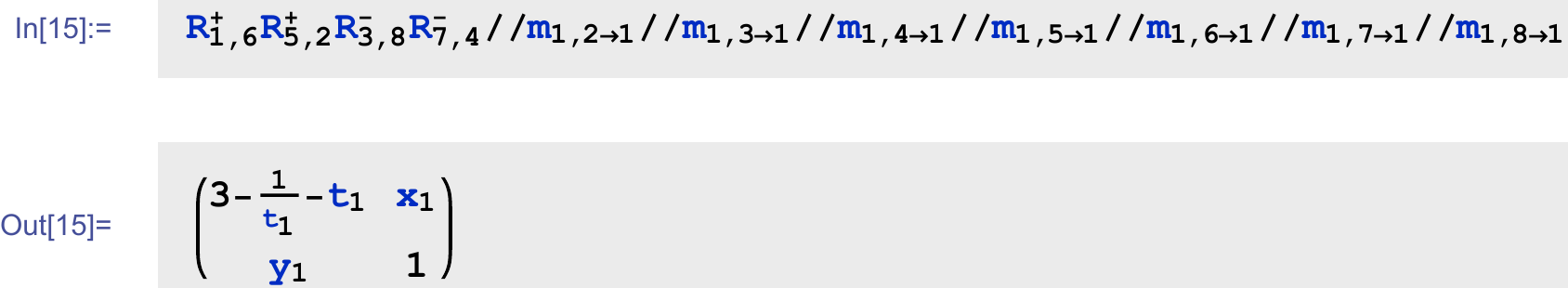}
  \end{center}
Notice the scalar part gives the Alexander polynomial. 

\bibliographystyle{alpha}
\bibliography{bibliography}
  
\end{document}